\documentclass[12pt]{amsart}

\usepackage{amssymb}
\usepackage{graphicx}
\usepackage{multicol}
\usepackage{mathrsfs}

\newtheorem{thm}{Theorem}[section]

\newtheorem{lemma}[thm]{Lemma}


\title{A new perspective on taxicab conic sections}
\author{Emily Frost,
Dylan Helliwell,
Suki Shergill
}
\address{Department of Mathematics, 
Seattle University.}
\email{helliwed(at)seattleu.edu}
\date{\today}
\thanks{Emily Frost and Suki Shergill were supported by Seattle University's College of Science and Engineering Summer Undergraduate Research Program in 2011 and 2018 respectively.}

\begin{document}

\begin{abstract}
We explore taxicab conic sections from the perspective of slicing taxicab cones by planes, as opposed to the more well-studied approach from the perspective of distance formulations.  After establishing a significant amount of structural framework, a complete characterization of the resulting taxicab conic sections is established, and a number of special cases are explored.
\end{abstract}

\maketitle

\section{Introduction}

In Euclidean space, conic sections are classically known and well studied.  The name ``conic section'' comes from the fact that they are realized exactly as the intersections of circular cones and planes.  From this geometric origin, formulations in terms of distance arise, and from these, the familiar algebraic representations of the curves are derived.

In $\mathbb{R}^2$, the taxicab distance between two points $x$ and $y$ is given by
\[
d(x,y) = |x_1 - y_1| + |x_2 - y_2|
\]
and, using this as an alternative to Euclidean distance, conic sections have been studied from the perspective of the distance formulations mentioned above:  starting with the distance formula for a given conic section, switch to taxicab distance and explore what objects arise.  This is a fruitful exercise and a number of interesting shapes and special cases emerge from the analysis.  See \cite{Krause, Reynolds, Laatsch, KAGO} for examples of this analysis, incorporating both two-foci definitions and focus-directrix definitions of taxicab conics.  Figure~\ref{distanceconicsfig} shows the variety of shapes that can arise from these definitions.

\begin{figure}
\begin{picture}(250,500)
\put(10,425){
\includegraphics[scale = .4, clip = true, draft = false]{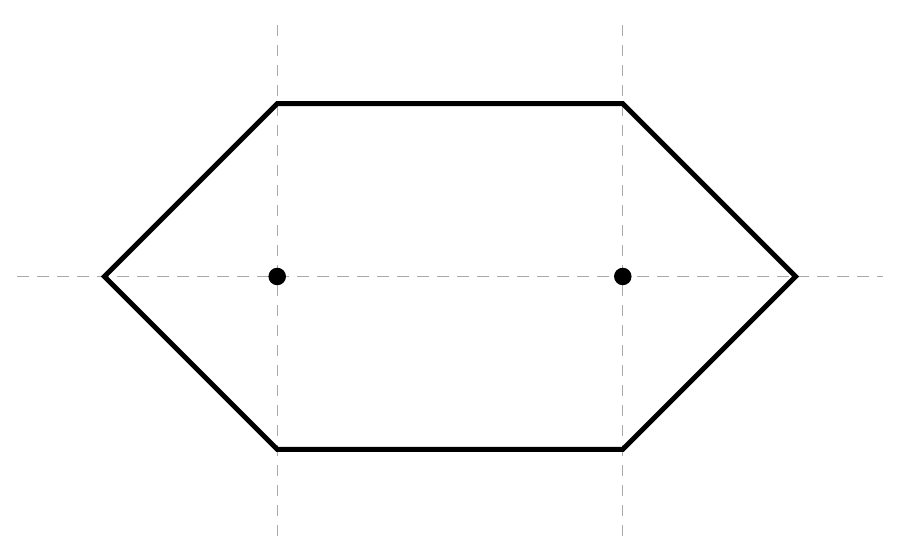}
}
\put(145,415){
\includegraphics[scale = .4, clip = true, draft = false]{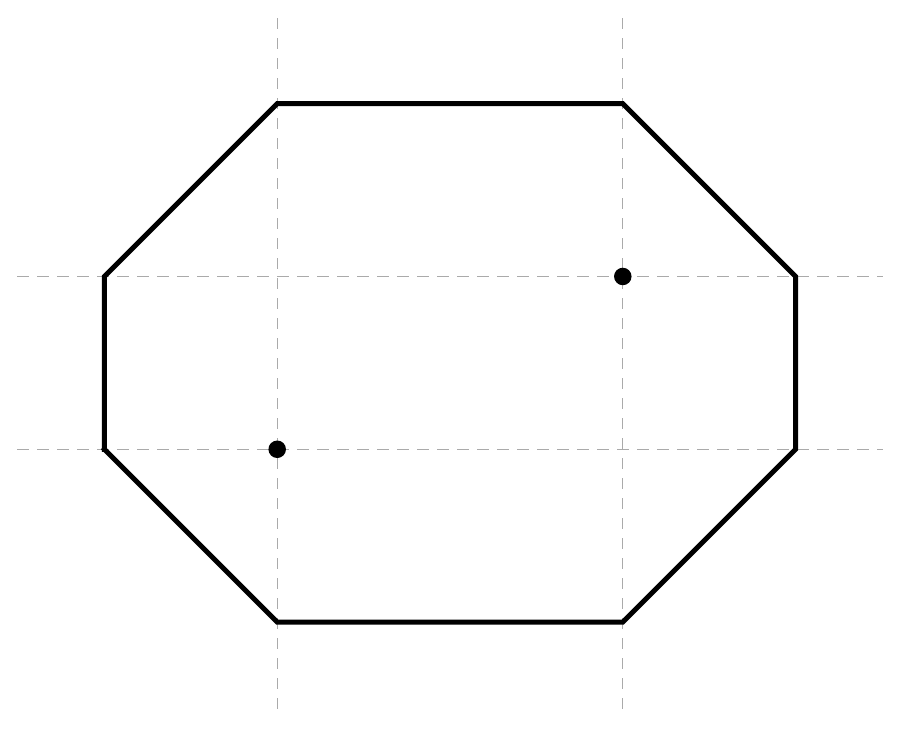}
}
\put(-50,320){
\includegraphics[scale = .4, clip = true, draft = false]{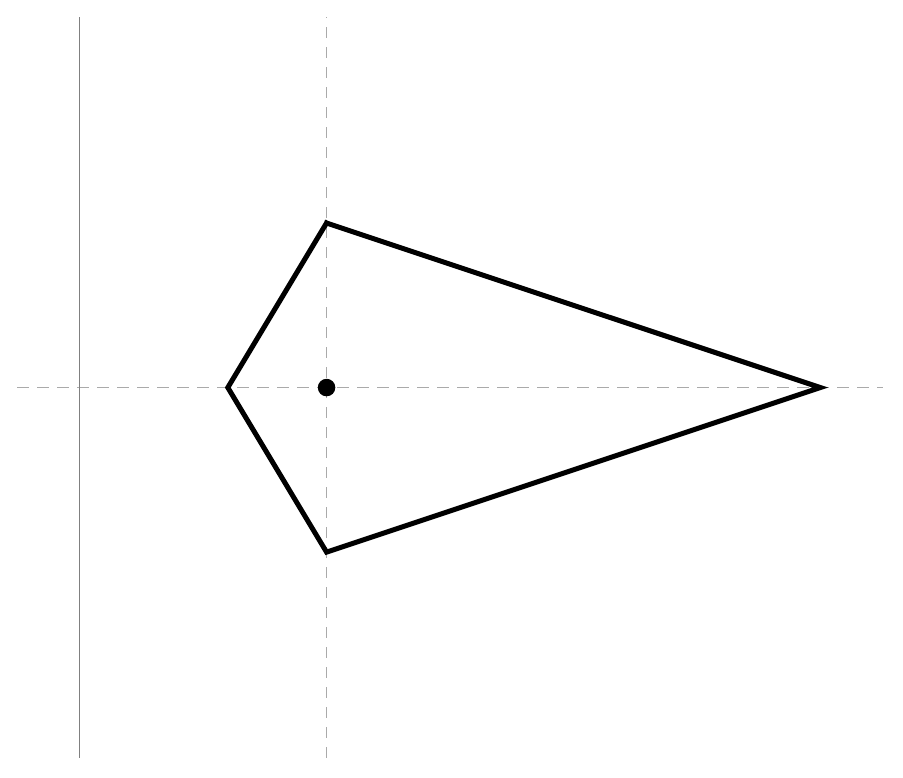}
}
\put(60,325){
\includegraphics[scale = .4, clip = true, draft = false]{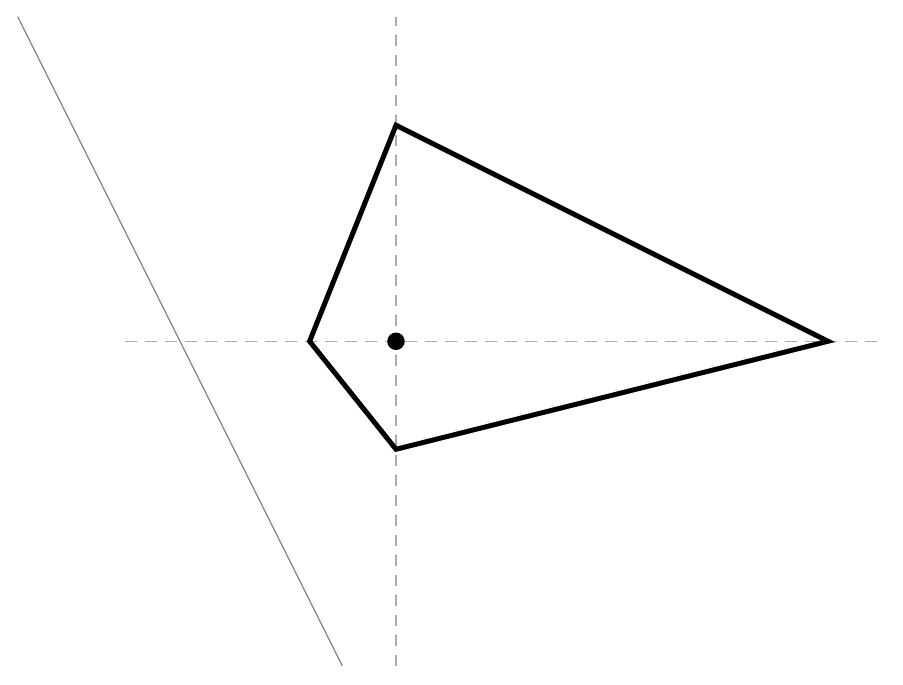}
}
\put(190,310){
\includegraphics[scale = .4, clip = true, draft = false]{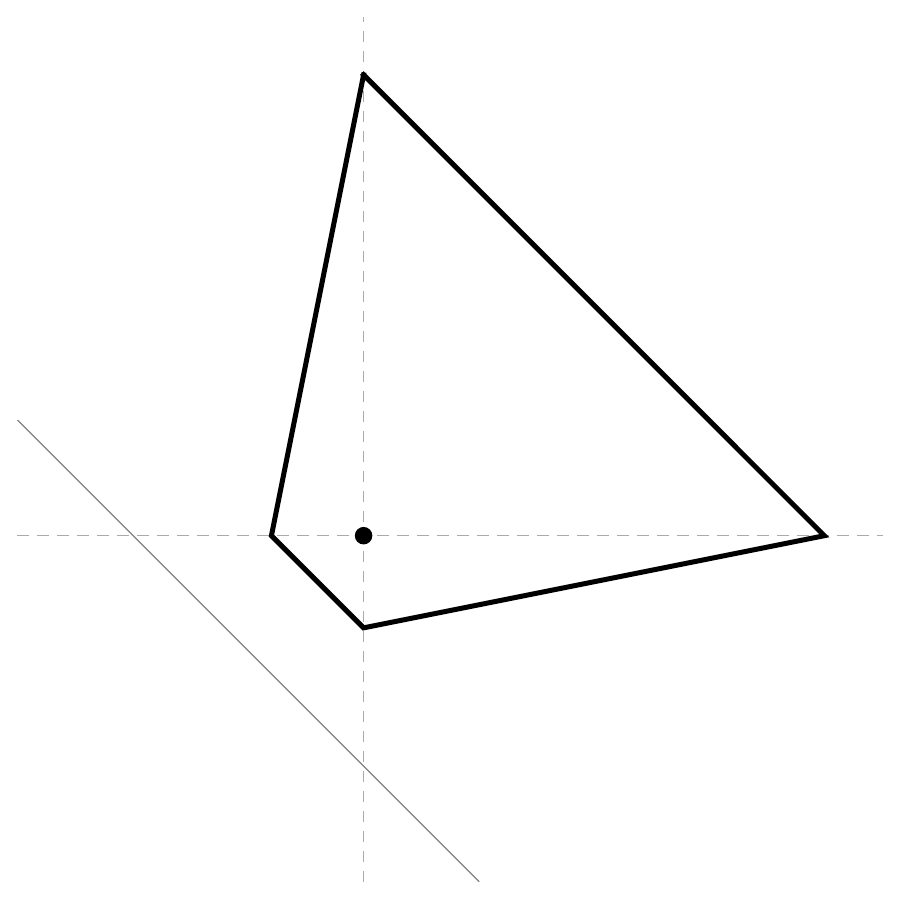}
}
\put(-55,210){
\includegraphics[scale = .4, clip = true, draft = false]{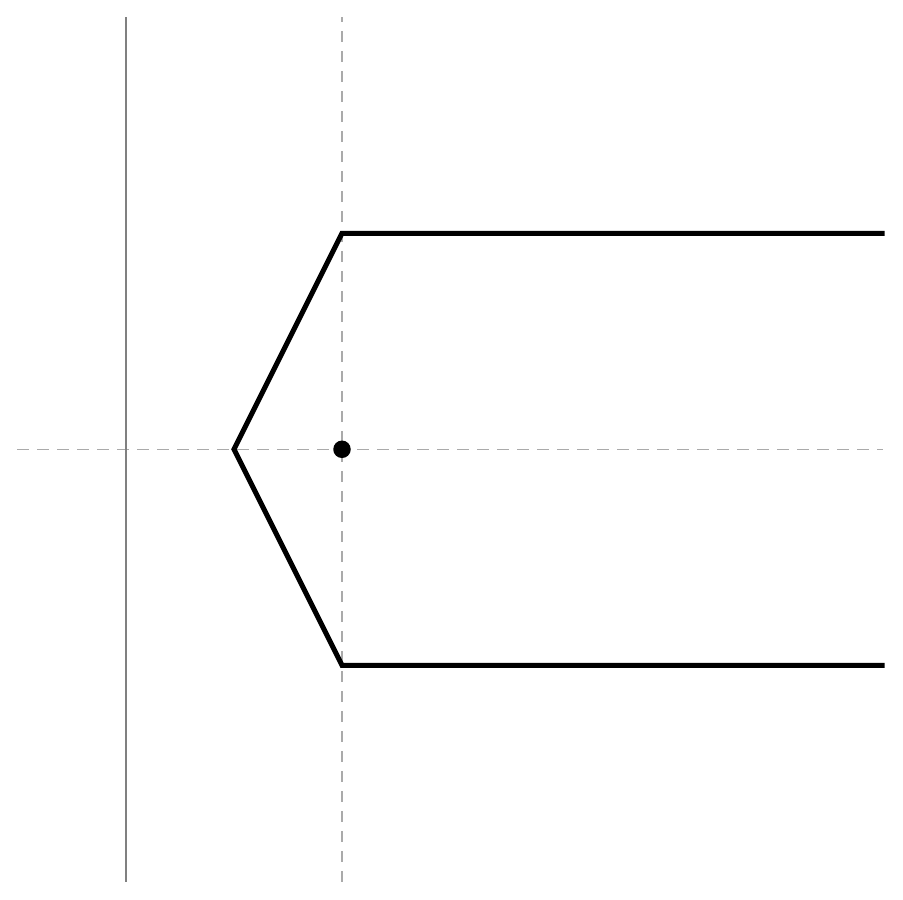}
}
\put(60,210){
\includegraphics[scale = .4, clip = true, draft = false]{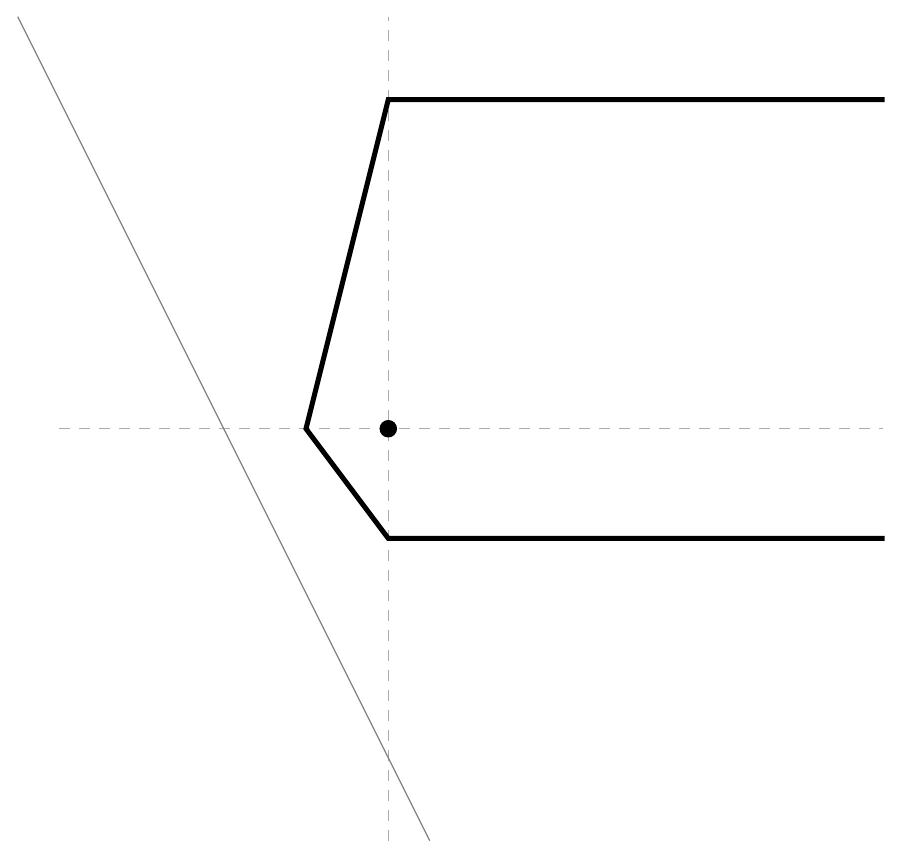}
}
\put(180,200){
\includegraphics[scale = .4, clip = true, draft = false]{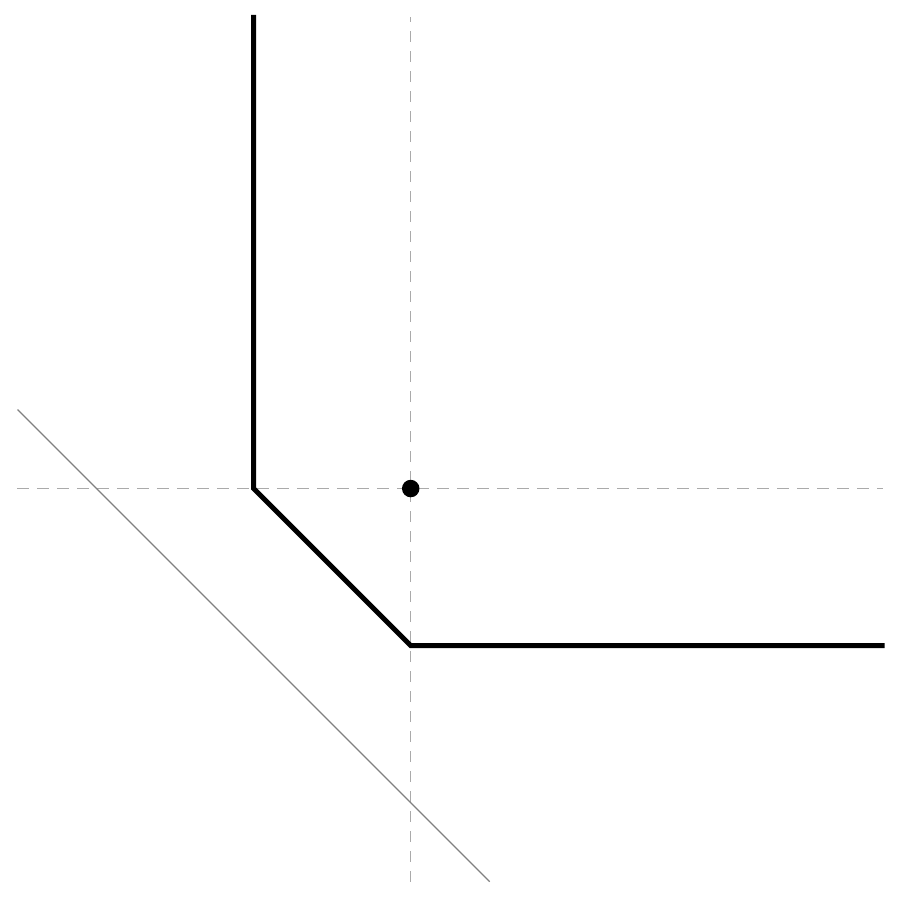}
}
\put(-50,110){
\includegraphics[scale = .4, clip = true, draft = false]{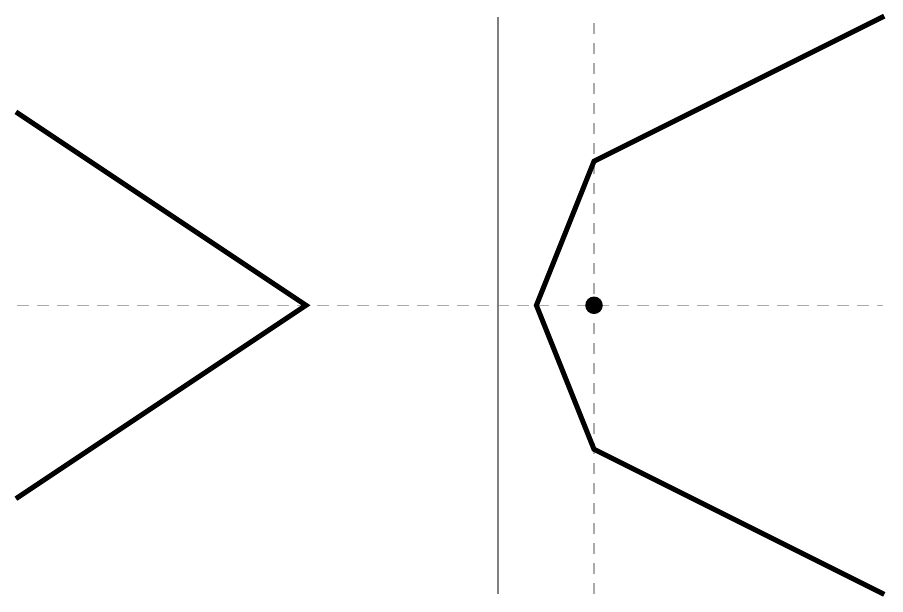}
}
\put(65,100){
\includegraphics[scale = .4, clip = true, draft = false]{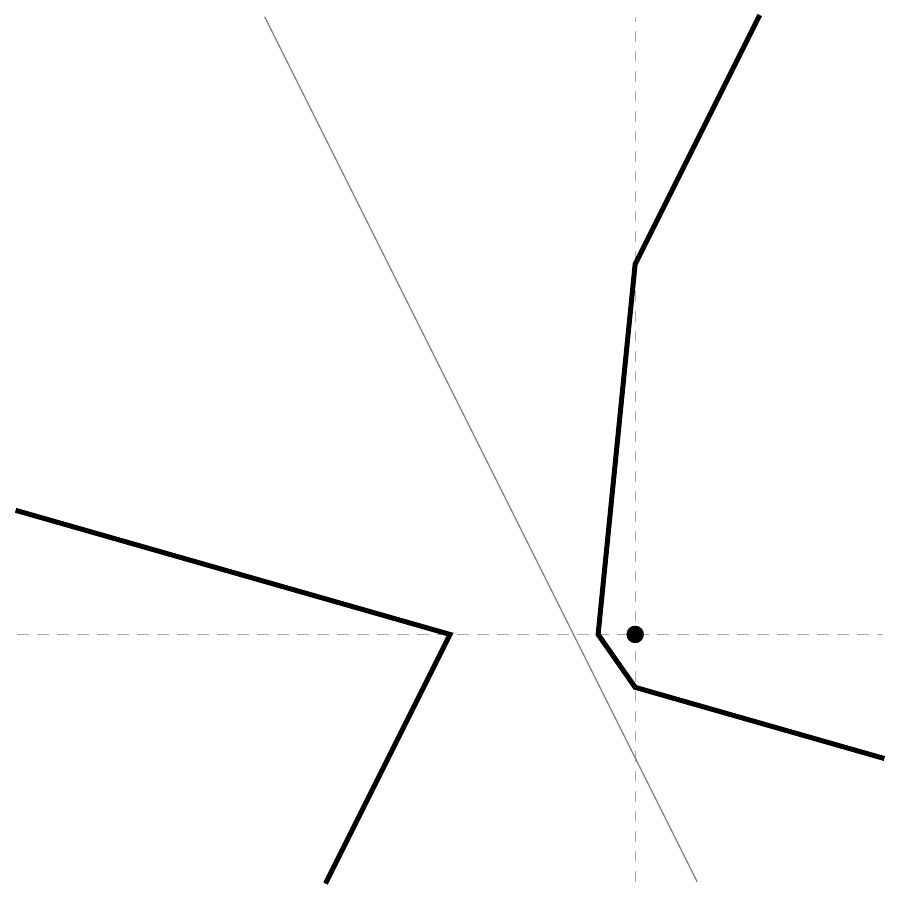}
}
\put(190,90){
\includegraphics[scale = .4, clip = true, draft = false]{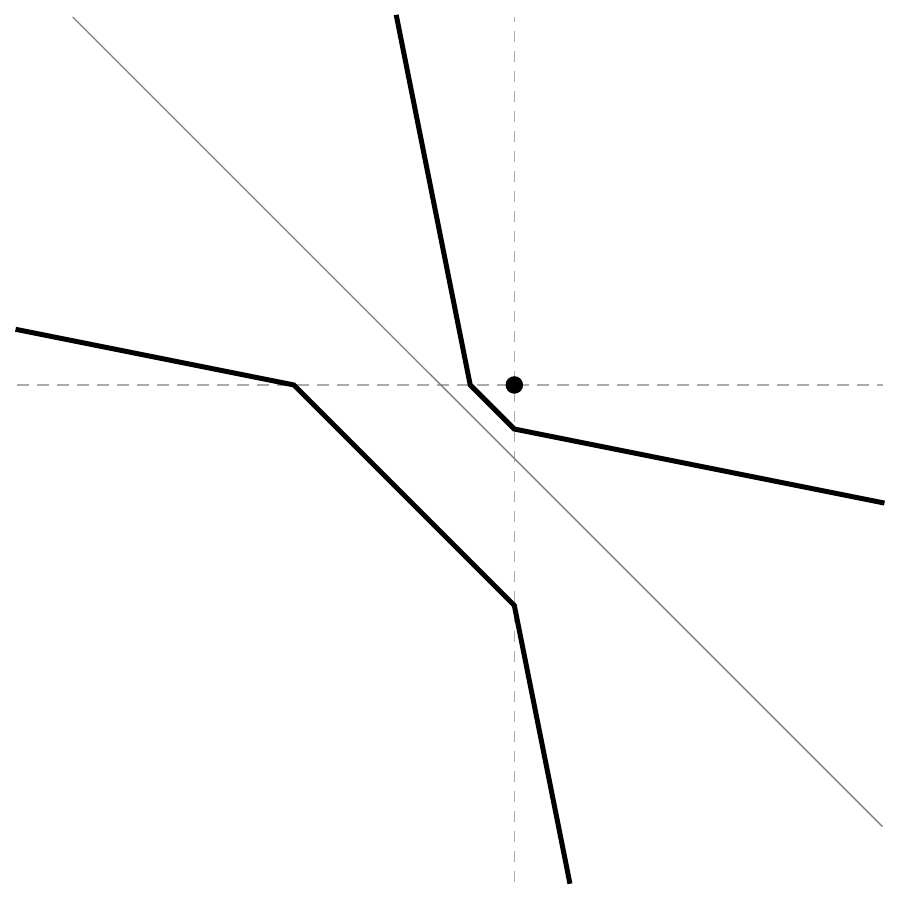}
}
\put(10,0){
\includegraphics[scale = .4, clip = true, draft = false]{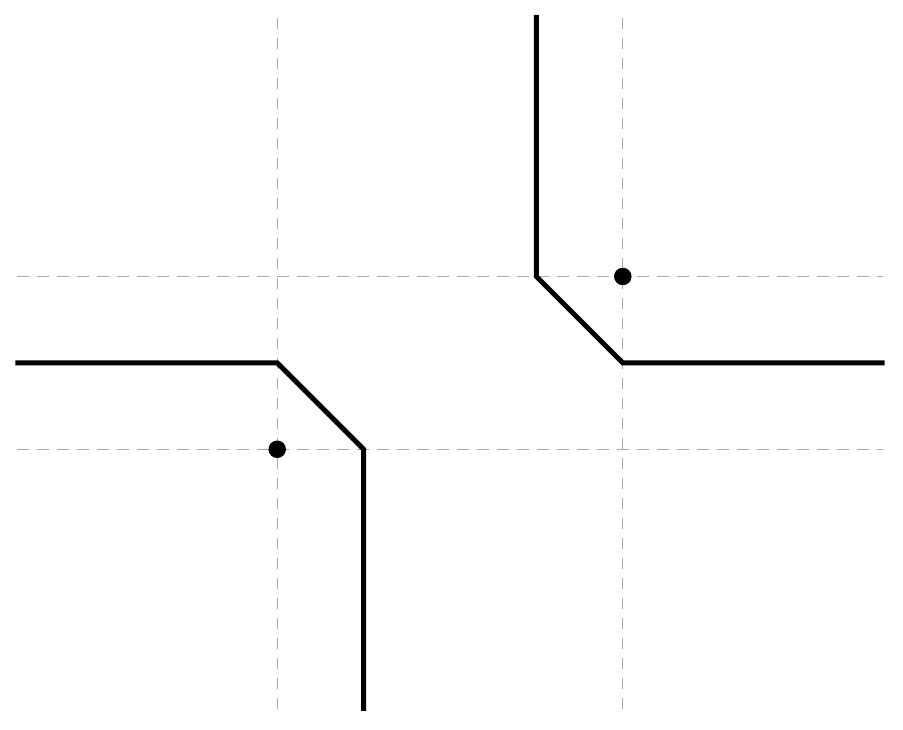}
}
\put(145,0){
\includegraphics[scale = .4, clip = true, draft = false]{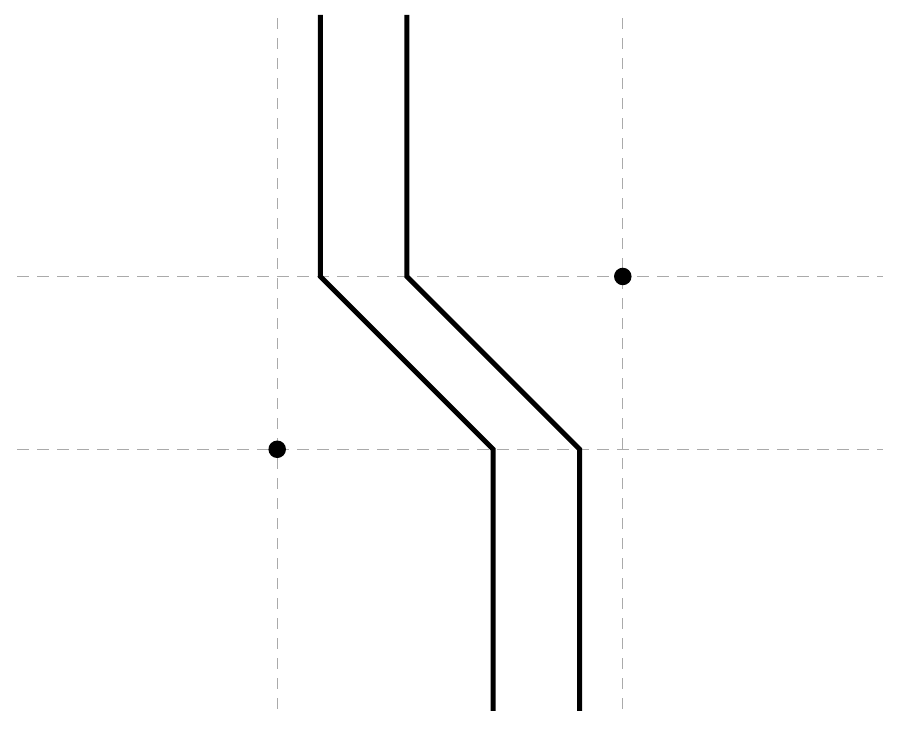}
}
\end{picture}
\caption{Some taxicab conic sections using the distance definitions.  First and last rows:  ellipses and hyperbolas using the two-foci definition; second, third, and fourth rows:  ellipses, parabolas, and hyperbolas using the focus-directrix definition. The foci and directrices, are also shown.} \label{distanceconicsfig}
\end{figure}

While this process is fairly natural, it is somewhat separate from the definition of conic section as the slice of a cone.  In this paper, the notion of a taxicab conic section is explored from the perspective of slicing cones.  In so doing, a number of technical issues arise which serve to develop a deeper understanding of the richness and subtlety of taxicab space in two and three dimensions.

For this paper, a cone is defined to be the set of points $x$ whose distance to a given line $\ell$ is a multiple of its distance to a given plane $P$
\[
C(\ell, P, \kappa) = \{x \in \mathbb{R}^3: d(x, \ell) = \kappa\, d(x, P)\}.
\]
To ensure that our conic sections are objects in taxicab two-space, we restrict our slicing plane to be a coordinate plane, while allowing the cone to vary.  The nature of the taxicab distance leads to objects that are piecewise linear.  See Figure~\ref{sectionconstructionfig} for a representation of the process under consideration.

Taxicab conic sections via sliced cones have been explored in \cite{Laatsch}.  In that setting, a single cone, defined using one coordinate axis for $\ell$ and the complementary coordinate plane for $P$, is used.  An arbitrary slicing plane is chosen, and the resulting intersection is projected onto $P$.  This process reproduces the conic sections defined using the focus-directrix method.  The exploration in this paper considers more general cones and eliminates the need to project since the slicing plane will already be a coordinate plane.

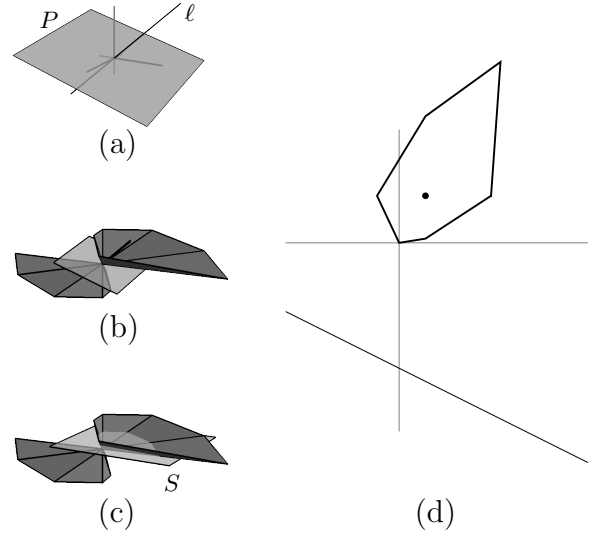
\begin{figure}
\begin{picture}(220,200)
\put(-10,130){
\includegraphics[scale = .35, clip = true, draft = false]{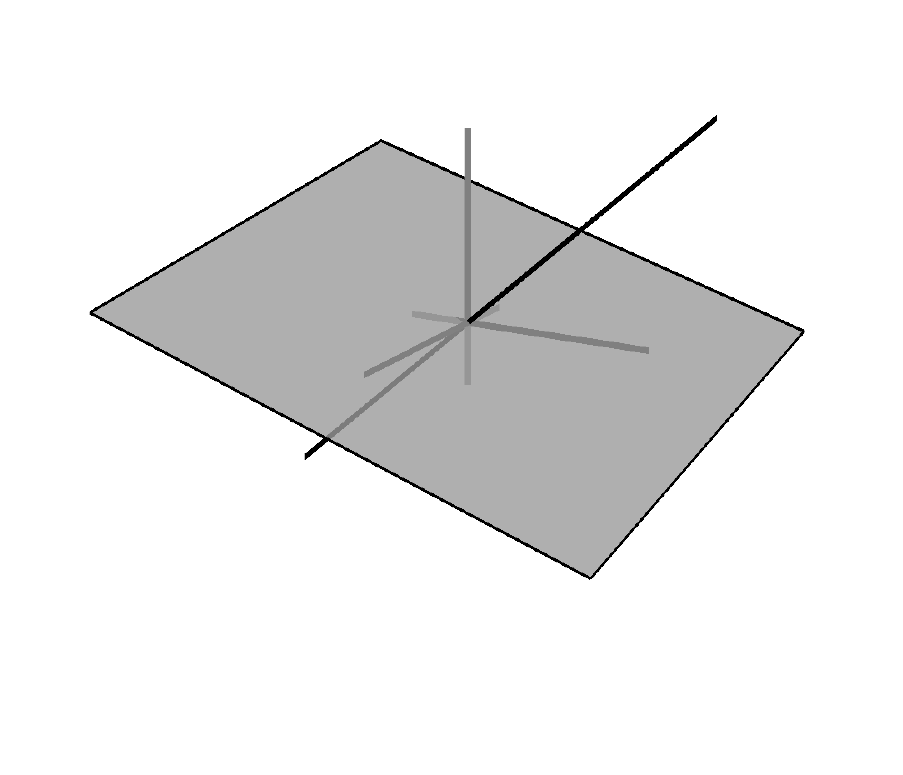}
}
\put(13,187){\scriptsize $P$}
\put(68,190){\scriptsize $\ell$}
\put(-45,60){
\includegraphics[scale = .6, clip = true, draft = false]{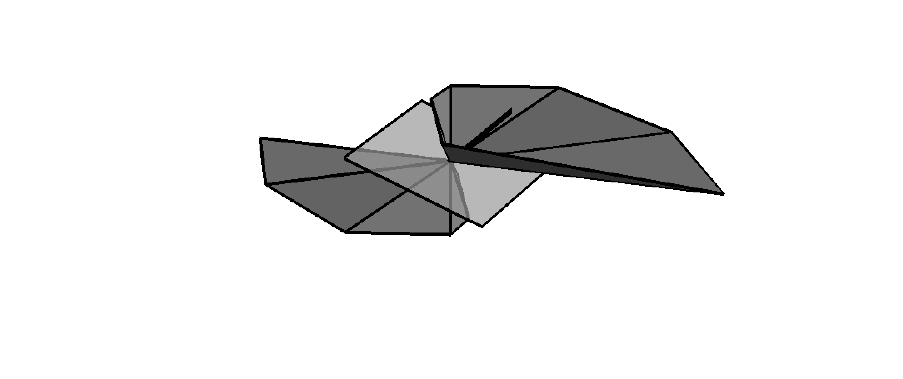}
}
\put(-45,-10){
\includegraphics[scale = .6, clip = true, draft = false]{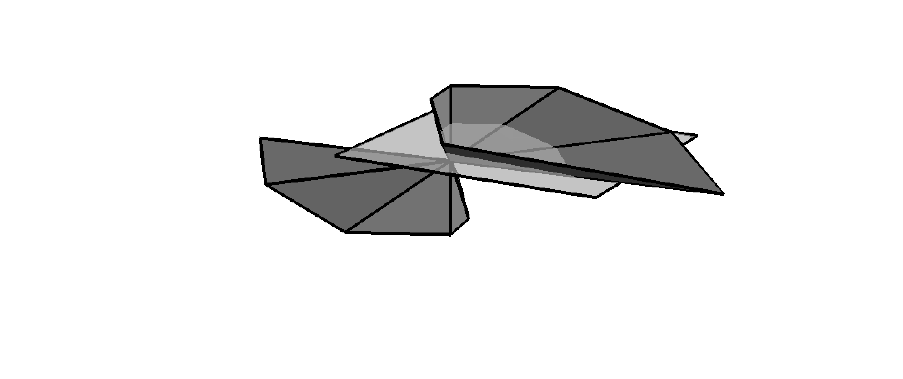}
}
\put(60,12){\scriptsize $S$}
\put(100,20){
\includegraphics[scale = .5, clip = true, draft = false]{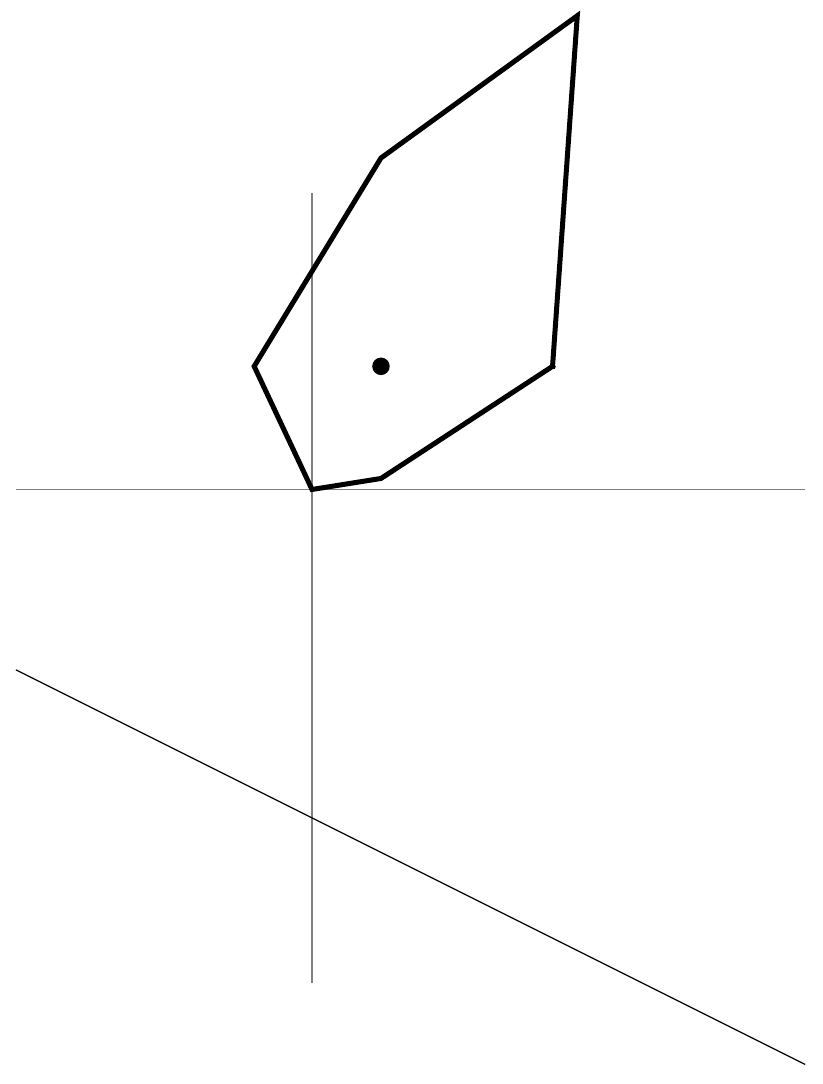}
}
\put(35,140){(a)}
\put(35,70){(b)}
\put(35,0){(c)}
\put(155,0){(d)}
\end{picture}
\caption{Given a line $\ell$ and plane $P$ (a), a cone is produced (b).  This cone is sliced by a coordinate plane $S$ (c), and the resulting intersection is represented in $\mathbb{R}^2$ (d).  In the final image, the point inside the conic section and the line below it indicate the intersections of $\ell$ and $P$ with $S$ respectively.} \label{sectionconstructionfig}
\end{figure}

A complete characterization of the conic sections arising from this more general method requires three main steps.  In Theorems~\ref{nonhorizontalverticesthm} and  \ref{horizontalverticesthm} we compute the vertices of a conic section.  Then, Theorems~\ref{sectionconstructionthm} and \ref{auxsectionconstructionthm} show how to connect these vertices by segments and rays to form the conic section.  Finally, Theorems~\ref{coniccharacterizationthm} and  \ref{horizontalconiccharacterizationthm} characterize when the resulting conic sections are ellipses, parabolas, or hyperbolas.  Multiple theorems are needed for each step because the analysis when $\ell$ is horizontal is somewhat different from when it is not horizontal.

This paper is organized as follows: in Section~\ref{taxicabmetricsection} we introduce the taxicab metric in $\mathbb{R}^2$ and $\mathbb{R}^3$, and discuss a number of its properties.  In Section~\ref{taxicabconessection} we discuss taxicab cones, paying special attention to strategies for measuring the distance between a point and a plane or a point and a line. This proves to be somewhat technical, but also leads to insight into how to think about the set of parameters that best characterize planes and lines in taxicab space.  Also in this section, we establish our conventions about the slicing plane and discuss its many important roles.  The analysis of conic sections themselves when the defining line is non-horizontal and when it is horizontal are in Sections~\ref{conicsectionswhenellisnonhorizontalsection} and~\ref{conicsectionswhenellishorizontalsection} respectively.  This analysis leads to the discovery of rich geometric relationships that in turn provide useful shortcuts for drawing the wide variety of conic sections that arise.  With methods for producing and understanding conic sections established, we explore some special cases in Section~\ref{specialcasessection}.  Finally, we conclude with some final thoughts and possible next steps in Section~\ref{nextstepssection}.

\section{The taxicab metric} \label{taxicabmetricsection}
In this section, we establish our notation and review some basic facts about the taxicab metric in $\mathbb{R}^2$ and $\mathbb{R}^3$ that motivate many of the choices that are useful for our analysis.

\subsection{Definitions}

The taxicab distance between points $x=(x_1, x_2, \ldots, x_n)$ and $y=(y_1, y_2, \ldots, y_n)$ in $\mathbb{R}^n$ is
\[
d(x,y)=|x_1-y_1|+|x_2-y_2|+ \cdots + |x_n-y_n|.
\]
More generally, the distance between a point $x$ and set $T$ is
\[
d(x, T) = \inf_{y \in T} d(x, y).
\]
We will also want to measure partial distances between a point $x$ and set $T$, so we define
\[
d_i(x, T) = \inf_{y \in T: y_j = x_j \forall j \neq i} |x_i - y_i|
\]
and
\[
d_{i,j}(x, T) = \inf_{y \in T: y_k = x_k \forall k \neq i, j}  |x_i - y_i| +  |x_j - y_j|.
\]
Note that $d_i(x, T)$ is the infimum of the distances between $x$ and points in $T$ on the $i$th coordinate line through $x$, and $d_{i,j}(x, T)$ is the infimum of the distances between $x$ and points in $T$ on the $(i,j)$th coordinate plane through $x$.

A sphere centered at the point $y$ with radius $r \geq 0$ in $(\mathbb{R}^n, d)$ is
\begin{align*}
\sigma_r(y) &= \{x \in \mathbb{R}^n: d(x,y) = r\} \\
&= \{x \in \mathbb{R}^n: |x_1 - y_1|+|x_2 - y_2| + \cdots + |x_n - y_n|=r\}.
\end{align*}
When $n = 2$, this is a square with vertices on the coordinate lines passing through the center.  When $n = 3$, this is an octahedron with vertices on the coordinate lines passing through the center.  See Figure~\ref{taxicabdistandspherefig}.  When $n = 3$, the subsets of $\sigma_r(y)$ determined by restricting to a coordinate slice passing through $y$ will be useful, and we call these three sets \textit{great circles} of the taxicab sphere.

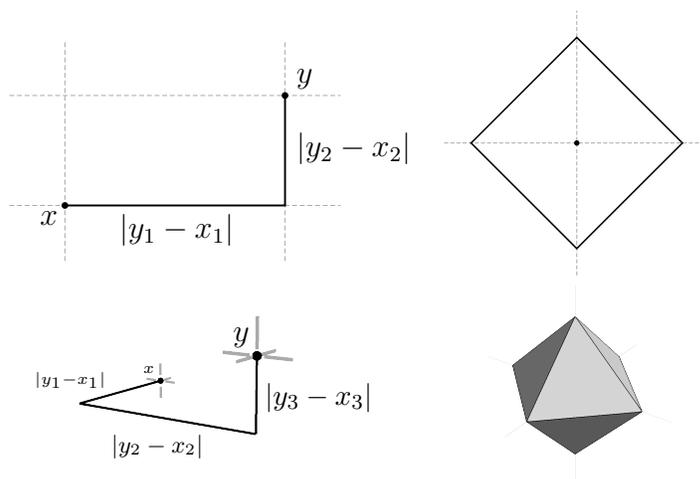
\begin{figure}
\begin{picture}(250,180)
\put(-20,80){
\includegraphics[scale = .5, clip = true, draft = false]{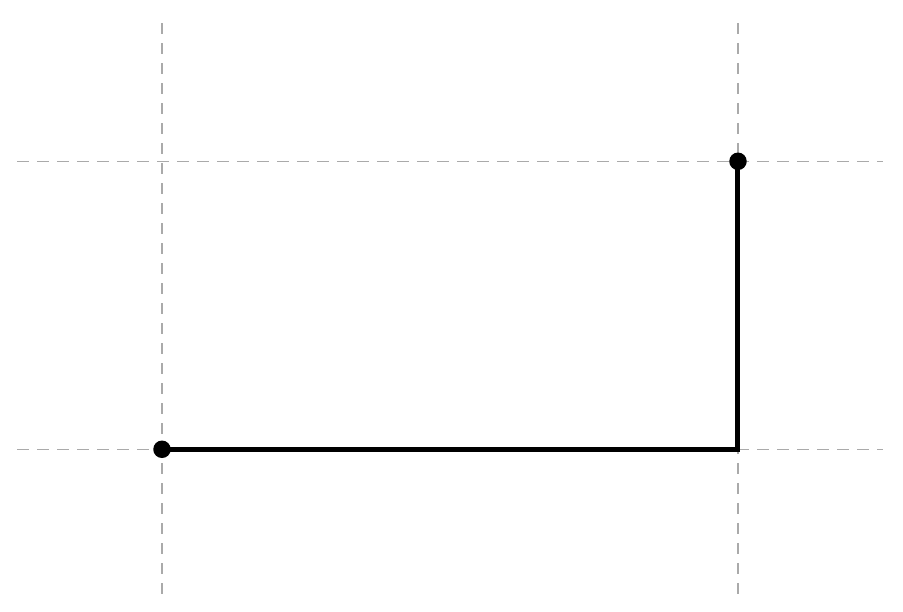}
}
\put(145,75){
\includegraphics[scale = .4, clip = true, draft = false]{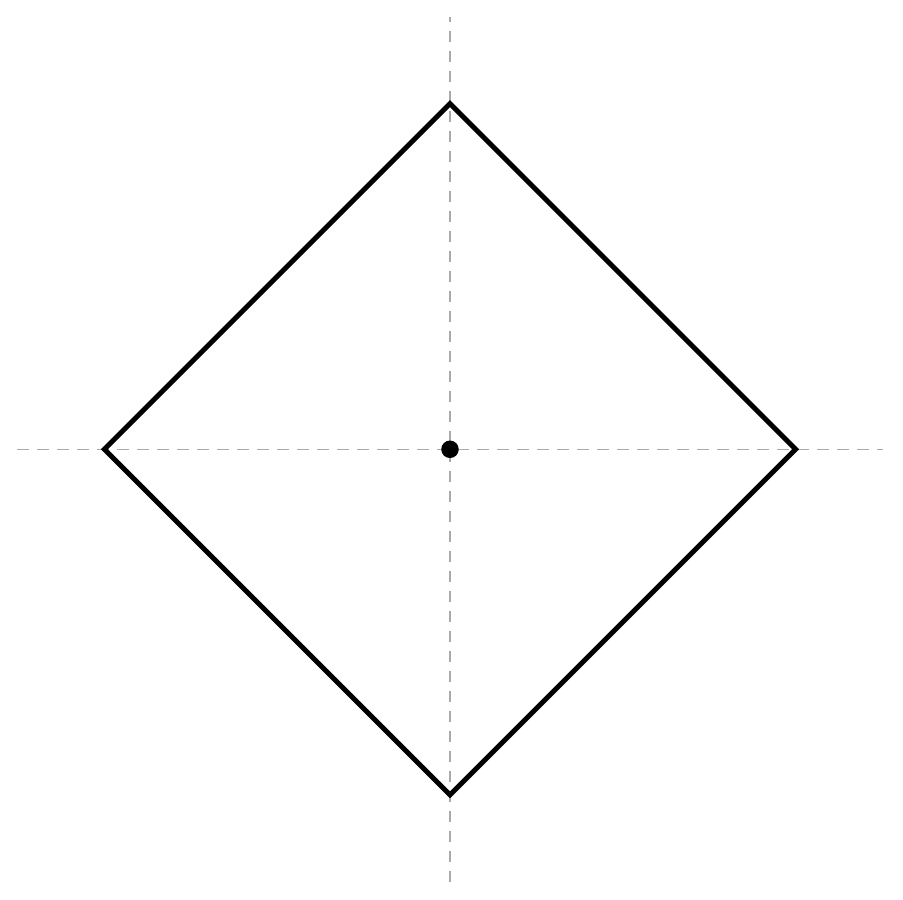}
}
\put(-10,0){
\includegraphics[scale = .5, clip = true, draft = false]{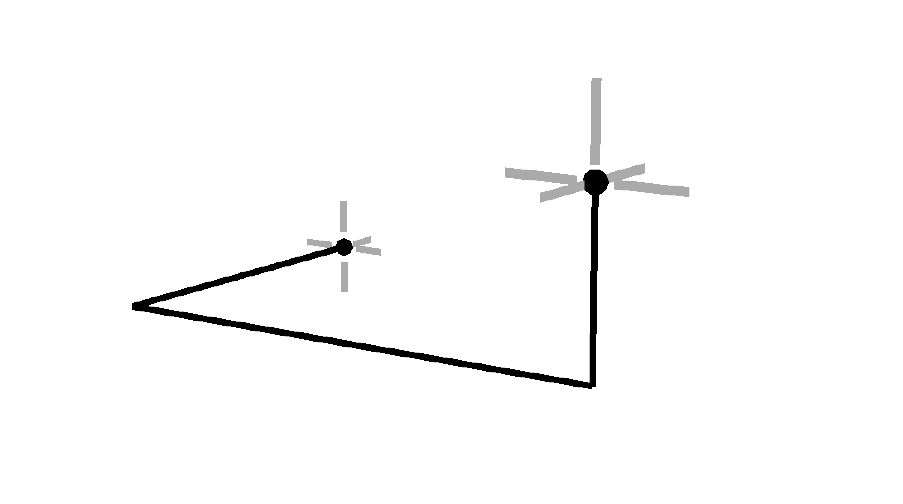}
}
\put(130,-40){
\includegraphics[scale = .5, clip = true, draft = false]{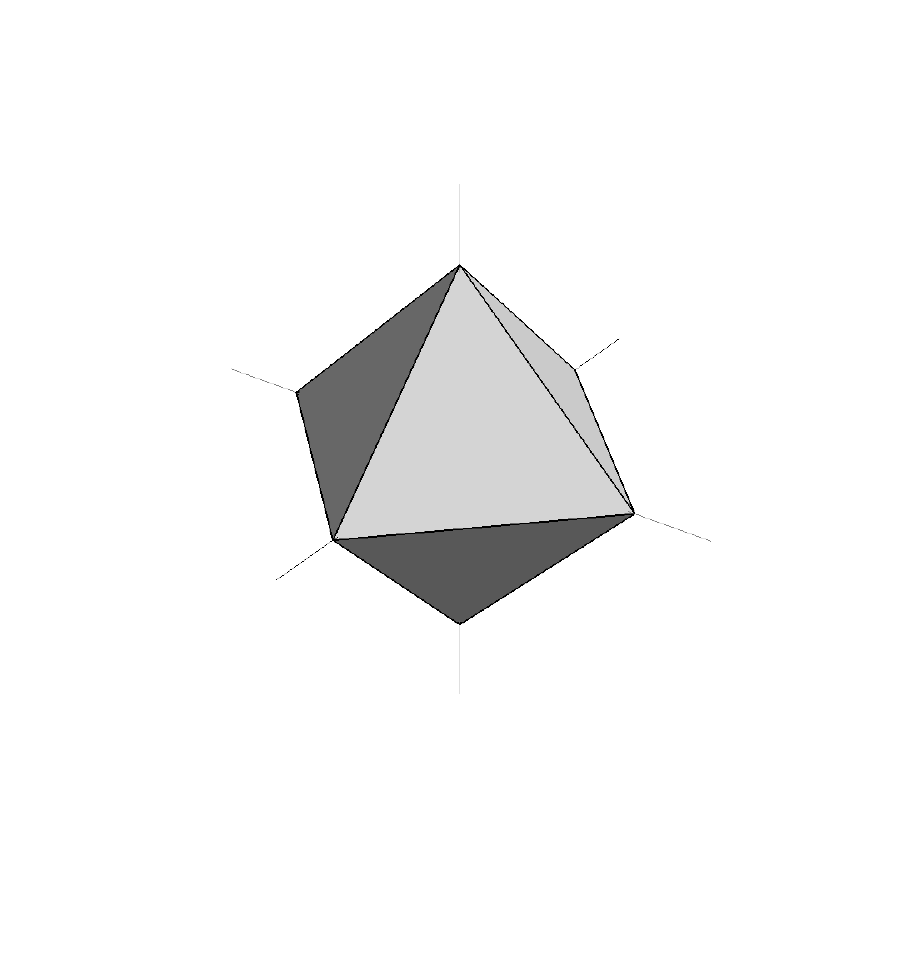}
}
\put(-2,96){$x$}
\put(95,150){$y$}
\put(28,91){$|y_1 - x_1|$}
\put(95,122){$|y_2 - x_2|$}
\put(37,40){$\scriptscriptstyle x$}
\put(71,52){$y$}
\put(-4,36){$\scriptscriptstyle|y_1 - x_1|$}
\put(25, 10){\scriptsize $|y_2 - x_2|$}
\put(83, 28){\small $|y_3 - x_3|$}
\end{picture}
\caption{Taxicab distance visualized in $\mathbb{R}^2$ and $\mathbb{R}^3$, along with a taxicab circle and a taxicab sphere.} \label{taxicabdistandspherefig}
\end{figure}

\subsection{Isometries}

The group of isometries of two- and three-dimensional taxicab space that fix a point are both finite, unlike for Euclidean space.  While a full justification takes a bit of work, it is not surprising that the group of isometries that fix a point in the taxicab plane is isomorphic to the symmetry group of a square \cite{isometries1,isometries2}, while for taxicab 3-space, it is isomorphic to the symmetry group of an octahedron.

For the purposes of this paper, it is enough to know that arbitrary translations, rotations about coordinate lines by multiples of $\frac{\pi}{2}$, and reflections across coordinate planes are all isometries.  The richer structure of the isometry group for $(\mathbb{R}^3, d)$ will not be necessary.

\subsection{Subspaces and inherited metrics} \label{subspacesec}
Given $(\mathbb{R}^3, d)$, the coordinate planes will naturally inherit a taxicab metric.  Less clear is what happens for planes that lie obliquely to one or more coordinate axes.  Some intuition can be gained by noting how planes through the origin intersect $\sigma_1(0)$.

Planes containing one coordinate line intersect in rhombi which can be stretched to form squares.  We leave it as an exercise for the reader to show that in this case, the induced metric is a taxicab metric but that appropriate coordinates need to be chosen to see this.  These coordinates could not be extended to coordinates in $\mathbb{R}^3$ that are related isometrically to the ambient coordinates since there is no isometry that  moves a non-coordinate plane to a coordinate plane.

Next, planes containing no coordinate lines do not inherit a taxicab metric at all.   Such planes intersect $\sigma_1(0)$ in hexagons and as a result, the isometry group for this space differs from that for the taxicab plane.  See \cite{Sowell} for an analysis of the induced metric on the plane containing the points $(1,0,0)$, $(0,1,0)$, and $(0,0,1)$.

In light of these technical issues, we make the choice in our analysis here to restrict to coordinate planes when we want the geometry in that plane to be a taxicab geometry that is consistent with the ambient taxicab structure.  As such, we do not restrict the planes that define the cone, but we do restrict the planes that slice the cone to produce a section.  These issues are also what motivate our definition of great circles.

\subsection{Angles}
Unlike the Euclidean metric, the taxicab metric does not arise from an inner product.  As such, while notions of angle can be introduced, none will enjoy all the properties and relationships we are used to in Euclidean geometry.  See \cite{ThompsonDray} for further discussion of this along with a formulation of angle in $(\mathbb{R}^2,d)$ that captures many of the geometric aspects with which we are familiar.

For our purposes, rather than introducing a specific definition of angle, we avoid the use of angle for any of the constructions performed.  The main impact of this is that it broadens the set of cones that are explored.  As indicated above, since our slicing plane is restricted to coordinate planes, we will allow any plane to be used to define cones.  Furthermore, since the notion of angle does not naturally emerge from taxicab distance, requiring the line and the plane to be perpendicular is somewhat artificial.  As such, we only require that the line and plane intersect at a single point, increasing the variety of conic sections found.  Also, see Section~\ref{lineplaneperpsubsec} for more discussion about this choice.

Note that in the Euclidean setting, a circular cone can be defined in terms of a plane $P$ and a line $\ell$ that is perpendicular to $P$.  Relaxing the perpendicularity constraint does change the shape of the cone, but does not increase the types of objects that might be called conic sections.  For a line and plane in general position, the equation defining the cone is still quadratic, and slicing that cone by a plane preserves the quadratic nature of the equation defining the resulting intersection.  Since all such quadratic equations define some conic section, nothing new is gained.  As such, it is worth considering a line and a plane in general position in the taxicab setting as well.

\section{Taxicab cones} \label{taxicabconessection}

In this section, we develop the structure necessary to understand the cones under consideration.  This involves a careful analysis of the distance between a point and a plane and the distance between a point and a line.  We also establish conventions for our slicing plane, and discuss its many uses.

\subsection{Cones}
Let $\ell$ be a line and let $P$ be a plane where, to avoid degenerate cases, $\ell \cap P$ consists of a single point.  Also, let $\kappa \in (0, \infty)$.  We define a cone to be the set of points where the distance to the line is equal to a constant multiple of the distance to the plane:
\[
C(\ell, P, \kappa)=\{x\in\mathbb{R}^3:d(x,\ell)= \kappa \, d(x,P)\}.
\]
We call $\ell$ the defining line and $P$ defining plane for the cone.  Without loss of generality, we restrict our attention to the cases where $\ell \cap P = \{(0,0,0)\}$.

We represent the defining plane by choosing $A \in \mathbb{R}^3 \backslash \{0\}$ as follows:
\[
P = P_A = \{y \in \mathbb{R}^3: A_1 y_1 + A_2 y_2 + A_3 y_3 = 0\}.
\]

We represent the defining line by choosing $a  \in \mathbb{R}^3 \backslash \{0\}$ such that
\[
\ell = \ell_a = \{at: t \in \mathbb{R}\}
\]
and we denote the point corresponding to the value $t$ by $\ell(t)$.

With these conventions in place, we note that a cone is degenerate if and only if $A_1 a_1 + A_2 a_2 + A_3 a_3 = 0$ and we will at times make use of the fact that $A_1 a_1 + A_2 a_2 + A_3 a_3$ is non-zero without explicitly mentioning it.

\subsection{The slicing plane} \label{slicingplanesec}

The primary motivation for this paper is to find sets in the taxicab plane that represent slices of cones.  As mentioned above, we choose to restrict the slicing plane to be a coordinate plane and without loss of generality, we restrict our choice of slicing plane further to the plane $S = \{x \in \mathbb{R}^3: x_3 = 1\}$.  This is justified in two steps.  First, the coordinate plane $\{x_j = h\}$ can be rotated to $\{x_3 = |h|\}$ using taxicab isometries.  Second, since lines and planes through the origin are invariant under dilation, cones are as well.  As such, the resulting conic sections for different choices of $h$ will be similar to one another, and so we can rescale to $h = 1$.

While this choice is arbitrary, we will use it to establish terminology that will clarify the variety of cases under consideration.  In particular, we say the slicing plane, and any plane parallel to it, is horizontal.

\subsection{The defining plane}

In this section, we collect some useful facts about the defining plane.  After determining a formula for the distance between a point and a plane, we explore parameter spaces for planes that capture some of the geometric features exposed by the distance formula.

\subsubsection{Distance between a point and a plane}

In order to perform computations involving a cone, a more explicit understanding of the distance between a point and a plane will be helpful.

\begin{thm} \label{disttoplanethm}
Given a point $x$ and a plane $P = P_A$,
\begin{equation}  \label{disttoplaneeq}
d(x, P) =  \frac{\left|A_1 x_1 + A_2 x_2 + A_3 x_3 \right|}{\mathrm{max}\{|A_1|, |A_2|, |A_3|\}}.
\end{equation}
\end{thm}

\begin{proof}
If $x \in P$ then the result follows immediately. Otherwise, consider $\sigma_r(x)$ centered at $x$. Starting with $r = 0$, as $r$ is increased until $\sigma_r(x)$ makes contact with the plane, when the first contact occurs the point(s) of contact will include a vertex. The distance from $x$ to that vertex is the distance to the plane so
\[
d(x, P) = \min_{i \in \{1, 2, 3\}} d_i(x, P).
\]

Equation~\eqref{disttoplaneeq} follows from explicitly computing $d_i(x, P)$.  Let $y^i$ be the point in the plane sharing two coordinates with $x$, and where the $i$ coordinate differs, noting that if $A_i = 0$, then this point does not exist and the corresponding partial distance is infinite.  Then, since $y^i \in P$,
\[
A_1 y^i_1 + A_2 y^i_2 + A_3 y^i_3 = 0
\]
so that
\[
y^i_i = \frac{-A_j y_j - A_k y_k}{A_i} = \frac{-A_j x_j - A_k x_k}{A_i}
\]
and so
\begin{align*}
d(x,y^{i}) &= |x_i - y_i^i| \\ 
		&= \left| x_i +  \frac{A_j x_j + A_k x_k}{A_i} \right| \\
		&= \frac{1}{|A_i|} |A_1 x_1 + A_2 x_2 + A_3 x_3|.
\end{align*}
From this, it follows that the distance to the plane is given by Equation~\eqref{disttoplaneeq}.
\end{proof}

In light of this result, we say a plane is shallow if $\max\{|A_1|, |A_2|, |A_3|\}$ is $|A_3|$ alone, we say a plane is transitional if $\max\{|A_1|, |A_2|, |A_3|\}$ is $|A_3|$ along with at least one of $\{|A_1|, |A_2|\}$, and we say a plane is steep otherwise.  In particular, we say a plane is vertical if $A_3 = 0$.

\subsubsection{Parameter space for defining planes}

As defined above, any $A \in \mathbb{R}^3\backslash \{0\}$ serves to define a plane.  However, if $\tilde{A}$ is a nonzero multiple of $A$ then both of these vectors define the same plane.  As such,  from a topological perspective, the set of planes can be identified with the projective plane $\mathbb{R}P^2$.  Inspired by Theorem~\ref{disttoplanethm}, a representation of $\mathbb{R}P^2$ that respects the way we measure the distance between a point and a plane is
\[
\mathscr{P} = \{A \in \mathbb{R}^3: \max\{|A_1|, |A_2|, |A_3|\} = 1\}.
\]
This is a cube centered at the origin and when restricting to $\mathscr{P}$, $d(x, P)$ simplifies to
\[
d(x, P) = |A_1 x_1 + A_2 x_2 + A_3 x_3|.
\]
While geometrically pleasing, we recognize that this is a double cover of $\mathbb{R}P^2$ with each plane represented twice because a parameter vector and its negative describe the same plane.

Thinking again about the sphere centered at $x$ that makes first contact with a given plane, the open faces of the cube $\mathscr{P}$ correspond to planes which make contact with this sphere at just a vertex, the edges of $\mathscr{P}$ correspond to planes that make contact along an edge of the sphere, and the vertices of $\mathscr{P}$ correspond to the planes that make contact along an entire face of the sphere.  Additionally, the face of the cube in which a given parameter lies identifies the vertex that will make first contact.

For this parameter space, shallow planes correspond to the top and bottom faces, transitional planes correspond to the top and bottom square edges, and steep planes correspond to the sides.

\subsubsection{An alternative parameter space}

If a plane $P_{\widetilde{A}}$ is not vertical, we can choose an alternative parameterization by dividing the parameter vector by $\widetilde{A}_3$.  This is a gnomonic projection of sorts for the upper half of $\mathscr{P}$.  The map is $p: \mathscr{P} \cap \{x \in \mathbb{R}^3: x_3 > 0\} \rightarrow \mathbb{R}^3$, $(A_1, A_2, A_3) = p(\widetilde{A}) = \left(\frac{\widetilde{A}_1}{\widetilde{A}_3}, \frac{\widetilde{A}_2}{\widetilde{A}_3}, 1 \right)$ and provides us with a unique parameter vector for each non-vertical plane.

In this setting, vertical planes can be parameterized by a circle at infinity, represented by nonzero vectors with $A_3 = 0$.  This does not produce a unique parameter for each plane, but we will find that this redundancy does not significantly impact the analysis.  As such, we have the alternative parameter space
\[
\mathscr{P}' = \{ A \in \mathbb{R}^3: A_3 = 1 \}
	\cup \{ A \in \mathbb{R}^3 \backslash \{0\}: A_3 = 0 \}.
\]

We will find that this alternative parameter space will serve well in light of our choice of a horizontal slicing plane.  In this setting, $P$ is shallow if $A$ lies in the open central square $(-1,1) \times (-1,1)$, transitional if $A$ lies in the boundary of this square, and steep otherwise.  See Figure~\ref{planeparameterspacefig} for visualizations of these parameter spaces.

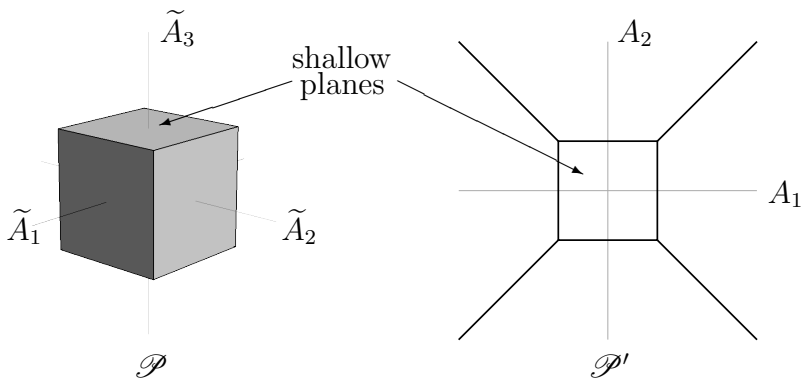
\begin{figure}
\begin{picture}(250,130)
\put(-70,-35){
\includegraphics[scale = .7, clip = true, draft = false]{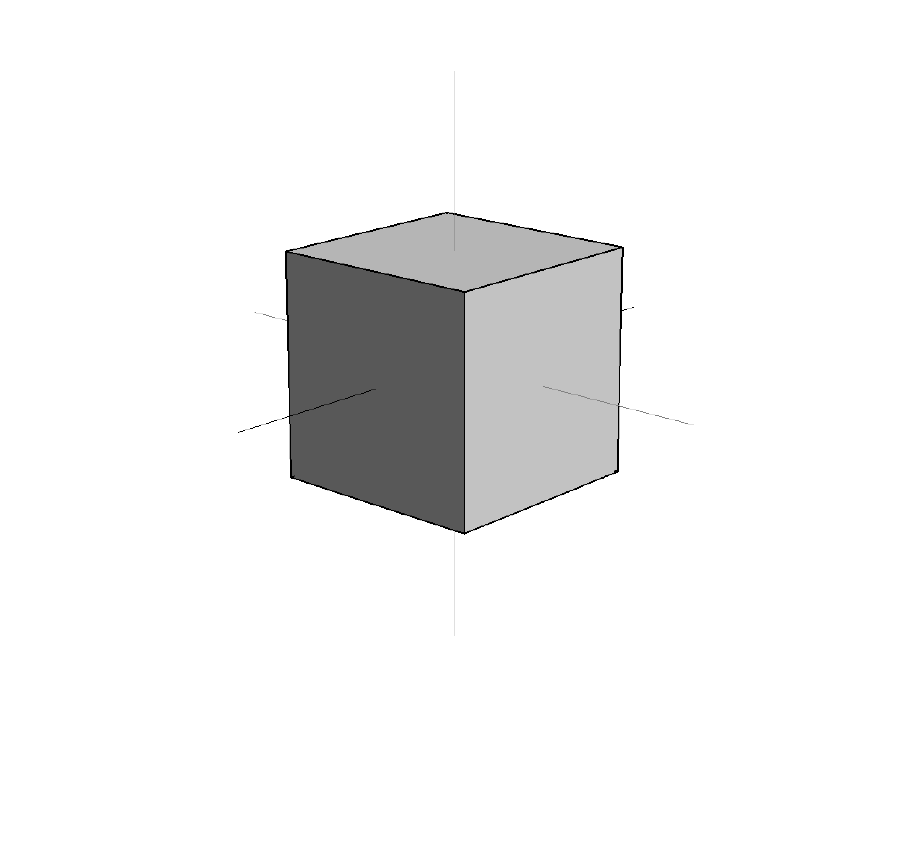}
}
\put(137,5){
\includegraphics[scale = .45, clip = true, draft = false]{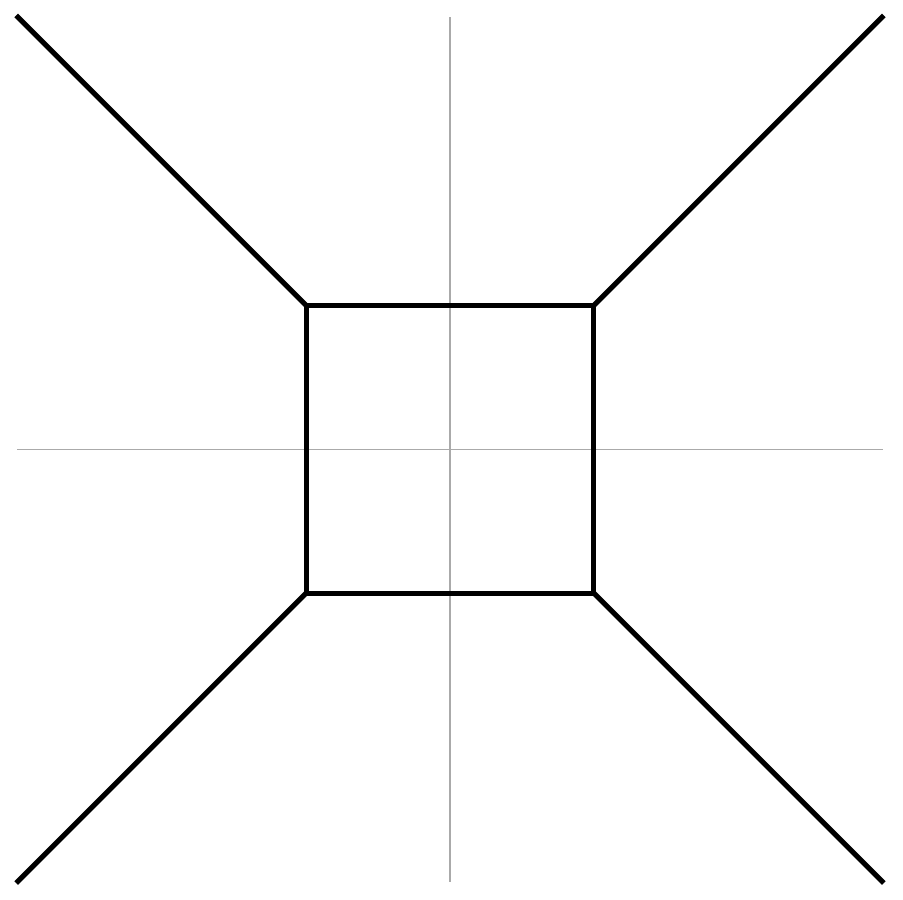}
}
\put(80,110){shallow}
\put(83,100){planes}
\put(80, 105){\vector(-3,-1){50}}
\put(120, 105){\vector(2,-1){70}}
\put(-28,45){$\widetilde{A}_1$}
\put(76,45){$\widetilde{A}_2$}
\put(30,120){$\widetilde{A}_3$}
\put(260,59){$A_1$}
\put(203,120){$A_2$}
\put(20,-7){$\mathscr{P}$}
\put(192,-7){$\mathscr{P}'$}
\end{picture}
\caption{Two parameter spaces for planes.  The cube $\mathscr{P}$ is a natural choice because of the way distance between a point and a plane is measured.  The planar parameter space $\mathscr{P}'$ is a good choice in the context of conic sections because of our choice of slicing plane.  Note that for $\mathscr{P}'$ the circle at infinity is not shown.} \label{planeparameterspacefig}
\end{figure}

\subsection{The defining line}

In this section, we collect some useful facts about the defining line.  We find that the computation of the distance between a point and a line proves to be somewhat more complicated than the distance between a point and a plane.  After introducing some technical facts about sums of absolute values, we then carefully explore the geometry in order to establish a general method for computing the distance between a point and a line.  Finally, we use this to develop appropriate parameter spaces for lines.

\subsubsection{Sums of absolute values}
Given a set of numbers $a_1, \ldots, a_N$, we say $a_j$ dominates the other values if
\[
|a_j| > \sum_{i = 1, i \neq j}^N |a_i|
\]
and we say $a_j$ transitionally dominates the other values if
\[
|a_j| = \sum_{i = 1, i \neq j}^N |a_i|.
\]
We will find that measuring the distance to a line $\ell_a$ requires knowing whether or not a component of the parameter $a$ dominates the other components.  The following technical lemma, the proof of which is left to the reader, captures the essential reason for this.

\begin{lemma} \label{sumabslemma}
Let
\begin{align*}
f(t) &= \sum_{i = 1}^N |b_i - m_i t| \\
    &= \sum_{i = 1}^N |m_i|\left|\frac{b_i}{m_i} - t\right|
\end{align*}
where the $m_i$ are all nonzero.  Note that if $\frac{b_i}{m_i} = \frac{b_j}{m_j}$ the corresponding terms can be merged.  Wth this in mind, suppose, without loss of generality, that the indexing is such that the $\frac{b_i}{m_i}$ are in strictly increasing order.  Then:
\begin{itemize}
\item If for all $j \in \{1, \ldots, N-1\}$
\[
\sum_{i = 1}^j |m_i| - \sum_{i = j+1}^N |m_i| \neq 0
\]
then the critical points for $f$ are exactly the points $t = \frac{b_i}{m_i}.$
\item If there exists a (necessarily unique) $j \in \{1, \ldots, N-1\}$ such that
\[
\sum_{i = 1}^j |m_i| - \sum_{i = j+1}^N |m_i| = 0
\]
then the critical points for $f$ are exactly the points $t = \frac{b_i}{m_i}$ and the interval $\left[\frac{b_j}{m_j}, \frac{b_{j+1}}{m_{j+1}}\right].$

\item The function $f$ is convex and as such has no local maxima.  It achieves one local, and hence global, minimum value at either a unique point or an interval.

\item Let $j$ be the index where
\[
\sum_{i=1}^{j-1} |m_i| < \sum_{i=j}^{N} |m_i|
\]
and
\[
\sum_{i=1}^{j} |m_i| \geq \sum_{i=j+1}^{N} |m_i|.
\]
Then the local minimum occurs at $\frac{b_j}{m_j}$ or, in the case of equality in the second line, on the interval $\left[\frac{b_j}{m_j}, \frac{b_{j+1}}{m_{j+1}}\right].$

\item Suppose for some $j \in \{1, \ldots, N\}$, $m_j$ dominates the other $m_i$.  Then the global minimum is realized at $t = \frac{b_j}{m_j}$.

\item In the special case where $N = 3$,
\begin{itemize}
\item if there is no dominant slope, then the minimum occurs
the middle value in the set $\left\{\frac{b_1}{m_1}, \frac{b_2}{m_2},\frac{b_3}{m_3} \right\}$;
\item if $m_j$ is transitionally dominant then the minimum also occurs at $\frac{b_j}{m_j}$, and if this value is not the middle value, then the minimum occurs along an interval.
\end{itemize}
\end{itemize}

\end{lemma}

See Figure~\ref{absgraphfig} for an illustration of the main results from Lemma~\ref{sumabslemma}.

\begin{figure}
\begin{picture}(250,100)
\put(-50,15){
\includegraphics[scale = .4, clip = true, draft = false]{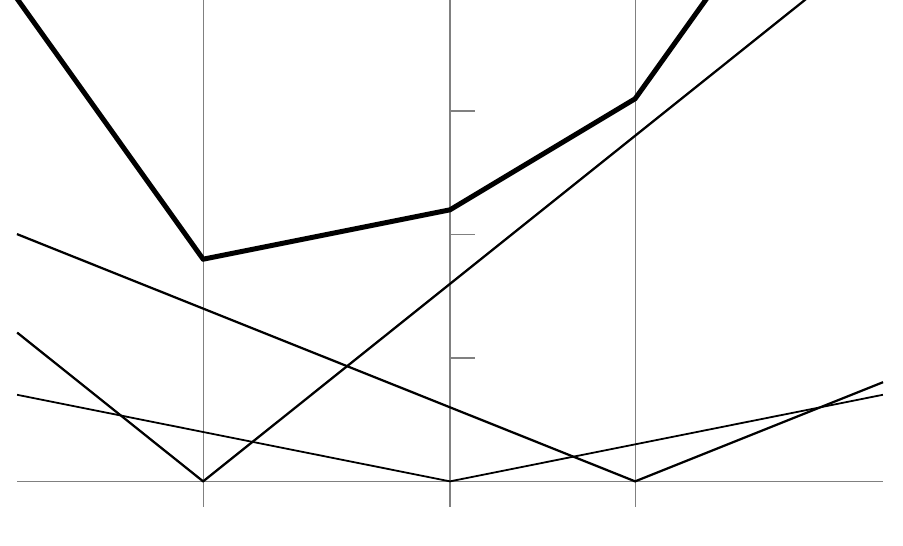}
}
\put(-28,13){$\scriptscriptstyle -2$}
\put(4,13){$\scriptscriptstyle 0$}
\put(23,13){$\scriptscriptstyle 1.5$}
\put(10,35){$\scriptscriptstyle 1$}
\put(65,15){
\includegraphics[scale = .4, clip = true, draft = false]{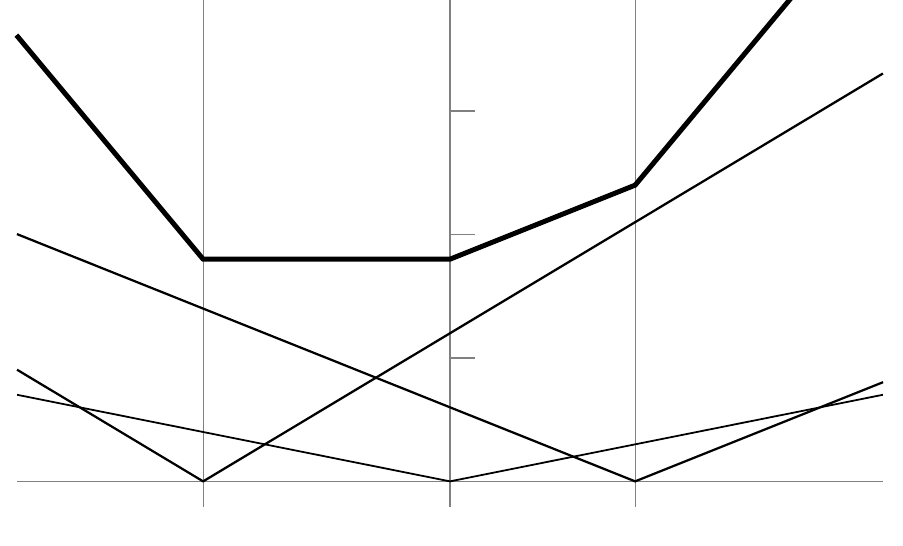}
}
\put(87,13){$\scriptscriptstyle -2$}
\put(119,13){$\scriptscriptstyle 0$}
\put(138,13){$\scriptscriptstyle 1.5$}
\put(125,35){$\scriptscriptstyle 1$}
\put(180,15){
\includegraphics[scale = .4, clip = true, draft = false]{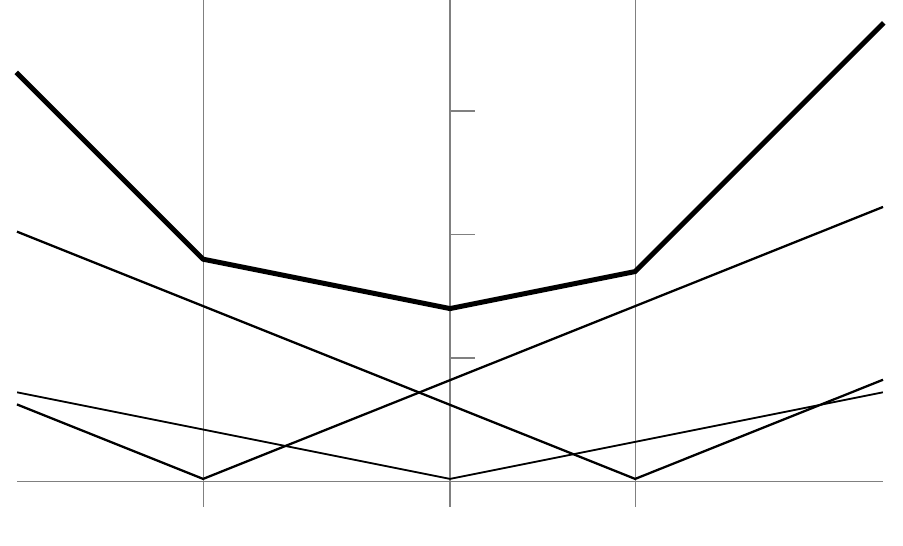}
}
\put(202,13){$\scriptscriptstyle -2$}
\put(234,13){$\scriptscriptstyle 0$}
\put(253,13){$\scriptscriptstyle 1.5$}
\put(240,49){$\scriptscriptstyle 2$}
\put(0,0){(a)}
\put(115,0){(b)}
\put(230,0){(c)}
\end{picture}
\caption{Graphs of
\[
f(t) = |b_1- m_1 t| + |b_2- m_2 t| + |b_3- m_3 t|
\]
where the darker function $f$ is the sum of the three lighter functions.  In all three cases,
\[
b_2 = 0, m_2 = .2, b_3 = .6, m_3 = .4.
\]
In (a), $b_1 = -1.6$ and $m_1 = .8$, $m_1$ dominates the other two slopes, and the global minimum occurs at $\frac{b_1}{m_1}$.  In (b), $b_1 = -1.2$ and $m_1 = .6$, $m_1$ transitionally dominates the other two slopes, and the global minimum is achieved along the interval $\left[\frac{b_1}{m_1}, \frac{b_2}{m_2}\right]$.  In (c), $b_1 = -.8$ and $m_1 = .4$, there is no dominant or transitionally dominant slope and the global minimum is achieved at the middle critical point.} \label{absgraphfig}
\end{figure}

\subsubsection{Distance between a point and a line}

Given a point $x$ and a line $\ell = \ell_a$,
\begin{align*}
d(x, \ell) &= \inf_{t \in \mathbb{R}}\, d(x, \ell(t)) \\
	&=  \inf_{t \in \mathbb{R}}\, |x_1-a_1t|+|x_2-a_2t|+|x_3-a_3t|.
\end{align*}
By Lemma~\ref{sumabslemma}, the infimum is realized at (at least) one of the values $t=\frac{x_1}{a_1}$, $t=\frac{x_2}{a_2}$, and $t=\frac{x_3}{a_3}$, and the particular minimizing value depends on whether or not one of the components $a_i$ dominates the others.
Geometrically, we can think about the distance to $\ell$ in a way that is similar to that for a plane.  Specifically, we can let a sphere $\sigma_r(x)$ grow until it touches $\ell$.  In this case, the first contact will involve a great circle of $\sigma_r(x)$.  As such, the point on $\ell$ that is closest to $x$ will always share at least one coordinate with $x$:
\[
d(x, \ell) = \min \{d_{1,2}(x, \ell), d_{1,3}(x, \ell), d_{2, 3}(x, \ell) \}.
\]

We seek a geometric way to determine which critical point minimizes the given function.   Let $P^i$ be the plane containing the $i$th coordinate axis and $\ell_a$:
\[
P^i = \{x \in \mathbb{R}^3: a_k x_j = a_j x_k\ \mathrm{where}\ i \neq j \neq k\}.
\]
Note that if $\ell$ is a coordinate line, the corresponding plane is not uniquely defined, and in fact we leave it undefined in this case.  Suppose all three components of $a$ are nonzero.  These three planes all contain $\ell_a$, and as such, they subdivide $\mathbb{R}^3$ into six wedge-shaped regions.  Let $W^i$ be the union of the two wedges, including the boundary planes, that do not have $P^i$ as part of their boundary.

\begin{lemma} \label{middlevaluelemma}
Suppose the components of $a$ are all nonzero.  The point $x$ is an element of $W^i$ if and only if $\frac{x_i}{a_i}$ is the middle value or a duplicated value in the set $\left\{\frac{x_1}{a_1}, \frac{x_2}{a_2}, \frac{x_3}{a_3} \right\}$.
\end{lemma}

\begin{proof}
First, since $P^i$ is defined by the equation $\frac{x_j}{a_j} = \frac{x_k}{a_k}$, duplicated values occur exactly on the planes $P^i$.  For all other points, if the lemma is true for one point in $W^i$, then it must be true for all points in $W^i$ since the inequalities defining the middle value cannot reverse without crossing one of the planes $P^j$.

Next, in order to avoid handling multiple cases depending on different $a$, consider the transformation $\varphi:\mathbb{R}^3 \rightarrow \mathbb{R}^3$, $y = \varphi(x)$ where $y_i = \frac{x_i}{a_i}$.  Note that $\frac{x_i}{a_i}$ is the middle value if and only if $y_i$ is the middle value.  Moreover the coordinate axes are preserved (although their orientation may be reversed) and $\varphi(\ell_a) = \ell_{(1,1,1)}$.  Hence, $\varphi(P^i) = Q^i$ where $Q^i$ is the plane containing the $y_i$-axis and $\ell_{(1,1,1)}$.

In this setting, it is easier to find representative points in each region to confirm the result.  To help visualize convenient points and the relationships among the various objects, restrict attention to the plane $R = \{y \in \mathbb{R}^3: y_1 + y_2 + y_3 = 1\}$.  Then confirm directly that points whose coordinates are permutations of $0,$ $\frac{1}{4}$, and $\frac{3}{4}$ establish the result.  For example, the point $\left( \frac{1}{4}, 0, \frac{3}{4} \right)$ lies in one of the wedges defined by $Q^2$ and $Q^3$, and the first coordinate is the middle value.  See Figure~\ref{disttolinefig}.
\end{proof}

\begin{figure}
\begin{picture}(250,230)
\put(0,0){
\includegraphics[scale = 1, clip = true, draft = false]{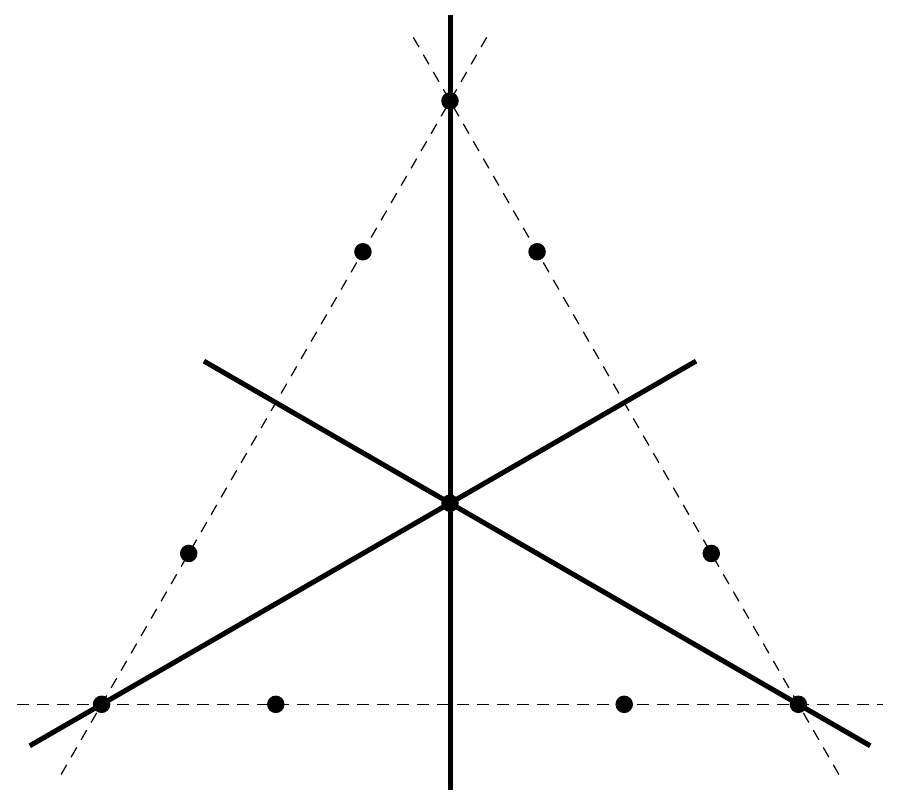}
}
\put(-4,36){\scriptsize $(1, 0, 0)$}
\put(235,36){\scriptsize $(0, 1, 0)$}
\put(141,201){\scriptsize $(0, 0, 1)$}
\put(145,84){\scriptsize $\ell$}
\put(208,130){\scriptsize $Q^1$}
\put(51,130){\scriptsize $Q^2$}
\put(130,-5){\scriptsize $Q^3$}
\put(110,130){\scriptsize $W^1$}
\put(159,47){\scriptsize $W^1$}
\put(145,130){\scriptsize $W^2$}
\put(101,47){\scriptsize $W^2$}
\put(172,80){\scriptsize $W^3$}
\put(83,80){\scriptsize $W^3$}
\put(67,13){\scriptsize $\left(\frac{3}{4}, \frac{1}{4}, 0 \right)$}
\put(167,13){\scriptsize $\left(\frac{1}{4}, \frac{3}{4}, 0 \right)$}
\put(214,75){\scriptsize $\left(0, \frac{3}{4}, \frac{1}{4} \right)$}
\put(164,161){\scriptsize $\left(0, \frac{1}{4}, \frac{3}{4} \right)$}
\put(19,75){\scriptsize $\left( \frac{3}{4}, 0, \frac{1}{4} \right)$}
\put(69,161){\scriptsize $\left( \frac{1}{4}, 0, \frac{3}{4} \right)$}
\end{picture}
\caption{The reference plane $R$ showing its intersection with various objects and the representative points in each wedge.  Note that if the coordinate with the value $\frac{1}{4}$ is in the $i$th position then that point is in the double-wedge  $W^i$.} \label{disttolinefig}
\end{figure}
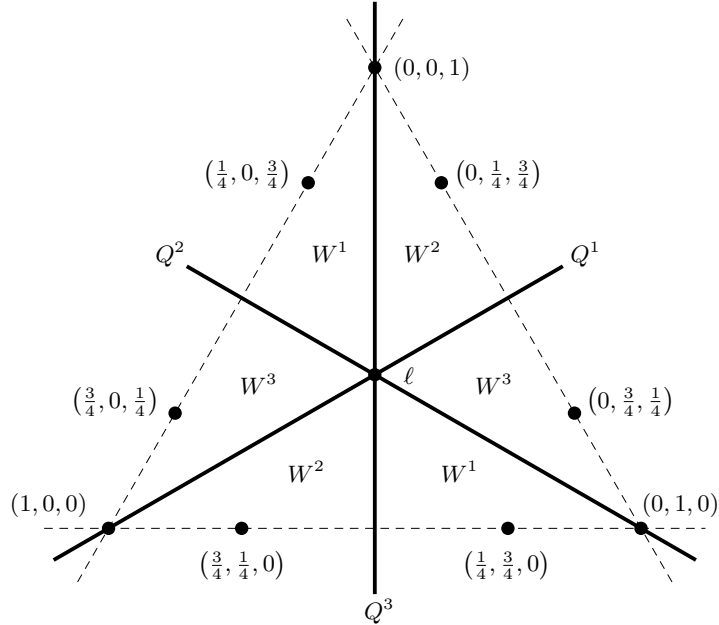

Now that we have a geometric way to identify the middle value in the set $\left\{\frac{x_1}{a_1}, \frac{x_2}{a_2}, \frac{x_3}{a_3} \right\}$, we can establish an explicit way to determine the distance between a point and a line.

\begin{thm} \label{disttolinethm}
Given a line $\ell_a$ and a point $x$,
\begin{itemize}
\item if there is a permutation $(i, j, k)$ of $(1, 2, 3),$ such that $a_i$ dominates or transitionally dominates $a_j$ and $a_k$ then $d(x,\ell)=d_{j,k}(x,\ell)$;
\item otherwise when $x$ lies in the double-wedge $W^i$, then $d(x,\ell) = d_{j,k}(x, \ell)$, where again $(i, j, k)$ is a permutation of $(1, 2, 3)$.
\end{itemize}
In either case,
\begin{equation} \label{disttolineeq}
d_{j,k}(x, \ell) =  d\left(x,\ell\left(\frac{x_i}{a_i}\right)\right)
			= \left| x_j-\frac{a_j}{a_i} x_i \right|+\left| x_k-\frac{a_k}{a_i} x_i \right|.
\end{equation}
\end{thm}

To clarify the subtlety in this result, note that in the first case $a$ determines which partial distance function to use independent of $x$, while in the second case $a$ determines the double-wedges and then with that structure in place, the location of $x$ determines which partial distance function to use.  Also note that if $x \in P^i$ then it lies in two double-wedges and both corresponding partial distance formulas apply, and if $x \in \ell$, all three partial distance formulas apply, and they all show that $d(x, \ell) = 0$.

\begin{proof}
Suppose first that there is a permutation $(i, j, k)$ of $(1, 2, 3),$ such that $a_i$ dominates or transitionally dominates the other two coordinates.  Then by Lemma~\ref{sumabslemma} the distance function $d(x, \ell(t))$ is minimized when $t = \frac{x_i}{a_i}$.

If there is no dominant or transitionally dominant coordinate of $a$ then none of the coordinates can be zero and $d(x, \ell)$ is determined by the middle value or duplicated value in  $\left\{\frac{x_1}{a_1}, \frac{x_2}{a_2}, \frac{x_3}{a_3} \right\}$, where in Lemma~\ref{sumabslemma}, the case of three distinct values corresponds to the special case where $N = 3$ and the case of duplicate values corresponds to $N = 2$ or, if it happens that $x$ is a multiple of $a$, then $N = 1$.  By Lemma~\ref{middlevaluelemma}, the middle or duplicate value is determined by the double-wedge in which $x$ resides.  If $x \in W^i$, the middle or duplicate value is $\frac{x_i}{a_i}$.

In either case, by plugging in directly we have
\begin{align*}
d(x, \ell) &= d\left( x,\ell \left(\frac{x_i}{a_i} \right) \right) \\
		&= \left| x_j-\frac{a_j}{a_i} x_i \right|+\left| x_k-\frac{a_k}{a_i} x_i \right|.
\end{align*}

\end{proof}

In light of Theorem~\ref{disttolinethm}, we say $\ell$ is \textit{steep} if $a_3$ is the dominant coordinate, we say $\ell$ is \textit{shallow} if $a_1$ or $a_2$ is dominant, we say $\ell$ is transitional if any coordinate is transitionally dominant, and we say $\ell$ is \textit{intermediate} if there is no dominant or transitionally dominant value in the set $\{a_1, a_2, a_3\}$.  Finally, we say $\ell$ is horizontal if $a_3 = 0$.

\subsubsection{Parameter space for defining lines}

As for planes, the set of lines can be thought of as $\mathbb{R}P^2$ and in light of Theorem~\ref{disttolinethm} an ideal parameter space for lines is the cubeoctahedron
\[
\mathscr{L}={\mbox{\large $\partial$}} \bigl\{a \in \mathbb{R}^3: \max\{|a_1|, |a_2|, |a_3|\} \leq 1\ \mbox{and}\ d(a, 0) \leq 2\bigr\}.
\]

In $\mathscr{L}$, the square faces are characterized by those points where one coordinate of $a$ is equal to $\pm1$ and dominates the other two coordinates.  The edges and vertices correspond to parameters that have a transitionally dominant coordinate.  The triangular faces are characterized by those points where there is no dominant coordinate.

In this setting, the steep lines correspond to the top and bottom square faces, the shallow lines correspond to the other square faces, the intermediate lines correspond to the triangular faces, and the transitional lines correspond to the edges and vertices

As with the parameter space for a plane, the cuboctahedron double-counts all lines since multiplying the parameter vector by $-1$ produces the same line.

\subsubsection{An alternative parameter space}
If a defining line $\ell_{\widetilde{a}}$ is not horizontal, we can choose an alternative parameterization by dividing the parameter vector by $\widetilde{a}_3$.  As with the situation for the plane, this is similar to a gnomonic projection.  Our map is $l: \mathscr{L} \cap \{x \in \mathbb{R}^3: x_3 > 0\} \rightarrow \mathbb{R}^3$, $(a_1, a_2, a_3) = l(\widetilde{a}) = \left(\frac{\widetilde{a}_1}{\widetilde{a}_3}, \frac{\widetilde{a}_2}{\widetilde{a}_3}, 1 \right)$.  This maps the top square of $\mathscr{L}$ to a square, the upper four triangles of $\mathscr{L}$ to semi-infinite strips, and the upper halves of the vertical squares of $\mathscr{L}$ to quadrants.  See Figure~\ref{lineparameterspacefig}.  As with the alternate parameter space for planes, this space misses horizontal lines.  Again, these are represented using a circle at infinity and we have
\[
\mathscr{L}' = \{ a \in \mathbb{R}^3: a_3 = 1 \}
	\cup \{ a \in \mathbb{R}^3 \backslash \{0\}: a_3 = 0 \}.
\]
In this setting $\ell$ is steep if $a$ lies in the central square, shallow if $a$ lies in one of the four quadrants, intermediate if $a$ lies in a semi-infinite strip, and transitional if $a$ lies in a boundary between these regions.

\begin{figure}
\begin{picture}(250,130)
\put(-70,-35){
\includegraphics[scale = .7, clip = true, draft = false]{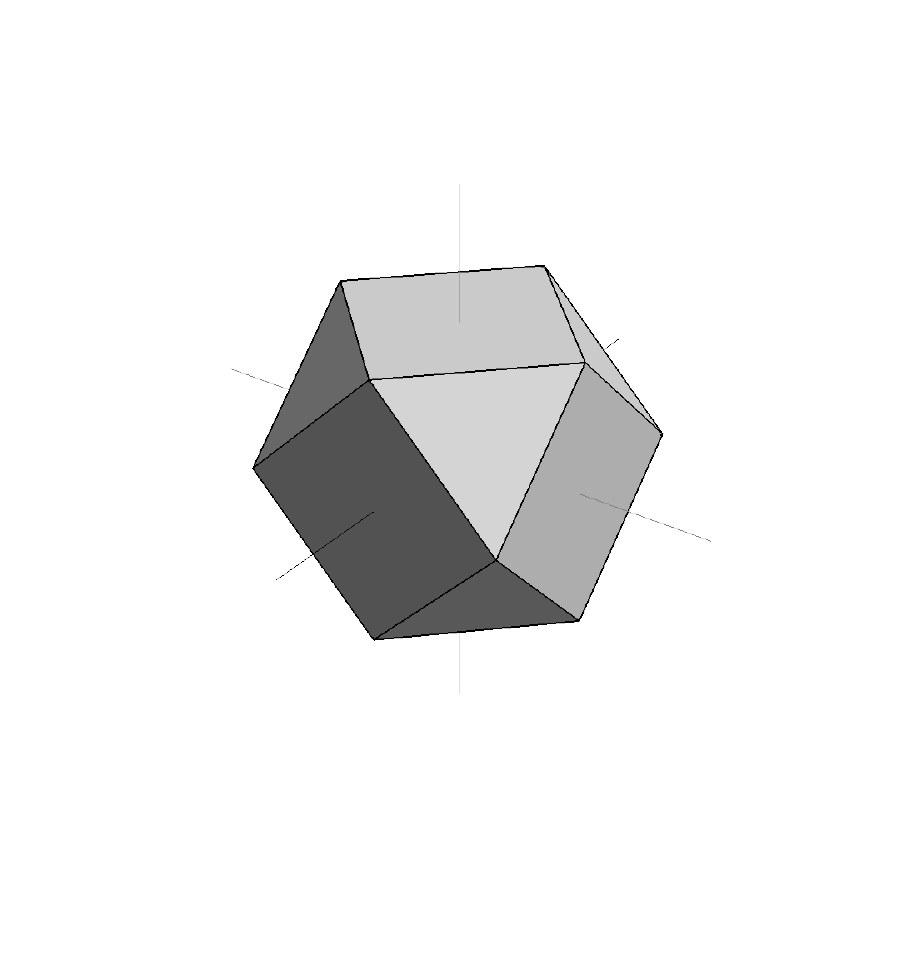}
}
\put(130,0){
\includegraphics[scale = .5, clip = true, draft = false]{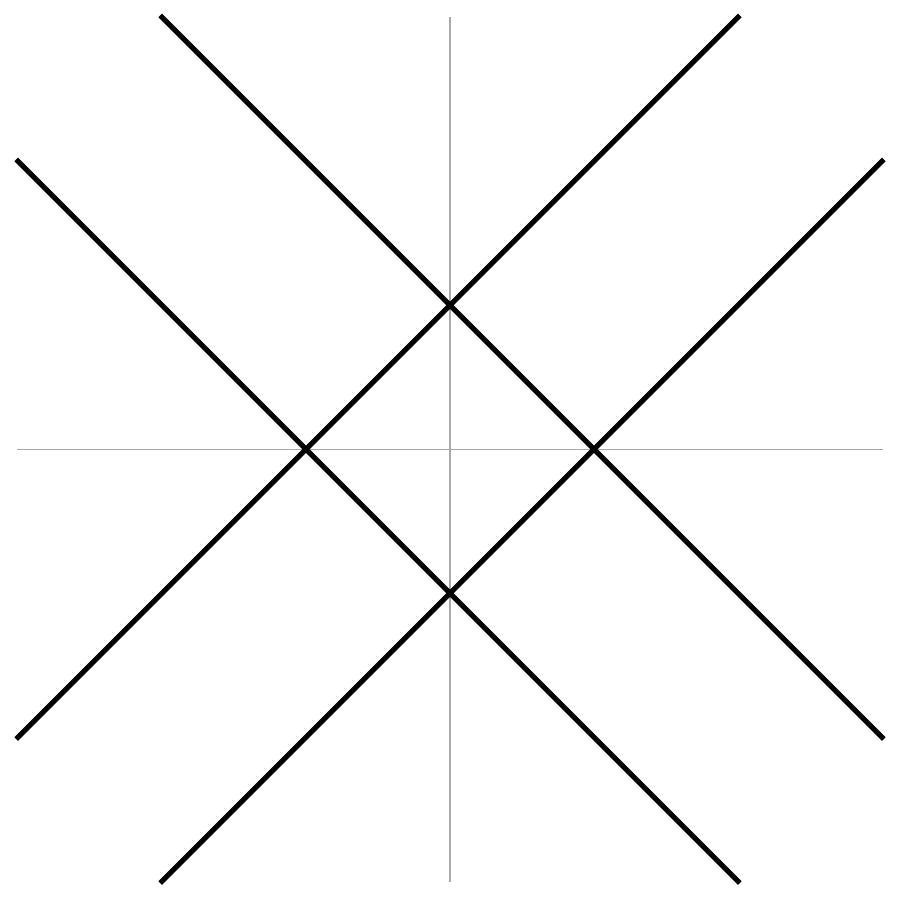}
}
\put(-21,39){$\widetilde{a}_1$}
\put(78,46){$\widetilde{a}_2$}
\put(30,120){$\widetilde{a}_3$}
\put(265,62){$a_1$}
\put(205,125){$a_2$}
\put(20,-7){$\mathscr{L}$}
\put(192,-7){$\mathscr{L}'$}
\end{picture}
\caption{Two parameter spaces for lines.  The cuboctahedron $\mathscr{L}$ is a natural choice because of the way distance between a point and a line is measured.  The planar parameter space $\mathscr{L}'$ is a good choice in the context of conic sections because of how we will deal with the slicing plane.  Note that for $\mathscr{L}'$ the circle at infinity is not shown.} \label{lineparameterspacefig}
\end{figure}
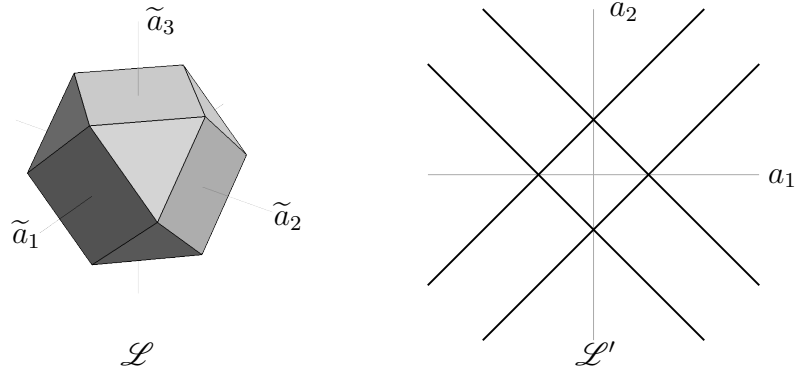

\subsection{The slicing plane revisited}

As discussed in Section~\ref{slicingplanesec}, we will use $S = \{x \in \mathbb{R}^3: x_3 = 1\}$ as our slicing plane.  This plane will serve a number of distinct but related purposes.

First, and most importantly, it will be the plane that contains the conic section resulting from a given cone.  Moreover, since $S$ is a coordinate plane, and thus inherits the taxicab metric, the third coordinate can be dropped and the various sets of interest can be interpreted as lying in usual 2-dimensional taxicab space.

Second, as long as the defining line $\ell$ is not horizontal, it will intersect $S$ and the point $a = \ell \cap S = (a_1, a_2, 1)$ can be interpreted as the parameter for $\ell$ from $\mathscr{L}'$.  As such, we will identify $S$ with the set of non-horizontal parameters in $\mathscr{L}'$.

Third, as long as the defining plane $P$ is not horizontal, it will intersect $S$ along a line.  If $A \in \mathscr{P}'$ defines $P$, then this line is
\[
P^S = P \cap S = \{A_1 x_1 + A_2 x_2 + \delta = 0\}.
\]
As above, we could identify the set of non-vertical parameters in $\mathscr{P}'$ with $S$, but we find that $P^S$ will be a more geometrically useful way to identify $P$.  An example of this utility is that a cone is degenerate if and only if $a \in P^S$.  Observe that $P$ is steep if $P^S$ intersects the open taxicab disk defined by $\sigma_1(0)$, it is transitional if $P^S$ intersects $\sigma_1(0)$ but not the open disk, and it is shallow otherwise.  This is convenient since $\sigma_1(0)$ also plays a role in $\mathscr{L}'$.

Fourth, as long as $P^i$ is defined and not horizontal, it will intersect $S$ along a line $\rho^i$ which we call a reference line.  Note that $\rho^1 = \{x \in S: x_2 = a_2\}$ and $\rho^2 = \{x \in S: x_1 = a_1\}$ are the coordinate lines through $a$ and $\rho^3 = \{x \in S: a_1 x_2 = a_2 x_1\}$ is the line through $(0,0,1)$ and $a$.

These lines are useful since they identify transitions in the formula for the distance between a point and a line.  In this role, they serve two subtly different purposes.  First, when $\ell$ is intermediate, the double-wedges defined by $\ell$ determine which partial distance to use to measure the distance between a point and $\ell$.  Since the reference lines $\rho^i$ are the boundaries of these double-wedges restricted to $S$, they identify transitions in which partial distance formula is to be used.  In this case, we say all three reference lines are active.  Second, if $\ell$ is not intermediate, the distance between a point and $\ell$ is determined by a single partial distance, and two of the reference lines indicate transitions in how the absolute value expressions in the partial distance resolve.  When $d_{j,k}(x, \ell)$ is being used, the reference lines $\rho^j$ and $\rho^k$ indicate the transitions and we say these two lines are active, while the third is inactive.

Finally, when $\ell$ is horizontal, the point $a$ does not lie in $S$, nor do $\rho^1$ or $\rho^2$.  However, $\rho^3$ is defined, and serves as a representative for the point at infinity that parameterizes $\ell$.

\subsubsection{Wedges revisited} \label{wedgesec}
While double-wedges $W^i$ are defined whenever all three $a_i$ are non-zero, they are most useful when $\ell$ is intermediate.  If $\ell$ is not intermediate and only two reference lines are active, the corresponding planes divide $\mathbb{R}^3$ into four wedges similar to the six wedges described previously.  The six wedges in the intermediate case and the four wedges in the non-intermediate case will both aid in the construction of taxicab conics.  As such, moving forward it should be understood that, when wedges are discussed, the ones being used depend on whether or not $\ell$ is intermediate and we indicate this by referring to the wedges as active.  The intersections of these 3-dimensional wedges and $S$ are 2-dimensional wedges separated by active reference lines.

See Figure~\ref{usesofslicingplanefig} for the various ways that information can be encoded in the slicing plane $S$.

\begin{figure}
\begin{picture}(215,210)
\put(0,0){
\includegraphics[scale = .8, draft = false]{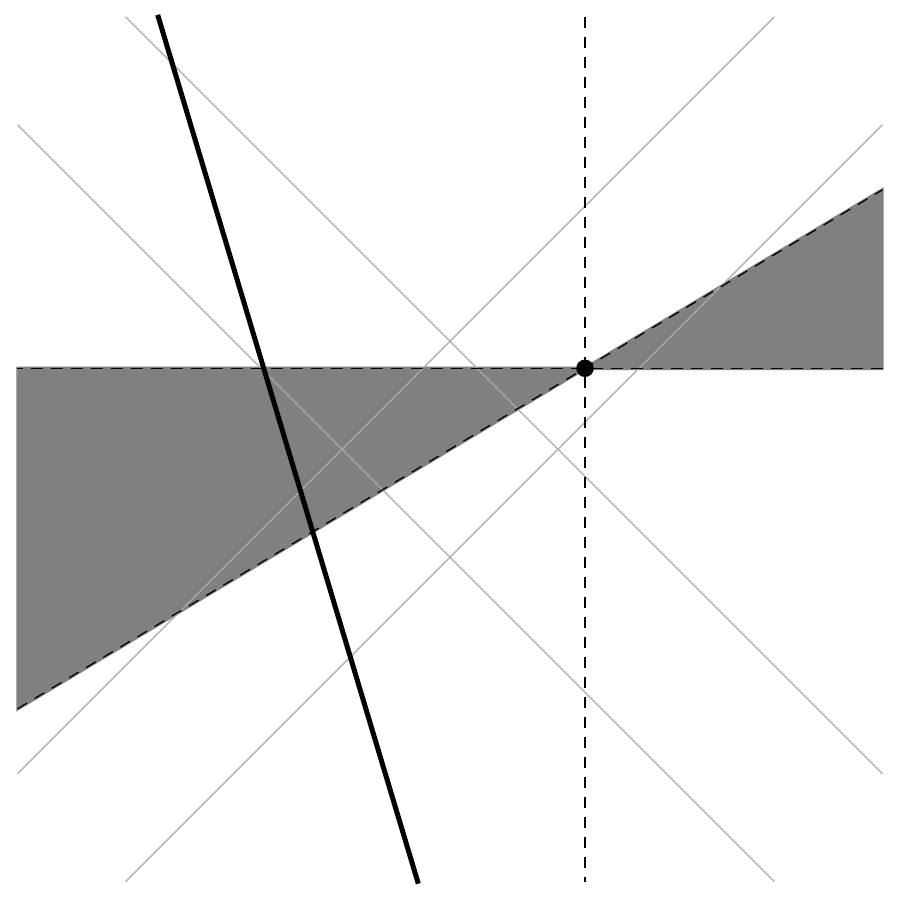}
}
\put(130,125){$a$}
\put(103,0){$P^S$}
\put(210,120){$\rho^1$}
\put(136,209){$\rho^2$}
\put(212,165){$\rho^3$}
\end{picture}
\caption{The slicing plane $S$ and its intersections with $\ell$ at $a$, $P$ at $P^S$,  the planes $P^i$ at the reference lines $\rho^i$, and the double-wedges $W^i$ ($W^2$ shaded).  Using the identification of $\mathscr{L}'$ with $S$, the gray lines indicate transitionally dominant parameters.  From this we can see that $\ell$ is intermediate since $a$ lies in a semi-infinite strip, and $P$ is shallow since $P^S$ does not intersect the central taxicab circle $\sigma_1(0)$.} \label{usesofslicingplanefig}
\end{figure}
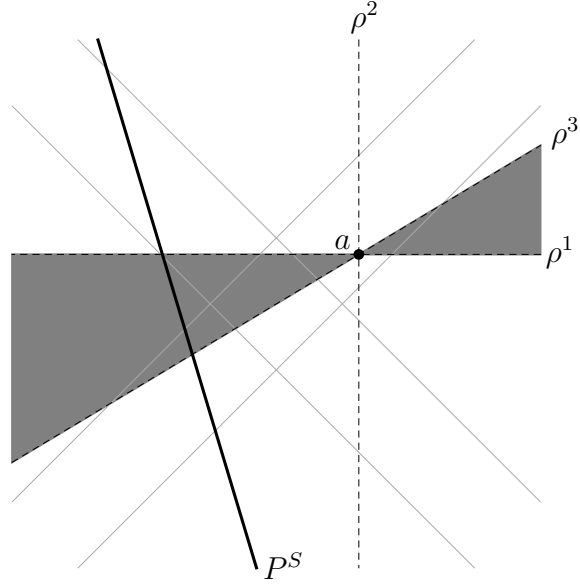

\section{Conic sections when $\ell$ is non-horizontal} \label{conicsectionswhenellisnonhorizontalsection}

With the structure developed above, we are now ready to understand and characterize conic sections.  We explore the case where $\ell$ is not horizontal here.  The case where $\ell$ is horizontal requires slightly different analysis and is covered in the next section.  Using $A \in \mathscr{P}'$, and $a \in \mathscr{L}'$, we begin by finding the vertices of the conic section, which will lie on the active reference lines determined by $\ell_a$.  Once these vertices are known, the conic section will consist of line segments or rays connecting these vertices.

\subsection{Vertices of the sections}

In all cases, we are analyzing the equation $d(x, \ell) = \kappa\, d(x, P)$, where $x \in S$.
For the distance to the plane, from Equation~\eqref{disttoplaneeq} we have
\[
d(x, P) = \frac{|A_1 x_1 + A_2 x_2 + \delta|}{M}
\]
where $\delta = 1$ if $P$ is not vertical and $0$ if $P$ is vertical, and $M = \max \{|A_1|, |A_2|, \delta\}$.

For the distance to the line, from Equation~\eqref{disttolineeq} we have
\begin{align*}
d_{1, 2}(x, \ell) &= |x_1 - a_1| + |x_2 - a_2|, \\
d_{1, 3}(x, \ell) &=\left|x_1 - \frac{a_1}{a_2} x_2 \right| + \left| 1 - \frac{1}{a_2} x_2 \right|, \\
d_{2, 3}(x, \ell) &=\left|x_2 - \frac{a_2}{a_1} x_1 \right| + \left| 1 - \frac{1}{a_1} x_1 \right|
\end{align*}

where, by virtue of Theorem~\ref{disttolinethm}, the choice of which partial distance to use depends on the location of $a$ and $x$.

Measuring the distance between a point on a reference line and $\ell$ reduces to a single absolute difference.  We can see this in two ways.  First, as mentioned before, these reference lines identify transitions in either the partial distance being used, or the way in which the absolute value expressions resolve.  Second, and more geometrically, the points $x$ in these reference lines correspond to those points in $S$ where the sphere measuring the distance form $x$ to $\ell$ intersects $\ell$ at a vertex.  Because of these facts, determining the points in the conic section on the reference lines is relatively simple and these points are important because they are the vertices of the conic section.

\begin{thm} \label{nonhorizontalverticesthm}
Given a cone $C(\ell_a, P_A, \kappa)$ where $a = (a_1, a_2, 1)$ and $A = (A_1, A_2, \delta)$, and slicing plane $S = \{x \in \mathbb{R}^3 :x_3 = 1\}$,
\begin{itemize}
\item if the reference line $\rho^1 = \{x \in S:  x_2 = a_2\}$ is active then the vertices lying on $\rho^1$ are
 \begin{equation} \label{v1formulaeqn}
 v^{1\pm} = \left(a_1 + \frac{A_1 a_1 + A_2 a_2 + \delta}{\pm \frac{M}{\kappa} - A_1}, a_2 , 1\right);
 \end{equation}
\item if the reference line $\rho^2 = \{x \in S:  x_1 = a_1\}$ is active then the vertices lying on $\rho^2$ are
 \begin{equation} \label{v2formulaeqn}
 v^{2\pm} = \left(a_1, a_2 + \frac{A_1 a_1 + A_2 a_2 + \delta}{\pm \frac{M}{\kappa} - A_2}, 1 \right);
 \end{equation}
 \item if the reference line $\rho^3 = \{x \in S: a_1 x_2 = a_2 x_1\}$ is active then the vertices lying on $\rho^3$ are
 \begin{equation} \label{v3formulaeqn}
v^{3 \pm} = (r^{\pm} a_1, r^{\pm} a_2, 1)
 \end{equation}
where
\[
r^{\pm} = 1 + \frac{A_1 a_1 + A_2 a_2 + \delta}{\pm \frac{M}{\kappa} - (A_1 a_1 + A_2 a_2)}.
\]
\end{itemize}
In each of the above, the sign choice in the superscript corresponds to the choice in the formula for that vertex.
\end{thm}

\begin{proof}
If the reference line $\rho^1$ is active, then on this line $d(x, \ell) = |x_1 - a_1|$ so the points in the conic section on $\rho^1$ are the solutions to
\[
 |x_1 - a_1| =  \kappa\, \frac{|A_1 x_1 + A_2 a_2 + \delta|}{M}.
 \]
 Resolving the absolute values results in two different equations, and the two solutions are
 \begin{align*}
 x_1 &= \frac{ \pm \frac{M}{\kappa} a_1 + A_2 a_2 + \delta}{\pm \frac{M}{\kappa} - A_1} \\
 	&= a_1 + \frac{A_1 a_1 + A_2 a_2 + \delta}{\pm \frac{M}{\kappa} - A_1}.
 \end{align*}
The calculation for the reference line $\rho^2$ is similar.

If the reference line $\rho^3$ is active, consider two cases.  First, if
$|a_1| \geq |a_2|$ then on $\rho^3$, $d(x, \ell) = d_{2,3}(x, \ell) = \left|1 - \frac{1}{a_1} x_1 \right|$ so the points in the conic section on this line are the solutions to
\[
 \left| 1 - \frac{1}{a_1} x_1 \right| =  \kappa \, \frac{\left|A_1 x_1 + A_2 \frac{a_2}{a_1}x_1 + \delta \right|}{M}.
 \]
Resolving the absolute values results in two different equations for $x_1$.  Once $x_1$ is found, $x_2$ can be computed resulting in
\begin{align*}
x_i &= \frac{\delta \pm \frac{M}{\kappa}}{\pm \frac{M}{\kappa} - A_1 a_1 - A_2 a_2}a_i \\
	&= \left(1 + \frac{A_1 a_1 + A_2 a_2 + \delta}{\pm \frac{M}{\kappa} - (A_1 a_1 + A_2 a_2)} \right) a_i.
\end{align*}

For the second case, note that this formula is symmetric in $a_1$ and $a_2$, so if $|a_2| \geq |a_1|$, the calculation analogous to that above using $d(x, \ell) = d_{1,3}(x, \ell) = \left|1 - \frac{1}{a_2} x_2 \right|$ results in the same formula.
\end{proof}

Note that the formulas above define points whether or not a reference line is active.  If a reference line is inactive, the resulting distance formula for $\ell$ does not experience a transition along this line and the corresponding points do not lie on the conic section.   Also note that it is possible for a vertex to be undefined due to a zero in the denominator and this significantly impacts the geometric nature of the resulting conic section.  In this instance, we say the vertex lies at infinity.

Equations~\eqref{v1formulaeqn}, \eqref{v2formulaeqn}, and \eqref{v3formulaeqn} can be rewritten in terms of deviations from $a$ as follows:
\begin{align}
v^{1 \pm} &= a + \left(\frac{A_1 a_1 + A_2 a_2 + \delta}{\pm \frac{M}{\kappa} - A_1}, 0, 0 \right),
			\label{v1altformulaeq} \\
v^{2 \pm} &= a + \left(0, \frac{A_1 a_1 + A_2 a_2 + \delta}{\pm \frac{M}{\kappa} - A_2}, 0 \right),
			\label{v2altformulaeq} \\
v^{3 \pm} &= a + \tilde{r}^{\pm} (a_1, a_2, 0) \label{v3altformulaeq}
\end{align}
where
\[
\tilde{r}^{\pm} = \frac{A_1 a_1 + A_2 a_2 + \delta}{\pm \frac{M}{\kappa} - (A_1 a_1 + A_2 a_2)}.
\]
These formulas will help with analysis that establish characteristics of the conic sections.

\subsection{Constructing the sections:  connecting the dots}

With the vertices known, the conic sections can be constructed by connecting certain vertices by straight lines.  This requires some care and in an attempt to simplify and clarify the statement of the next theorem, we first establish some terminology.

As discussed in Section~\ref{wedgesec}, the intersection of an active double-wedge and $S$ is defined by two active reference lines and consists of two 2-dimensional wedges.  The boundaries of these wedges are rays.  We say two such rays $\gamma$ and $\lambda$ are adjacent if together they are the boundary of an active wedge.  We say two rays $\gamma$ and $\eta$ are anti-adjacent if $\eta$ and $\lambda$ form a line and $\gamma$ and $\lambda$ are adjacent.  See Figure~\ref{adjacencyfig}(a).  Note that if there are only two active reference lines, then a pair of adjacent rays are also anti-adjacent.

\begin{figure}
\begin{picture}(300,90)
\put(-30,12){
\includegraphics[scale = .4, clip = true, draft = false]{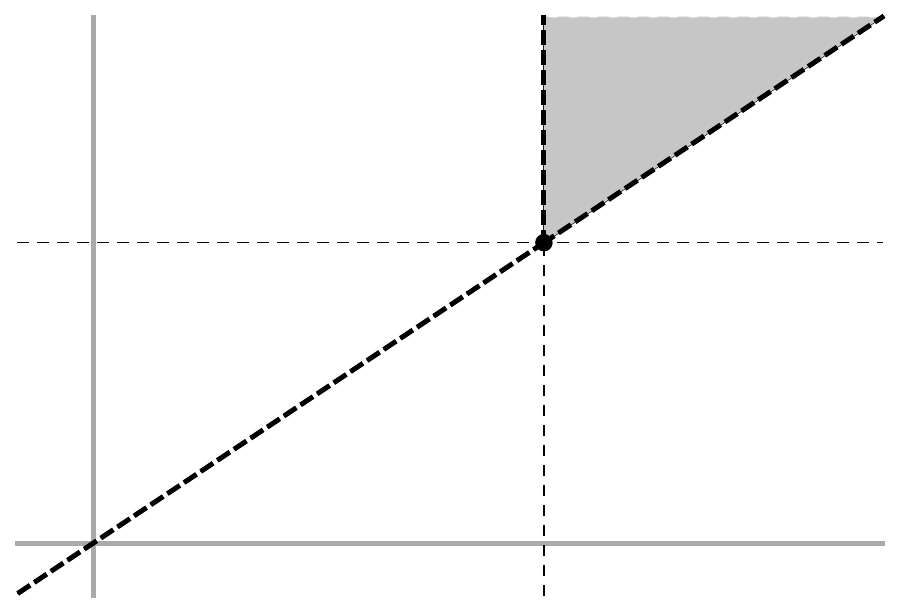}
}
\put(20,0){(a)}
\put(90,10){
\includegraphics[scale = .4, clip = true, draft = false]{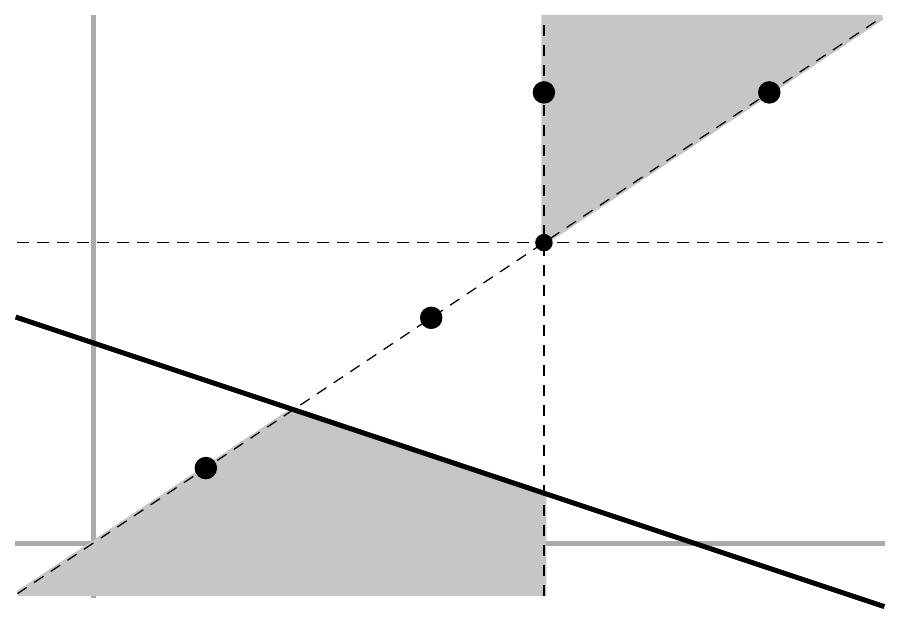}
}
\put(140,0){(b)}
\put(210,8){
\includegraphics[scale = .4, clip = true, draft = false]{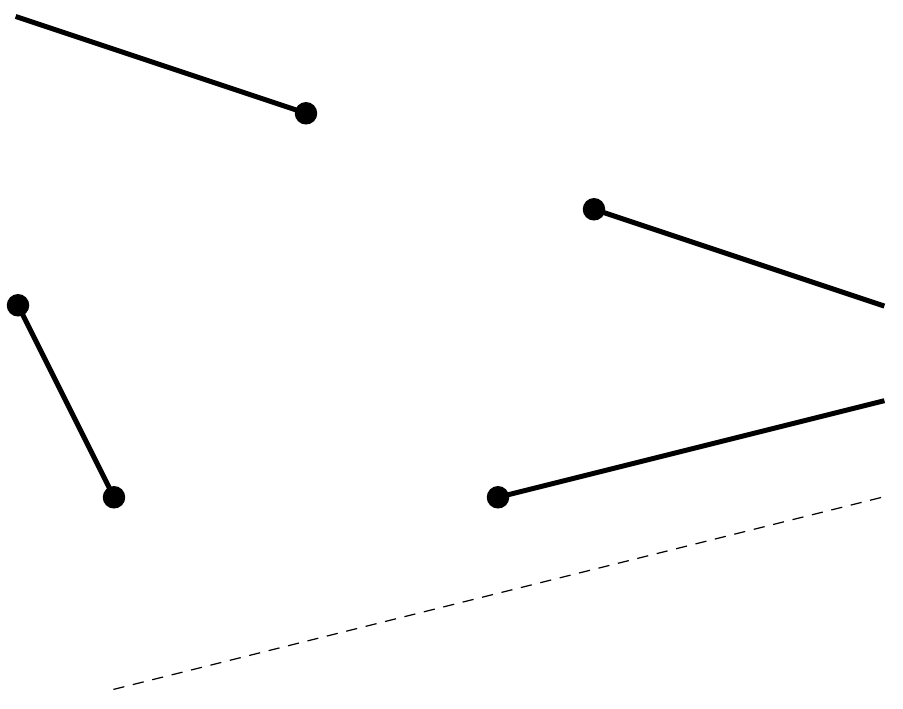}
}
\put(260,0){(c)}

\put(39,48){\scriptsize $a$}
\put(30,65){\scriptsize $\gamma$}
\put(60,63){\scriptsize $\lambda$}
\put(17, 35){\scriptsize $\eta$}
\put(160, 48){\scriptsize $a$}
\put(90, 36){$\scriptscriptstyle P^S$}
\put(150, 70){\scriptsize $p$}
\put(182,65){\scriptsize $q$}
\put(112,29){\scriptsize $r$}
\put(137,45){\scriptsize $s$}
\put(210,55)
{
\put(0,2){\scriptsize $p$}
\put(20,-26){\scriptsize $q$}
\put(42,22){\scriptsize $r$}
\put(68,14){\scriptsize $s$}
\put(52,-24){\scriptsize $v$}
\put(75,-35){\scriptsize $\rho$}
}
\end{picture}
\caption{In (a), the rays $\gamma$ and $\lambda$ are adjacent, while $\gamma$ and $\eta$ are anti-adjacent.  In (b), the point $p$ is adjacent to $q$ and anti-adjacent to $r$.  If $\rho^1$ is active, then $p$ is neither adjacent nor anti-adjacent to $s$.  If $\rho^1$ is not active then $p$ and $s$ are adjacent.  In (c), the segment associated to $p$ and $q$, the complementary rays associated to $r$ and $s$, and one of the rays associated to $v$ and $\rho$.} \label{adjacencyfig}
\end{figure}
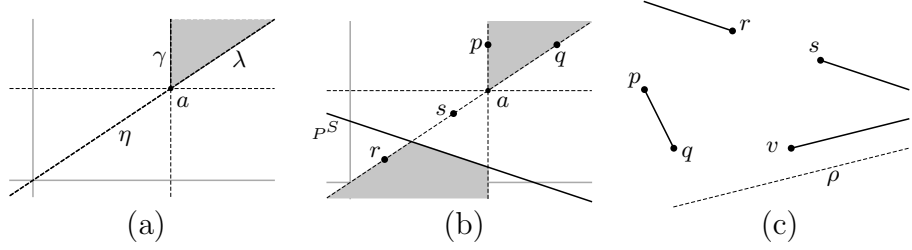

We say two vertices are adjacent if they are on adjacent rays and they are on the same side of $P^S$.  We say two vertices are anti-adjacent if they are on anti-adjacent rays and are on opposite sides of $P^S$.  See Figure~\ref{adjacencyfig}(b).  If a vertex lies at infinity, it is interpreted as lying on both sides of $P^S$.

Let $\gamma$ be the line through two points $v$ and $w$.  We say the line segment with endpoints $v$ and $w$ is the segment associated to $v$ and $w$.  We say the points in $\gamma$ that are not part of the segment are the complementary rays associated to $v$ and $w$.  Given a point $v$ and a line $\rho$, we say the two rays based at $v$ and parallel to $\rho$ are the rays associated to $v$ and $\rho$.  See Figure~\ref{adjacencyfig}(c).

To more easily discuss the vertices and the lines they define, let $s_i \in \{+, -\}$ so that $v^{i s_i}$ is one of the two vertices on $\rho^i$.  Also let $\gamma^{i s_i j s_j}$ be the line defined by $v^{i s_i}$ and $v^{j s_j}$.

\begin{thm} \label{sectionconstructionthm}
Given a cone $C = C(\ell, P, \kappa)$ the conic section $C \cap S$ consists of the vertices on the active reference lines, the line segments associated to adjacent vertices, the complementary rays associated to anti-adjacent vertices, and, if a vertex $v$ is adjacent to a vertex at infinity on reference line $\rho$, the ray associated to $v$ and $\rho$ that does not cross $P^S$.
\end{thm}

See Figure~\ref{connectthedotsfig} for an illustration of this theorem.

\begin{figure}
\begin{picture}(300,370)
\put(90,270){
\includegraphics[scale = .4, clip = true, draft = false]{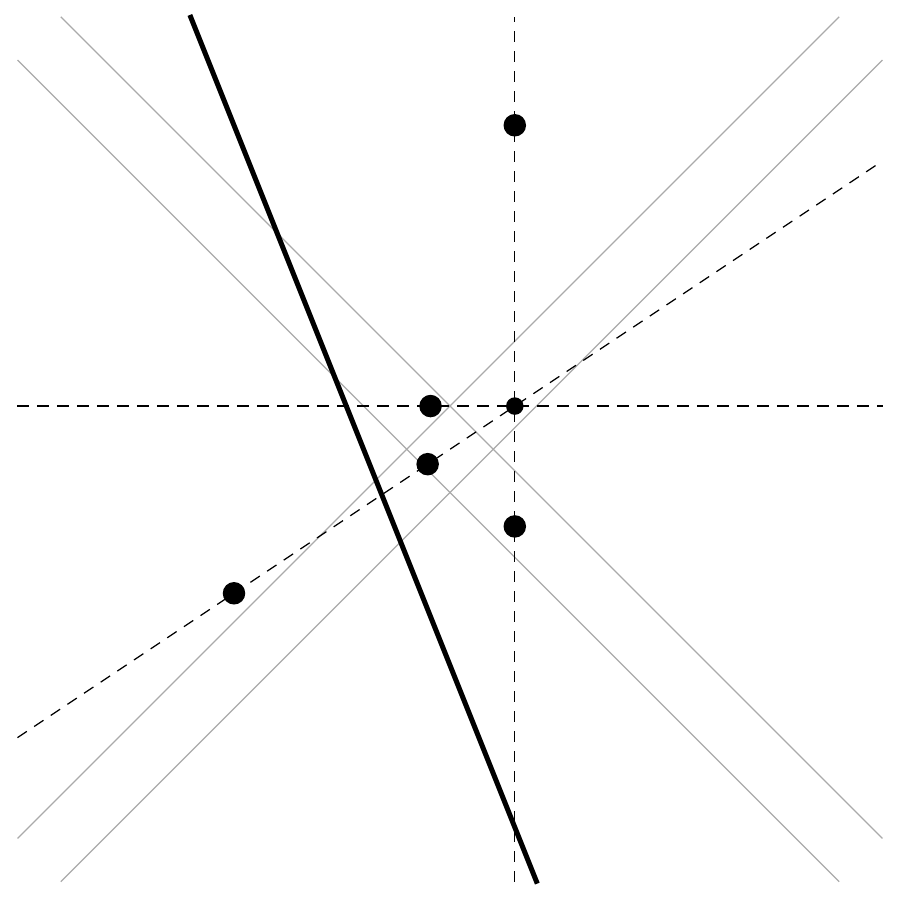}
}
\put(146,260){(a)}
\put(-30,140){
\includegraphics[scale = .4, clip = true, draft = false]{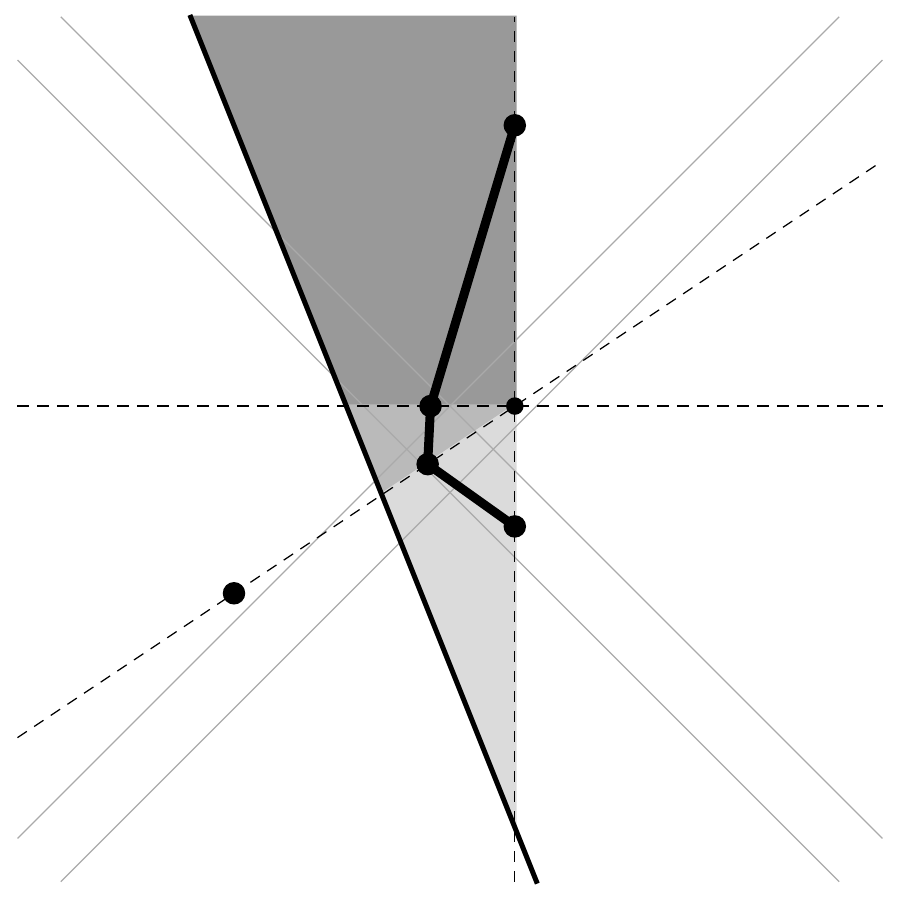}
}
\put(26,130){(b)}
\put(90,140){
\includegraphics[scale = .4, clip = true, draft = false]{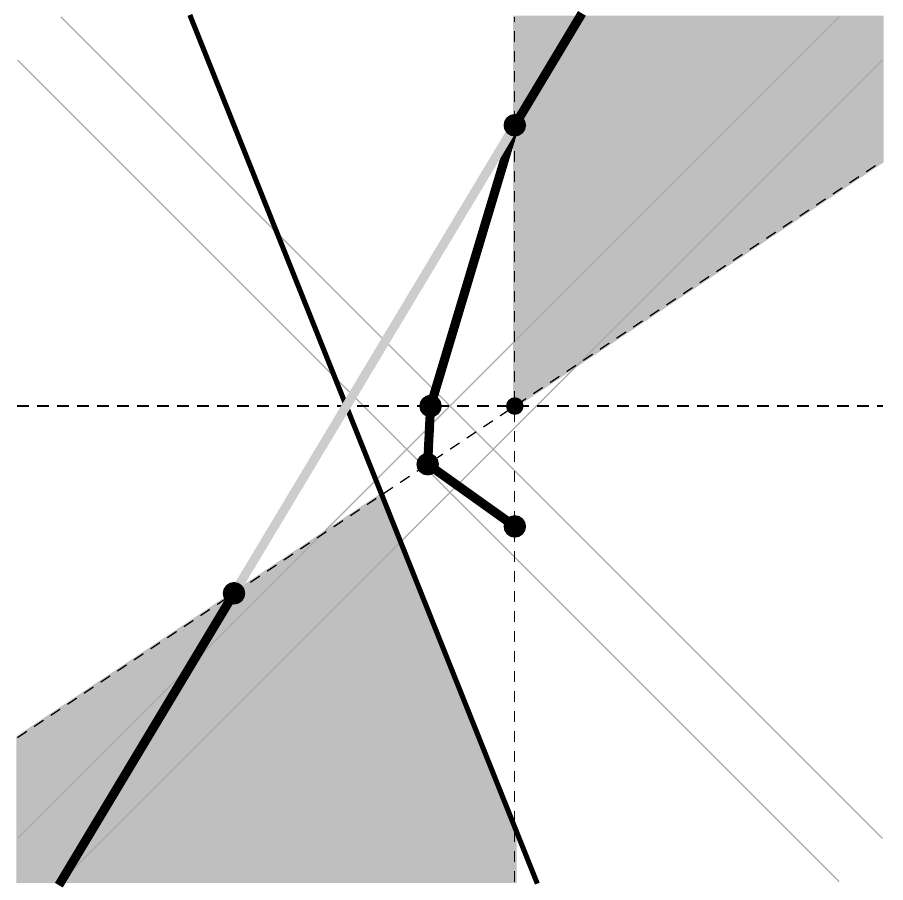}
}
\put(146,130){(c)}
\put(210,140){
\includegraphics[scale = .4, clip = true, draft = false]{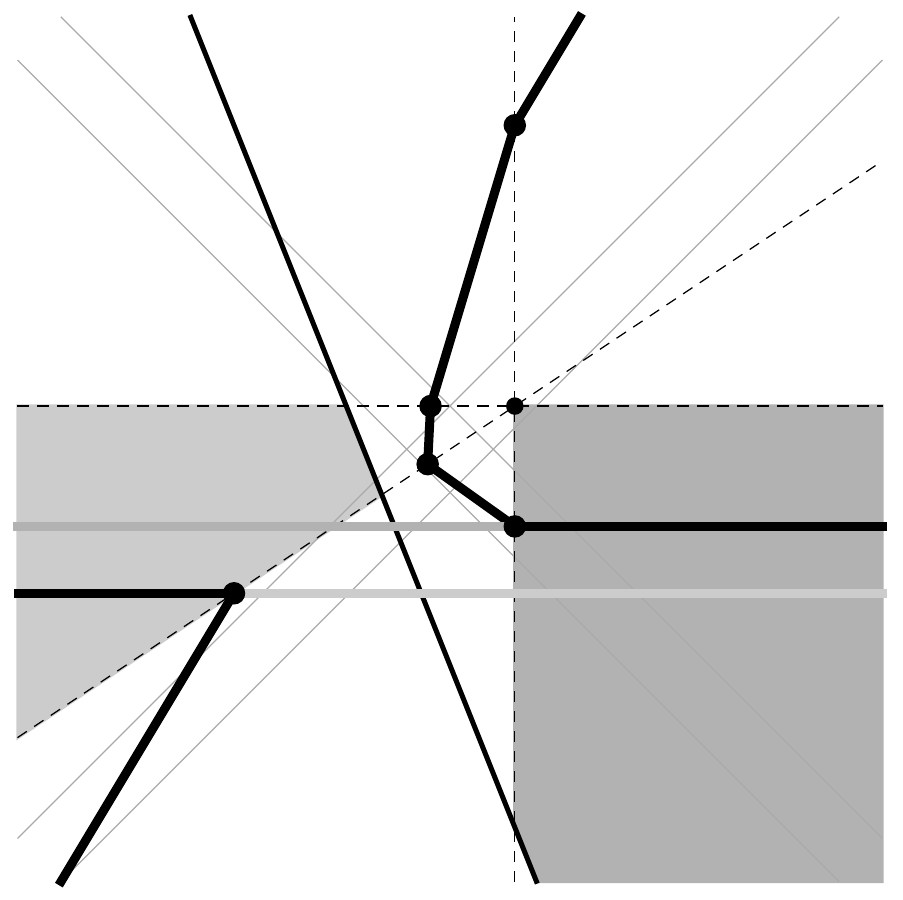}
}
\put(266,130){(d)}
\put(90,10){
\includegraphics[scale = .4, clip = true, draft = false]{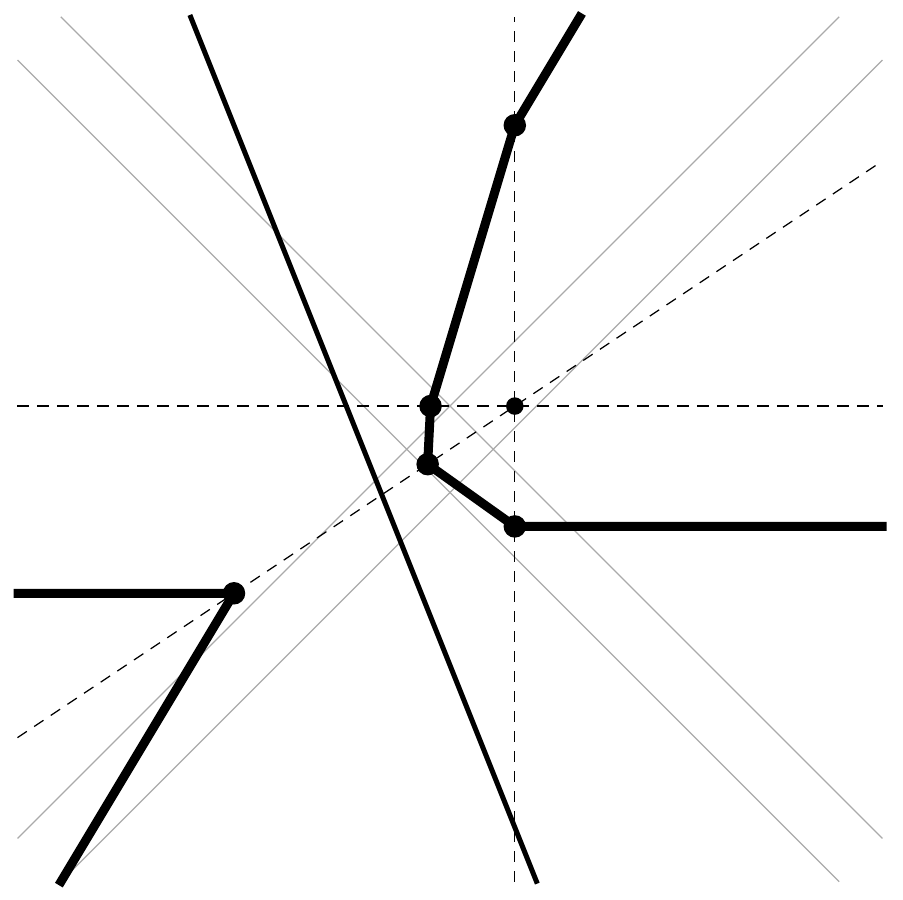}
}
\put(146,0){(e)}
\end{picture}
\caption{Connecting the dots to produce a conic section.  In this example, $A = \left(\frac{1}{2}, \frac{1}{5}, 1\right)$, $a = \left(\frac{3}{2}, 1, 1 \right)$, and $\kappa = 2$.  Starting with the vertices (a), connect adjacent vertices using their associated segments (b), then add complementing rays associated to anti-adjacent vertices (c), and finally, if there are vertices at infinity, add rays associated to vertices and reference lines (d) to produce the complete conic section (e).} \label{connectthedotsfig}
\end{figure}
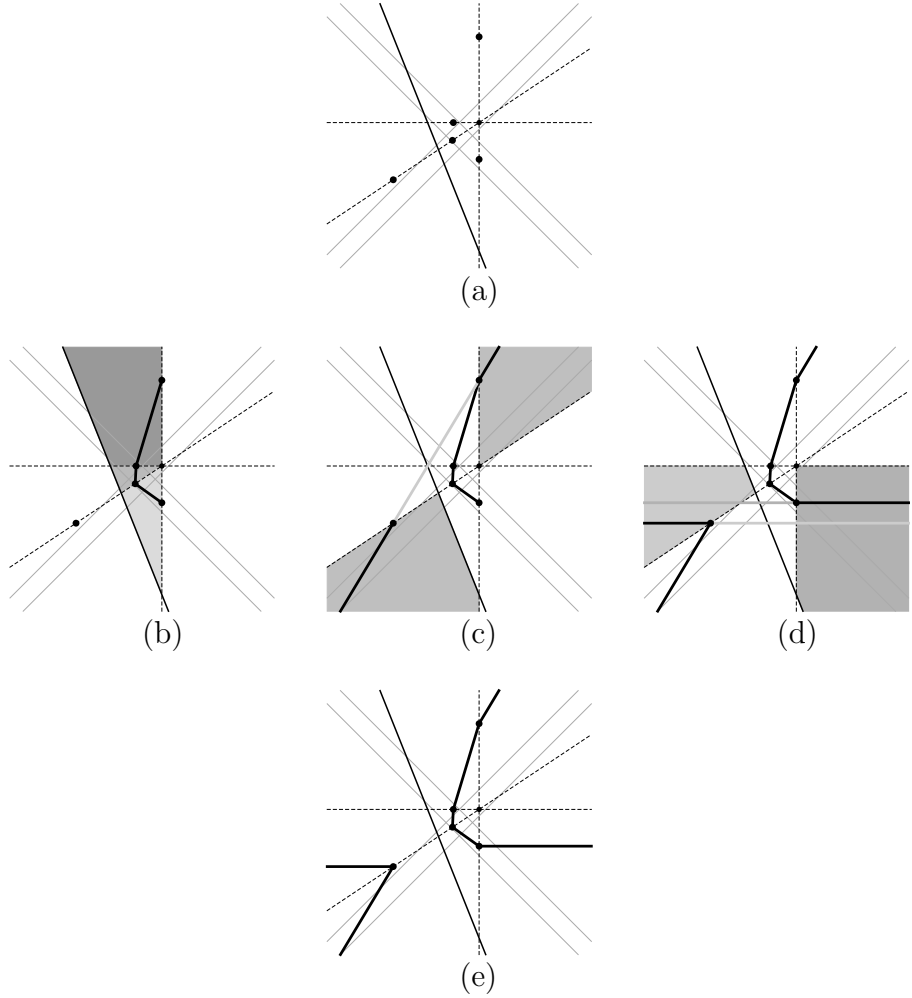

\begin{proof}

Let $W$ be an active double wedge.  In this wedge, there is a partial distance $d_{j,k}(x, \ell)$ being used to determine $d(x, \ell)$.  The double-wedge is the union of wedges $\hat{W}$ and $\check{W}$.  In $\hat{W}$ the absolute values that make up $d_{j,k}(x, \ell)$ resolve in a particular way and in $\check{W}$, these absolute values resolve in the opposite way.  At the same time, $d(x, P)$ resolves in one way on one side of $P^S$ and the opposite way on the other side.  The resulting linear equation valid in $\hat{W}$ and on one side of $P^S$ is also valid in $\check{W}$ and on the other side of $P^S$.  Let $\widetilde{W}$ be this region.  The shaded region in Figure~\ref{adjacencyfig}(b) is an example of such a region.  Any point that satisfies the resulting linear equation and that lies in $\widetilde{W}$ is part of the conic section.

Let $v^{i s_i}$ and $v^{j s_j}$ be adjacent or anti-adjacent vertices on the boundary of $\widetilde{W}$, with neither at infinity.  These points are already known to solve the desired linear equation, so $\gamma^{i s_i j s_j}$ is the set of all points solving the linear equation.  If $v^{i s_i}$ and $v^{j s_j}$ are adjacent, then the segment associated to $v^{i s_i}$ and $v^{j s_j}$ lies in $\widetilde{W}$.  Similarly, if $v^{i s_i}$ and $v^{j s_j}$ are anti-adjacent then the complementing rays lie in $\widetilde{W}$ with one ray in $\hat{W}$ and the other ray in $\check{W}$.

Finally, if $v^{i s_i}$ and $v^{j s_j}$ are adjacent and $v^{j s_j}$ lies at infinity, then $\gamma^{i s_i j s_j}$ must intersect $\widetilde{W}$ through $v^{i s_i}$ but cannot cross $\rho^j$ since, if it did, the resulting intersection $v^{j s_j}$ would not be at infinity.  Therefore, $\gamma^{i s_i j s_j}$ is parallel to $\rho^j$ and the intersection of $\gamma^{i s_i j s_j}$ and $\widetilde{W}$ is the ray associated to $v^{i s_i}$ and $\rho^j$.
\end{proof}

\subsection{Auxiliary points on $P^S$}
The vertices found above enjoy a geometric relationship with $P^S$.  Given two reference lines $\rho^i$ and $\rho^j$, select one vertex on $\rho^i$ and one on $\rho^j$.  This defines a line.  The complementary pair of points also forms a line.  It turns out that these lines intersect at a point on $P^S$ that we call an auxiliary point.  This property was observed for the special case of $C(\ell_{0,0,1}, P_{0,0,1},\kappa)$ in \cite{Laatsch}.

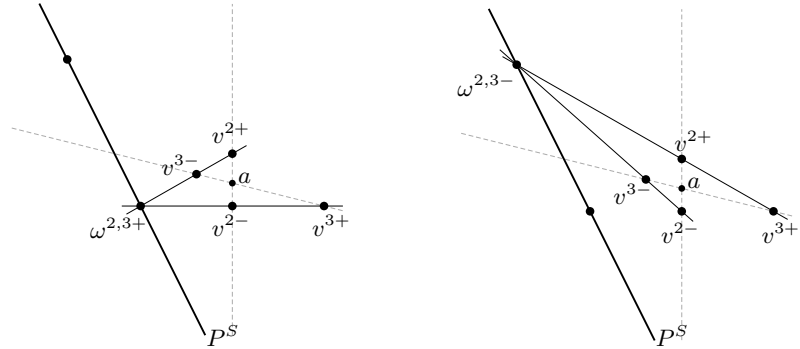
\begin{figure}
\begin{picture}(300,130)
\put(-10,2){
\includegraphics[scale = .5, clip = true, draft = false]{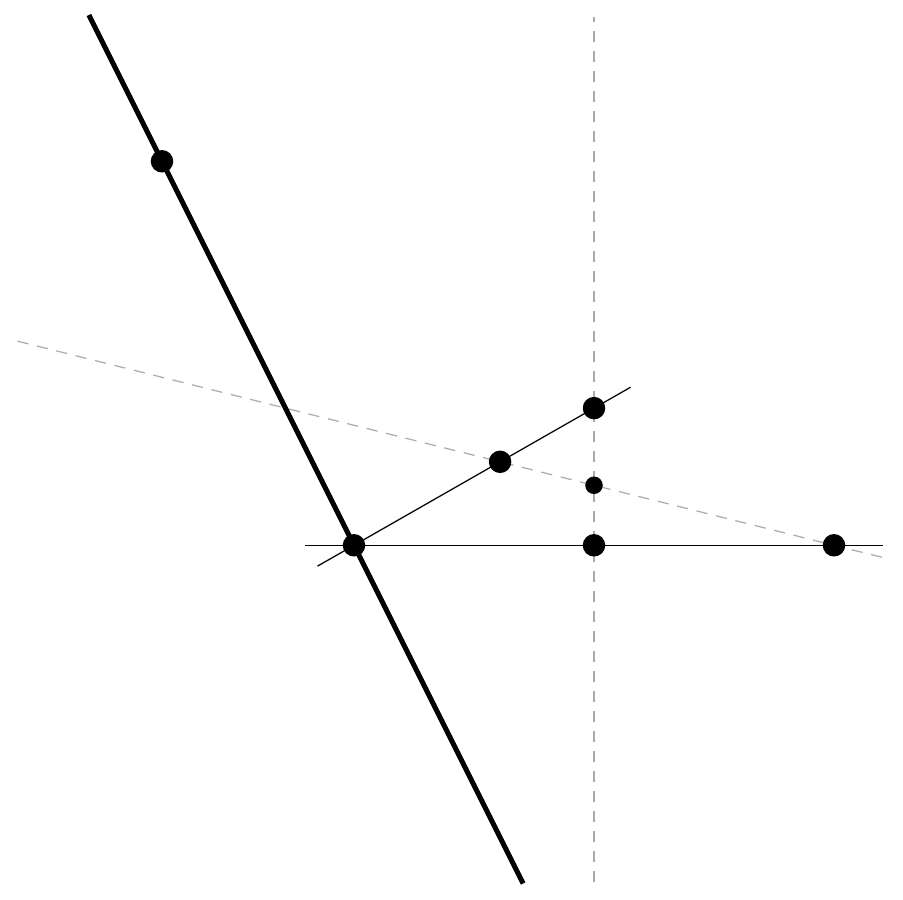}
}
\put(160,0){
\includegraphics[scale = .5, clip = true, draft = false]{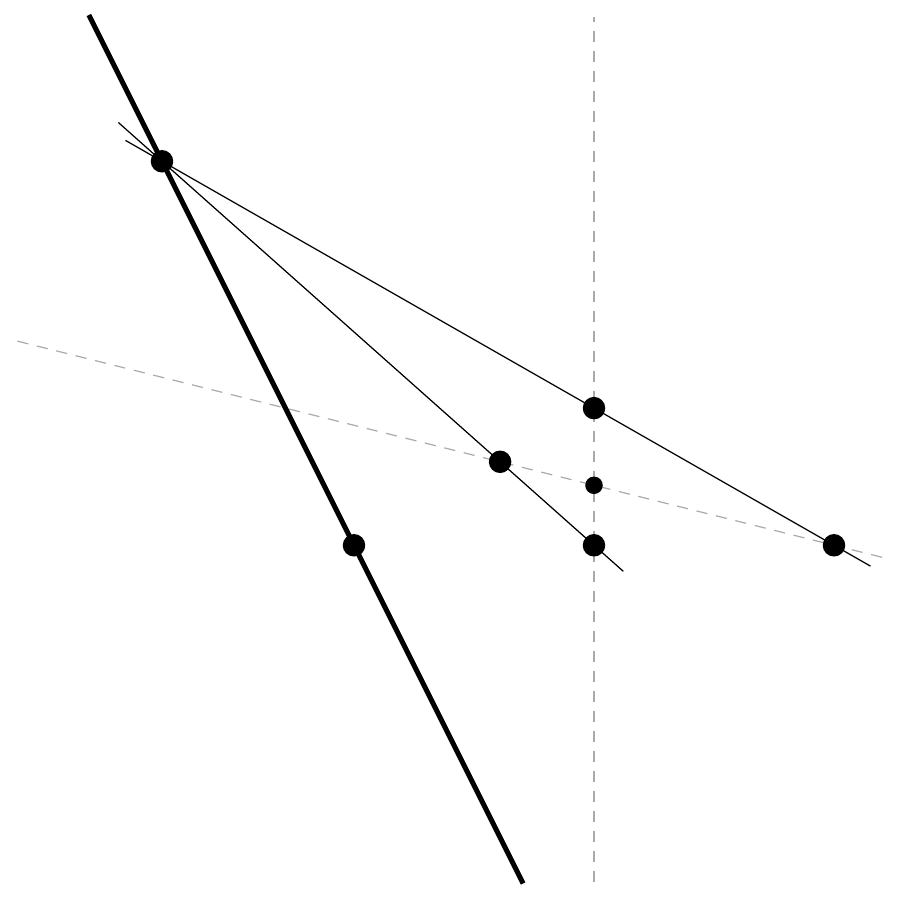}
}
\put(26,41){\scriptsize $\omega^{2,3+}$}
\put(53,65){\scriptsize $v^{3-}$}
\put(110,42){\scriptsize $v^{3+}$}
\put(72,42){\scriptsize $v^{2-}$}
\put(72,77){\scriptsize $v^{2+}$}
\put(70,0){\scriptsize $P^S$}
\put(82, 62){\scriptsize $a$}
\put(165,95){\scriptsize $\omega^{2,3-}$}
\put(224,54){\scriptsize $v^{3-}$}
\put(280,39){\scriptsize $v^{3+}$}
\put(242,39){\scriptsize $v^{2-}$}
\put(247,74){\scriptsize $v^{2+}$}
\put(240,0){\scriptsize $P^S$}
\put(252, 60){\scriptsize $a$}
\end{picture}
\caption{Auxiliary points.  In both images, $A = \left(\frac{1}{2}, \frac{1}{4}, 1 \right)$, $a = \left(2, -\frac{1}{2}, 1 \right)$, and $\kappa = \frac{1}{2}$.  On the left, $\gamma^{2+3-}$ and $\gamma^{2-3+}$ meet at $\omega^{2,3+}$.  On the right,$\gamma^{2+3+}$ and $\gamma^{2-3-}$ meet at $\omega^{2,3-}$.} \label{auxpointfig}
\end{figure}

\begin{thm} \label{auxpointthm}
Given the conic section $C \cap S$ for cone $C = C(\ell, P, \kappa)$ where $\ell$ is non-horizontal, for $i \neq j$
\begin{itemize}
\item the lines $\gamma^{i+j+}$ and $\gamma^{i-j-}$ meet at a point on $P^S$;
\item the lines $\gamma^{i+j-}$ and $\gamma^{i-j+}$ meet at a point on $P^S$.
\end{itemize}
The auxiliary points corresponding to the vertices $v^{1 \pm}$ and $v^{2 \pm}$ are
\[
w^{1,2 \pm} = \left(\frac{\pm A_2 a_1 + A_2 a_2 + \delta}{-A_1 \pm A_2},
				\frac{A_1 a_1 \pm A_1 a_2 + \delta}{\pm A_1 - A_2}, 1\right).
\]
The auxiliary points corresponding to the vertices $v^{1 \pm}$ and $v^{3 \pm}$ are
\[
w^{1,3 \pm} = \left(\frac{-(\delta a_1 \pm A_2 a_2 \pm \delta)}{A_1 a_1 + A_2 a_2 \pm A_1},
				\frac{\pm A_1 a_2 - \delta a_2}{A_1 a_1 + A_2 a_2 \pm A_1}, 1 \right).
\]
The auxiliary points corresponding to the vertices $v^{2 \pm}$ and $v^{3 \pm}$ are
\[
w^{2,3 \pm} = \left(\frac{\pm A_2 a_1 - \delta a_1}{A_1 a_1 + A_2 a_2 \pm A_2},
				\frac{-(\pm A_1 a_1 + \delta a_2 \pm \delta)}{A_1 a_1 + A_2 a_2 \pm A_2}, 1 \right).
\]
In each of the above, the sign choice in the superscript corresponds to the choice in the formula for that point.
\end{thm}

See Figure~\ref{auxpointfig} for an illustration of this result.

\begin{proof}
This can be checked by a direct but tedious computation since we know the coordinates for $v^{i\pm}$ and could write out the formulas for the lines $\gamma^{i s_i j s_j}$, but the following proof is more illuminating.  The vertices are on the lines defined by resolving the absolute values in
\[
d_{i,j}(x, \ell) = \kappa\, d(x, P).
\]  
For a given pair of lines, the potential vertices defining one of the lines are determined by a particular choice of resolving the absolute values, and the potential vertices defining the other line are determined by flipping the choices on one side of the equation, while keeping the choice on the other.

The resulting equations are for the lines in a given pair.  Next, note that these two equations can be written in the form $\alpha = \beta$ and $-\alpha = \beta$ where $\alpha$ is a resolution of $d_{i, j}(x, \ell)$ and $\beta$ is the resolution of $d(x, P)$.  Note that the solution occurs if and only if $\alpha = \beta = 0$.  But $\beta = 0$ is equivalent to $A_1 x_1 + A_2 x_2 + \delta = 0$, which is the equation for the line $P^S$.

For each choice of reference lines, there are two ways to distribute the vertices, resulting in two different auxiliary points, and there are three pairs of reference lines, so there are a total of six auxiliary points on $P^S$, the formulas for which are found by explicitly solving the resulting system of equations.
\end{proof}

From the formulas, we can see that the auxiliary points do not depend on $\kappa$.  Before knowing the formulas, this follows from the fact that $\kappa$ disappears when we modify the linear system being solved.

Similarly to the reference lines, the auxiliary points are not always active in the sense that they lie on extensions of segments or rays in the conic section.  If $\ell_a$ is intermediate, the only active auxiliary points are those that arise from adjacent or anti-adjacent vertices.  If $\ell_a$ is not intermediate, then only the vertices on active reference lines produce active auxiliary points.  As such, only three auxiliary points are active if $\ell_a$ is intermediate, and only two auxiliary points are active otherwise.

With the help of the auxiliary points, an alternative method for constructing a conic section can be formulated.  Note that auxiliary points can lie at infinity, but if all active auxiliary points are finite, we have the following:

\begin{thm} \label{auxsectionconstructionthm}
Given a cone $C = C(\ell, P, \kappa)$, such that all active auxiliary points are finite.  Consider the rays based at each active auxiliary point and containing the associated vertices.  On each ray, consider the set of points between two vertices or, if there is only one vertex, the set of points on the side of the vertex that does not include the auxiliary point.  The conic section $C \cap S$ consists of the union over the rays of these points.
\end{thm}

See Figure~\ref{connectthedotsauxfig} for an illustration of this.  While this theorem does not provide a particularly useful way to construct conic sections when $\ell$ is non-horizontal, it will serve as an effective construction method for conic sections resulting from horizontal defining lines.  We prove the theorem here in the case where $\ell$ is non-horizontal, and save the proof when $\ell$ is horizontal for the next section.

\begin{proof}[Proof when $\ell$ is non-horizontal] Since the active auxiliary points are identified by those lines that have adjacent or anti-adjacent vertices lying on them, the rays described here are parts of these lines and resulting sets of points are just alternate characterizations of the associated segments and rays identified in Theorem~\ref{sectionconstructionthm}.
\end{proof}

If there are active auxiliary points at infinity, the construction described by Theorem~\ref{auxsectionconstructionthm} can be modified.  An auxiliary point lies at infinity precisely when the lines defining it are parallel to $P^S$. In this case, the resulting part of the conic section is always a segment associated to the two active vertices defining a line producing the auxliary point at infinity.

\begin{figure}
\begin{picture}(300,500)
\put(90,400){
\includegraphics[scale = .4, clip = true, draft = false]{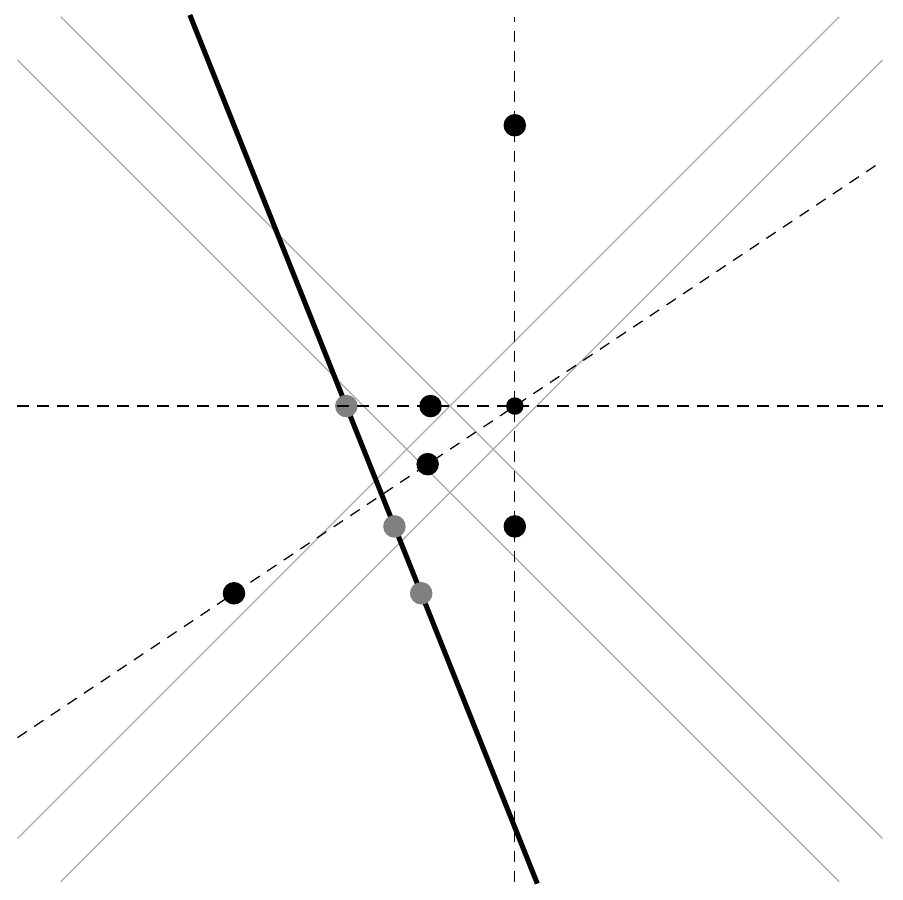}
}
\put(146,390){(a)}
\put(-30,270){
\includegraphics[scale = .4, clip = true, draft = false]{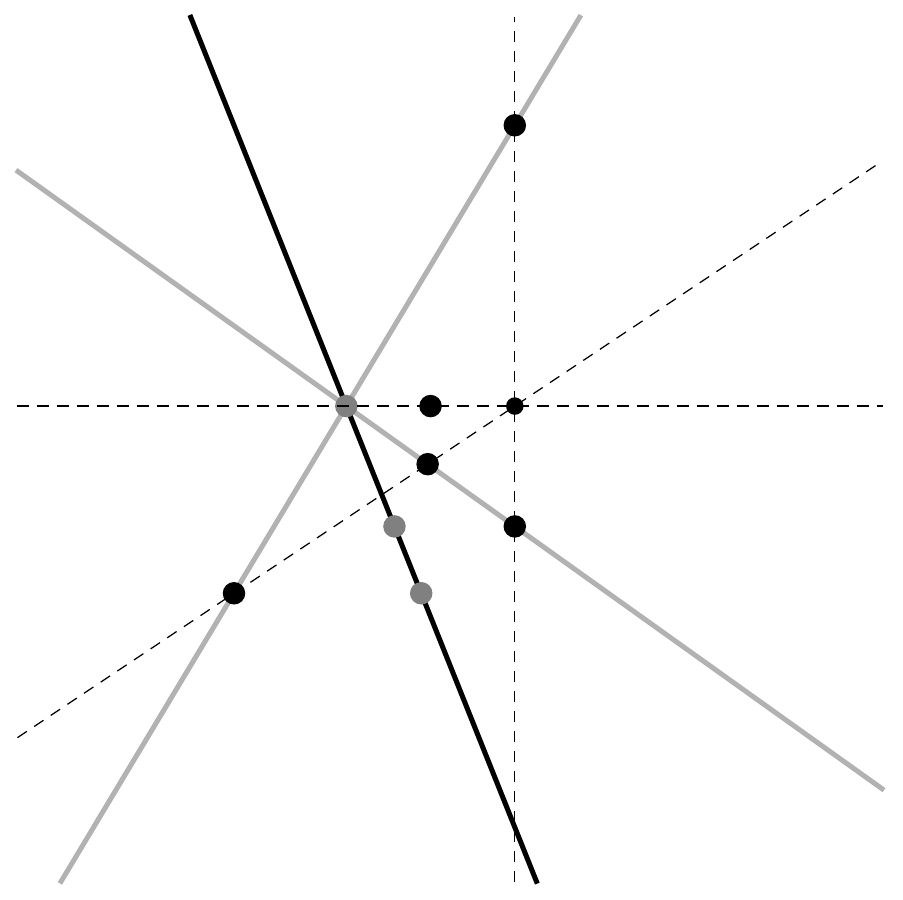}
}
\put(34,265){\vector(0,-1){15}}
\put(90,270){
\includegraphics[scale = .4, clip = true, draft = false]{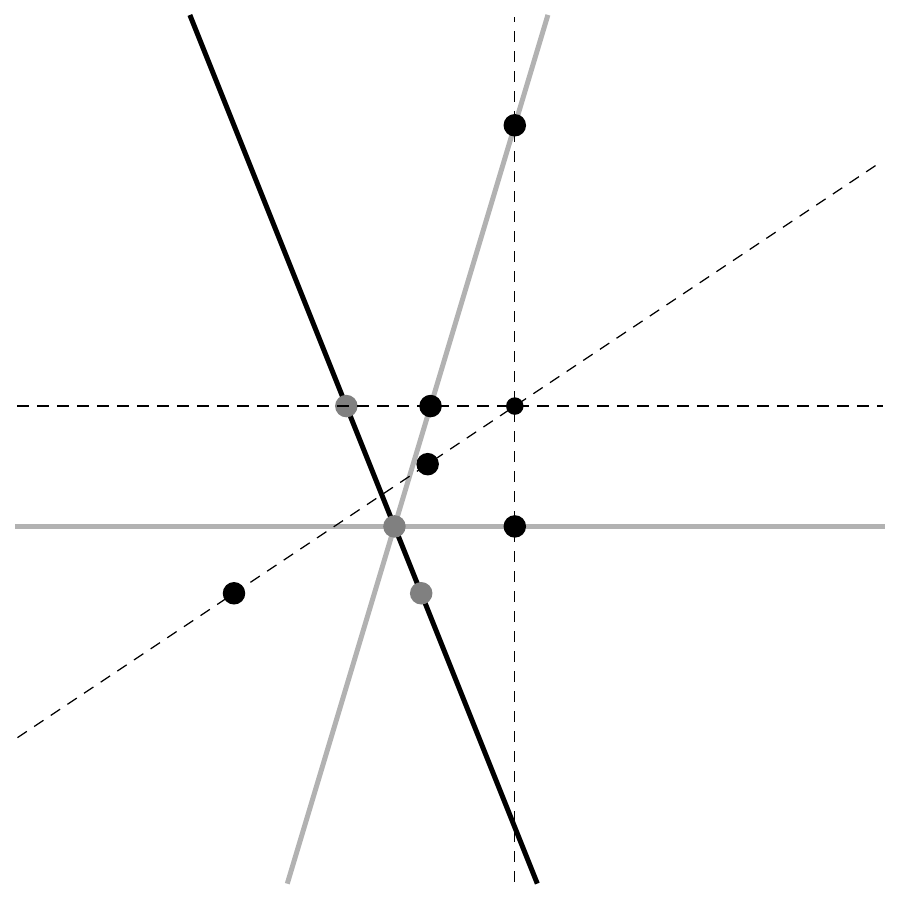}
}
\put(154,265){\vector(0,-1){15}}
\put(210,270){
\includegraphics[scale = .4, clip = true, draft = false]{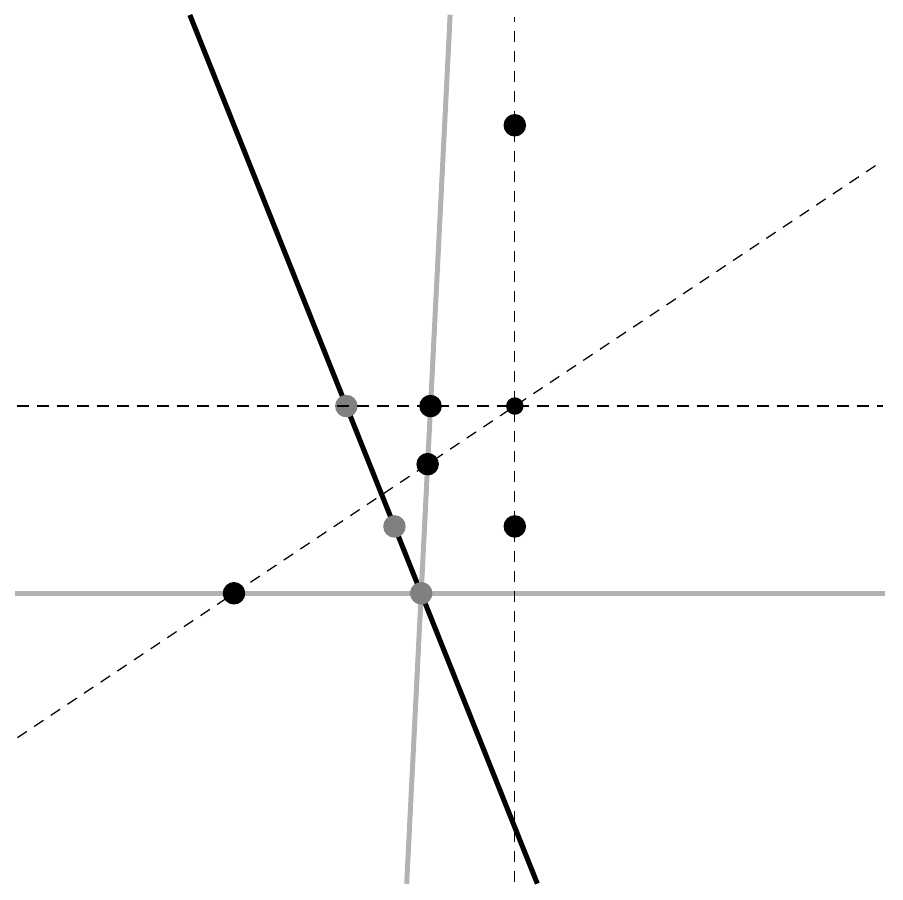}
}
\put(274,265){\vector(0,-1){15}}
\put(-30,140){
\includegraphics[scale = .4, clip = true, draft = false]{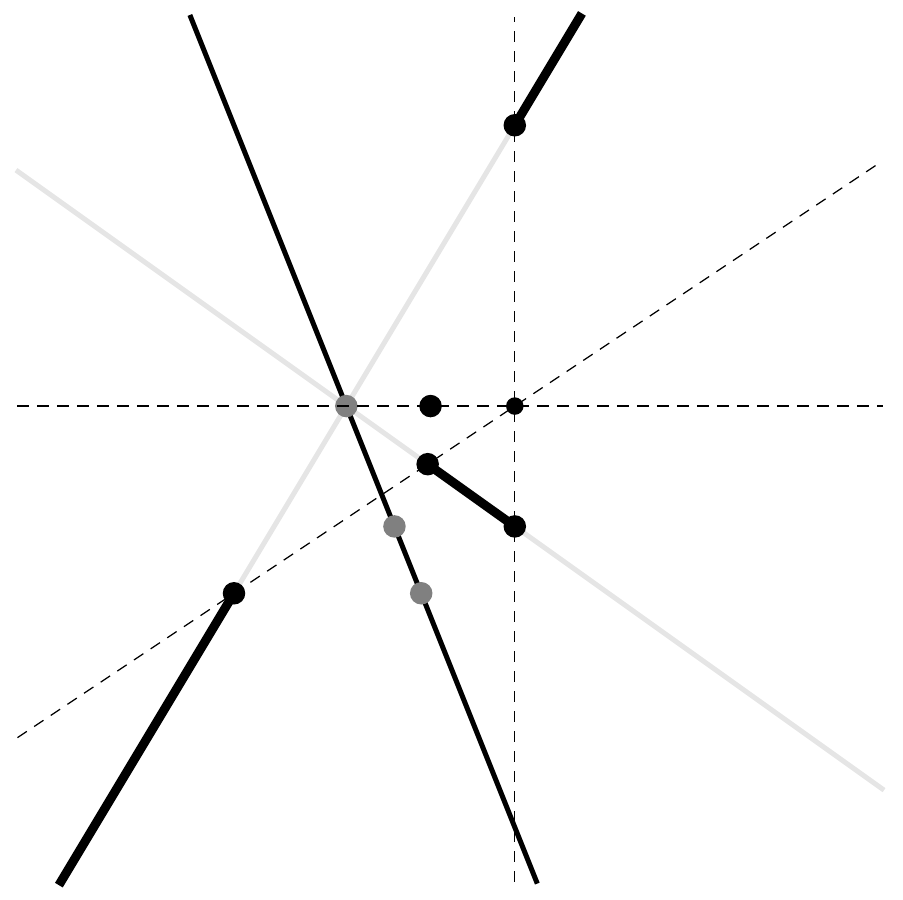}
}
\put(26,130){(b)}
\put(90,140){
\includegraphics[scale = .4, clip = true, draft = false]{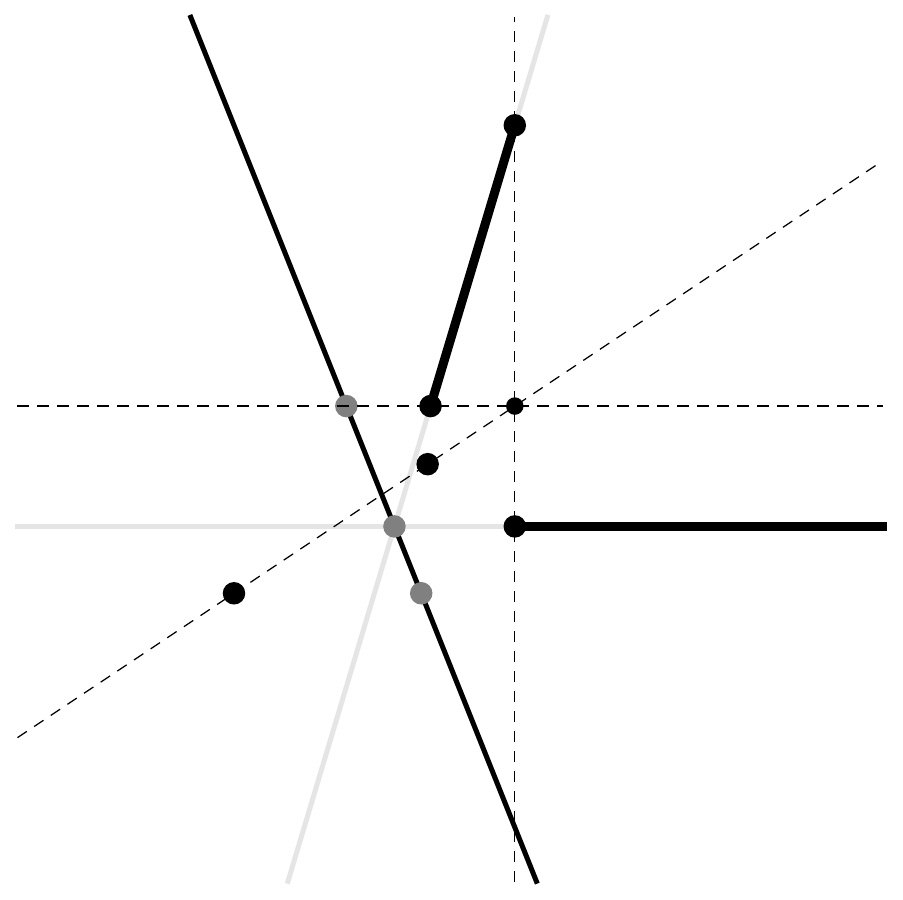}
}
\put(146,130){(c)}
\put(210,140){
\includegraphics[scale = .4, clip = true, draft = false]{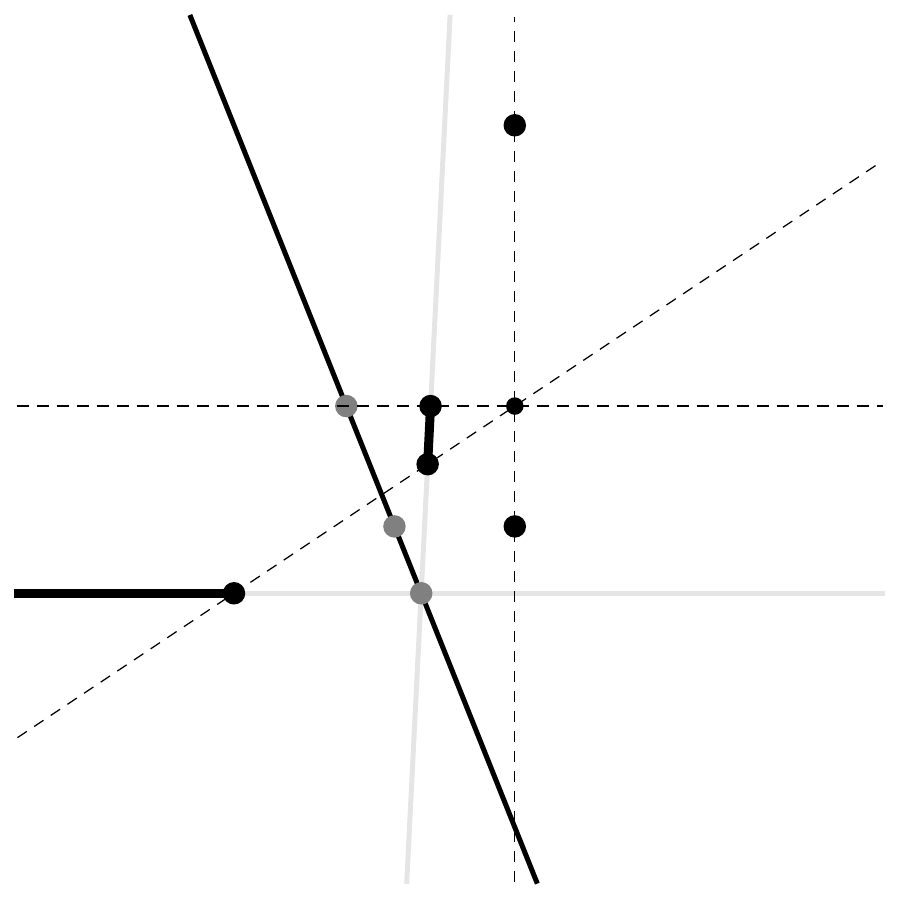}
}
\put(266,130){(d)}
\put(90,10){
\includegraphics[scale = .4, clip = true, draft = false]{connect-the-dots-final.pdf}
}
\put(146,0){(e)}
\end{picture}
\caption{Using the rays through auxiliary points to produce a conic section.  In this example, just like in Figure~\ref{connectthedotsfig}, $A = \left(\frac{1}{2}, \frac{1}{5}, 1\right)$, $a = \left(\frac{3}{2}, 1, 1 \right)$, and $\kappa = 2$.  Starting with the vertices and active auxiliary points (a), for the rays associated to each auxiliary point, identify the points between two vertices or beyond a single vertex on each ray (b), (c), (d) to produce the complete conic section (e).} \label{connectthedotsauxfig}
\end{figure}

\subsection{Characteristics of sections} \label{sectioncarsec}

Theorems~\ref{sectionconstructionthm} and \ref{auxsectionconstructionthm} provide methods for constructing conic sections, but do not indicate what kinds of shapes might arise.  Here, we establish some general facts about how the vertices must be located, including ways to identify when the vertices will lie on the same side of $P^S$ as $a$, at infinity, or on the opposite side of $P^S$.  These characteristics combined with Theorems~\ref{sectionconstructionthm} and \ref{auxsectionconstructionthm} allow us to sketch conic sections without detailed calculations and allow us to distinguish realistic sketches from non-realistic ones.

As in the Euclidean setting, we say a conic section is an ellipse if it is a bounded set, a parabola if it is unbounded with one component, and a hyperbola if is unbounded with two components.  Also as in the Euclidean setting, ellipses arise when the slicing plane completely slices one half of the cone, parabolas arise when the slicing plane slices one half of the cone, but is parallel to at least one line in the cone, and hyperbolas arise when the slicing plane slices both halves of the cone.  Since our slicing plane is always parallel to the $(x_1,x_2)$-plane, for a parabola, at least one edge of the cone lies in the $(x_1,x_2)$-plane and each half of the cone lies on one side of this plane.  For hyperbola both halves of the cone include points that lie above and below this plane.

\begin{lemma} \label{vertexpositionlemma}
Given the conic section $C \cap S$ for a cone $C = C(\ell_a, P_A, \kappa)$, on each active reference line $\rho^i$,
\begin{itemize}
\item exactly one of the vertices $v^{i \pm}$ lies on the segment of $\rho^i$ between $P^S$ and $a$;
\item if neither vertex lies at infinity, then they lie on the opposite sides of $a$ if and only if they lie on the same side of $P^S$.
\end{itemize}
\end{lemma}

\begin{proof}
The first statement follows immediately from the Intermediate Value Theorem and the monotonicity of the distance functions $d(x, \ell)$ and $d(x, P)$ restricted to the line $\rho^i$.

The second statement follows from the fact that $P^S$ and $a$ subdivide $\rho^i$ into two rays and a segment.  By the first statement, only one vertex $v$ can lie on the segment, so the other vertex must lie on one of the rays.  The ray on the opposite side of $a$ as $v$ lies on the same side of $P^S$, and vice versa.
\end{proof}

With the help of this lemma we can make some immediate observations:  On one hand, $C \cap S$ is an ellipse if and only if all pairs of active vertices $v^{i \pm}$ are finite and lie on opposite sides of $a$.  In this case, there are no anti-adjacent vertices and so the conic section is just a union of segments.  On the other hand, $C \cap S$ is a hyperbola if an only if for some pair of active vertices, both vertices are finite, lie on the same ray associated to $a$, and, by necessity, $P^S$ separates them.

While this lemma and these observations provide some structure, it is not clear when the conditions are actually met.  In light of these observations, define the set
\[
Q_{A, \kappa} = \left\{x \in S: |A_1 x_1 + A_2 x_2| < \frac{M}{\kappa} \right\}.
\]
We call this set the characterizing strip for the cone $C(\ell_a, P_A, \kappa)$.  Also, let $e^{1 \pm} = (\pm 1, 0, 1)$ and $e^{2 \pm} = (0, \pm 1, 1)$ be the vertices of $\sigma_1(0,0,1) \cap S$.  The following lemma characterizes where these vertices lie relative to $P^S$ and when they lie at infinity.

\begin{lemma} \label{striplemma}
Let $C \cap S$ be the conic section for the cone $C = C(\ell_a, P_A, \kappa)$ where $\ell_a$ is non-horizontal, let $Q = Q_{A, \kappa}$ be the characterizing strip for $C$, and let $i \in \{1, 2\}$.
\begin{itemize}
\item If $e^{i \pm} \in Q$, then the corresponding vertices $v^{i \pm}$ lie on the same side of $P^S$.  If $a \in Q$, then the vertices $v^{3 \pm}$ lie on the same side of $P^S$.
\item If $e^{i \pm} \in \partial Q$ then one of the corresponding vertices $v^{i \pm}$ lies at infinity.  If $a \in \partial Q$, then one of the vertices $v^{3 \pm}$ lies at infinity.
\item If $e^{i \pm} \in S \backslash \overline{Q}$ then the corresponding vertices $v^{i\pm}$ lie on opposite sides of $P^S$.  If $a \in S \backslash \overline{Q}$, then the vertices $v^{3 \pm}$ lie on opposite sides of $P^S$. 
\end{itemize}
\end{lemma}

\begin{proof}
In light of Lemma~\ref{vertexpositionlemma} it is enough to determine under what conditions the vertices lie on the same or opposite sides of $a$ or at infinity.   This is determined completely by the denominators in Equations~\eqref{v1altformulaeq} \eqref{v2altformulaeq}, and \eqref{v3altformulaeq}, and the positions of $e^{i \pm}$ and $a$ relative to $Q$ simply provide a geometric representation of this.  When $e^{i \pm}$ or $a$ lie in $Q$, there is a sign change for the corresponding vertices so those vertices must lie on opposite sides of $a$.  If any of these points lie on $\partial Q$, the corresponding denominator is zero for one sign choice and the corresponding vertex lies at infinity.  If any of these points lie outside $\overline{Q}$ the corresponding denominators do not change sign so those vertices lie on the same side of $a$.
\end{proof}

With the help of this lemma, we can now completely characterize under what conditions we will produce ellipses, parabolas, and hyperbolas.

\begin{thm} \label{coniccharacterizationthm}
Let $C \cap S$ be the conic section for the cone $C = C(\ell_a, P_A, \kappa)$ where $\ell_a$ is non-horizontal, let $Q = Q_{A, \kappa}$ be the characterizing strip for $C$, and let $i \in \{1, 2\}$.
\begin{itemize}
\item If $\sigma_1(0)$ and $a$ lie in  $Q$, then $C \cap S$ is an ellipse.
\item If $\sigma_1(0)$ and $a$ lie in $\overline{Q}$, and at least one intersects $\partial Q$, then $C \cap S$ is a parabola.
\item If any of the vertices $e^{i \pm}$ of $\sigma_1(0)$ or $a$ lie in $S \backslash \overline{Q}$, then $C \cap S$ is a hyperbola.
\end{itemize}
\end{thm}

\begin{proof}
This follows directly from Lemma~\ref{striplemma} by considering all active vertices at once.
\end{proof}

See Figure~\ref{stripfig} for examples illustrating this result.

\begin{figure}
\begin{picture}(360,360)
\put(0,200){
\includegraphics[scale = .6, clip = true, draft = false]{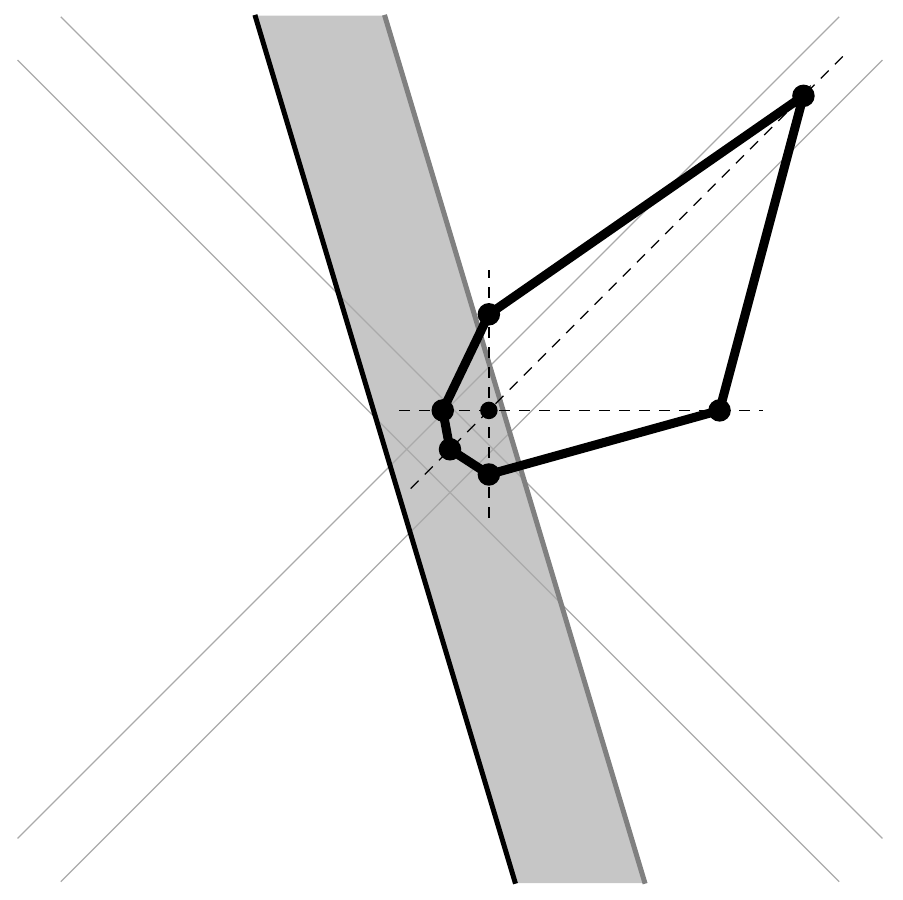}
}
\put(30,185){(a) \scriptsize $a = \left(\frac{9}{10}, \frac{9}{10},1 \right),\ \kappa = 1$}
\put(200,200){
\includegraphics[scale = .6, clip = true, draft = false]{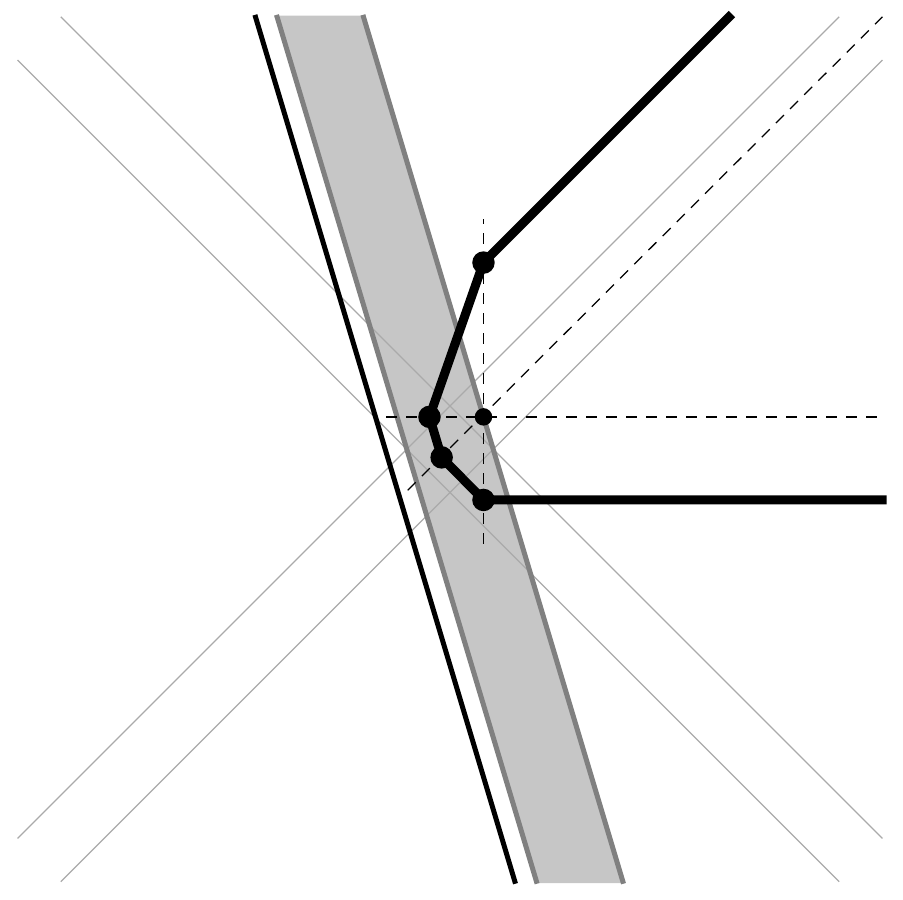}
}
\put(230,185){(b) \scriptsize $a = \left(\frac{31}{40}, \frac{3}{4},1 \right),\ \kappa = \frac{3}{2}$}
\put(0,15){
\includegraphics[scale = .6, clip = true, draft = false]{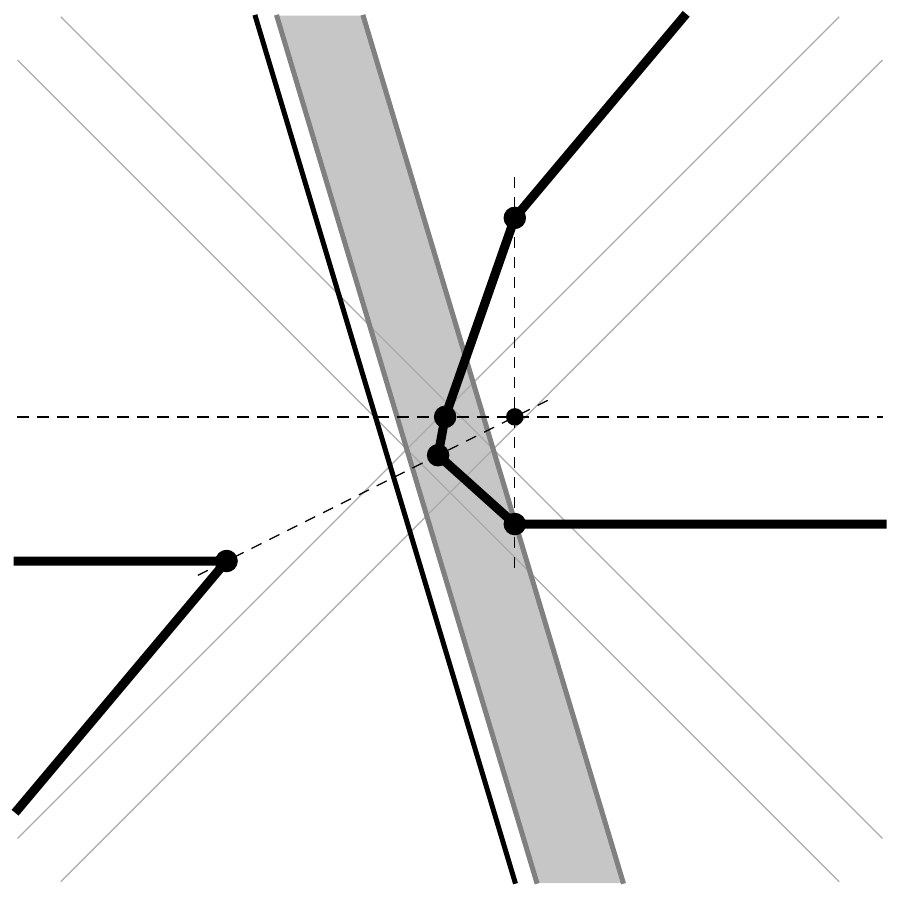}
}
\put(30,0){(c) \scriptsize $a = \left(\frac{3}{2}, \frac{3}{4},1 \right),\ \kappa = \frac{3}{2}$}
\put(200,15){
\includegraphics[scale = .6, clip = true, draft = false]{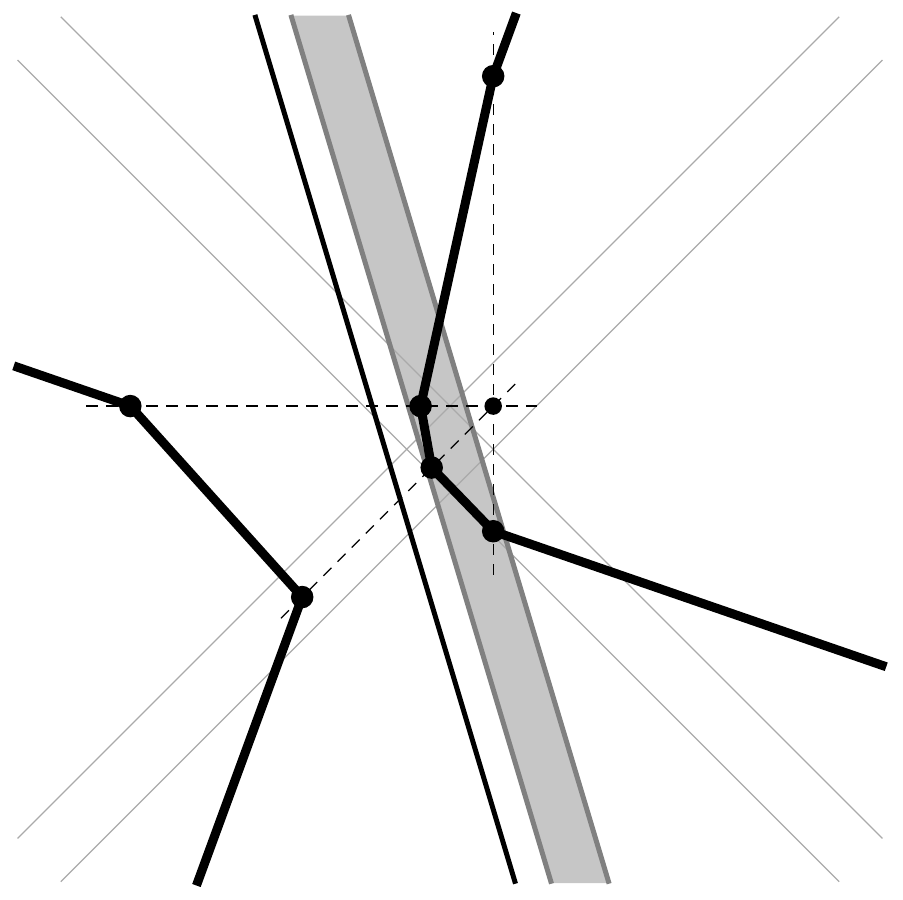}
}
\put(230,0){(d) \scriptsize $a = \left(1, 1, 1 \right),\ \kappa = \frac{9}{4}$}

\end{picture}
\caption{The type of conic section depends on the location of $a$ and $\sigma_1(0)$ relative to $Q$ which is shaded.  These are just four of many possibilities.  In these cases, $A = \left( \frac{2}{3}, \frac{1}{5}, 1 \right)$.  The darker line parallel to $Q$ is $P^S$ which coincides in (a) with one edge of $Q$.  The parameters $a$ and $\kappa$ are given under the corresponding image.  In (a), $a$ and $\sigma_1(0)$ lie inside $Q$ and we have an ellipse.  In (b), $a$ and $e^{1 \pm}$ lie on $\partial Q$, $e^{2 \pm}$ lies in $Q$, and we have a parabola.  In (c), $a$ lies outside $\overline{Q}$, $e^{1 \pm}$ lie on $\partial Q$, $e^{2 \pm}$ lies in $Q$, and we have a hyperbola.  In (d), $a$ and $e^{1 \pm}$ lie outside $\overline{Q}$, $e^{2 \pm}$ lies in $Q$, and we have a hyperbola.} \label{stripfig}
\end{figure}
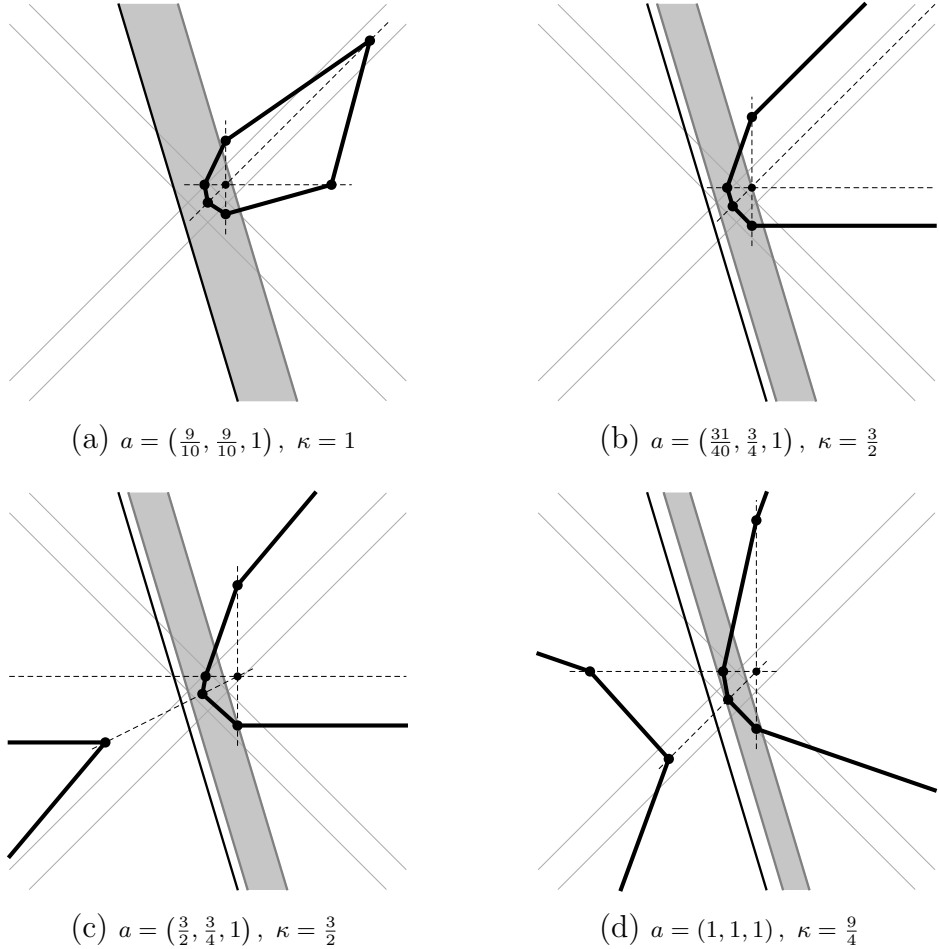

\section{Conic sections when $\ell$ is horizontal} \label{conicsectionswhenellishorizontalsection}

While the approach here starts similarly to the non-horizontal case, there are some significant differences that develop.

\subsection{Vertices of the sections}

When $\ell$ is horizontal, $\ell \cap S$ is the empty set, so $a$ is not identified with a point in $S$.  Similarly, the reference lines $\rho^1$ and $\rho^2$ do not exist.  The reference line $\rho^3$ does exist, and we can use it as a way to represent $a$ by virtue of the fact that $\rho^3 = P^3 \cap S = \{x \in S: a_1 x_2 = a_2 x_1\}$ encodes $a$.

\begin{thm} \label{horizontalverticesthm}
Given a cone $C(\ell_a, P_A, \kappa)$ where $a = (a_1, a_2, 0)$ and $A = (A_1, A_2, \delta)$, and slicing plane $S = \{x \in \mathbb{R}^3:  x_3 = 1\}$, the reference line $\rho^3 = \{x \in S: a_1 x_2 = a_2 x_1\}$ is always active and the vertices lying on $\rho^3$ are
\[
v^{3 \pm} = (r^{\pm} a_1, r^{\pm} a_2, 1)
\]
where
\[
r^{\pm} = 1+ \frac{\pm \frac{M}{\kappa} + A_1 a_1 + A_2 a_2 + \delta}{-(A_1 a_1 + A_2 a_2)}.
\]
\end{thm}

\begin{proof}
Consider two cases.  First, if $|a_1| \geq |a_2|$, then the partial distance being used is
\[
d(x, \ell) = d_{2, 3}(x, \ell) = \left|x_2 - a_2 \frac{x_1}{a_1} \right| + 1.
\]
Restricting attention to points on $\rho^3$, note that $d(x, \ell) = 1$, which should not be surprising since these points lie directly above $\ell$.  Hence, the vertices on $\rho^3$ are solutions to
\[
1 = \kappa\, \frac{|A_1 x_1 + A_2 x_2 + \delta |}{M}.
\]
Resolving the absolute values in two different ways and solving yields
\begin{align*}
x_i &= \frac{\delta \pm \frac{M}{\kappa}}{-A_1 a_1 - A_2 a_2} a_i \\
	&= \left(1+ \frac{\pm \frac{M}{\kappa} + A_1 a_1 + A_2 a_2 + \delta}{-(A_1 a_1 + A_2 a_2)} \right) a_i.
\end{align*}
For the second case, note that this formula is symmetric in $a_1$ and $a_2$, so if $|a_2| \geq |a_1|$, the calculation analogous to that above using $d(x, \ell) = d_{1, 3}(x, \ell)$ results in the same formula.
\end{proof}

While the vertices $v^{3 \pm}$ can be rewritten in a fashion similar to Equation~\eqref{v3altformulaeq}, we find that doing so is both somewhat artificial and unnecessary; artificial because $a$ is not a point in $S$ and unnecessary because the conic sections resulting from horizontal defining lines are relatively simple.

\subsection{Auxiliary points on $P^S$ and constructing the sections}

Unlike the case where $\ell$ is non-horizontal, here we do not have enough vertices to construct the resulting conic section using the methods of Theroem~\ref{sectionconstructionthm}, nor can we define the auxiliary points as the intersections of lines defined by vertices.  Nonetheless, auxiliary points exist, and once they are found, they can be used to construct the sections using Theorem~\ref{auxsectionconstructionthm}.

\subsubsection{Auxiliary points}

Here we establish the result for horizontal lines similar to Theorem~\ref{auxpointthm}.

\begin{thm}
Given the conic section $C \cap S$ for cone $C = C(\ell_a, P_A, \kappa)$ where $\ell_a$ is horizontal, there are four auxiliary points on $P^S$, two of which are active at a time.

If $|a_1| \geq |a_2|$, then the active auxiliary points are
\[
w^{I \pm} = \left(\frac{\pm A_2 a_1 - \delta a_1}{A_1 a_1 + A_2 a_2},
			\frac{-(\pm A_1 a_1 + \delta a_2)}{A_1 a_1 + A_2 a_2}, 1 \right).
\]
If $|a_2| \geq |a_1|$, then the active auxiliary points are
\[
w^{I\!I \pm} = \left(\frac{-(\pm A_2 a_2 + \delta a_1)}{A_1 a_1 + A_2 a_2},
			\frac{\pm A_1 a_2 - \delta a_2}{A_1 a_1 + A_2 a_2}, 1 \right).
\]
\end{thm}

Note that if $|a_1| = |a_2|$ then $w^{I \pm} = w^{I\!I \mp}$.

\begin{proof}
When $|a_1| \geq |a_2|$, the equation defining the cone is
\[
d_{2,3}(x, \ell) = \kappa \, d(x, P)
\]
which expands to
\[
\left|x_2 - a_2 \frac{x_1}{a_1} \right| + 1 = \kappa\, \frac{|A_1 x_1 + A_2 x_2 + \delta|}{M}.
\]
Each absolute value can be resolved in two ways, leading to four different equations.  Given a particular resolution, the equation determined by making the opposite choice on both absolute values will have the same slope as the original, so the four lines form a parallelogram.  Two of the vertices of this parallelogram are the vertices $v^{3 \pm}$.   These vertices correspond to where the absolute value on the left is equal to zero.  The other two vertices correspond to where the absolute value on the right is equal to zero, but the right hand side is $d(x, P)$ so these points lie on $P^S$.  Moreover, they are the intersections of the lines defined by
\[
\pm \left( x_2 - a_2 \frac{x_1}{a_1} \right) + 1 = 0
\]
and
\[
A_1 x_1 + A_2 x_2 + \delta = 0.
\]
Solving this system, we find the two auxiliary points
\[
w^{I \pm} = \left(\frac{\pm A_2 a_1 - \delta a_1}{A_1 a_1 + A_2 a_2},
			\frac{-(\pm A_1 a_1 + \delta a_2)}{A_1 a_1 + A_2 a_2}, 1 \right).
\]

If $|a_2| \geq |a_1|$, then the same analysis as above applies to the equation
\[
d_{1,3}(x, \ell) = \kappa \, d(x, P)
\]
which results in the auxiliary points
\[
w^{I\!I \pm} = \left(\frac{-(\pm A_2 a_2 + \delta a_1)}{A_1 a_1 + A_2 a_2},
			\frac{\pm A_1 a_2 - \delta a_2}{A_1 a_1 + A_2 a_2}, 1 \right).
\]
\end{proof}

\subsubsection{Constructing the sections:  connecting the dots}

When $\ell$ is horizontal, since there are only two vertices with which to work, we do not have a result analogous to Theorem~\ref{sectionconstructionthm}, but the auxiliary points allow us to construct the resulting conic section using Theorem~\ref{auxsectionconstructionthm}.  See Figure~\ref{hcasefig}.  Note that in this case, the auxiliary points are never at infinity.  We prove the horizontal case here.

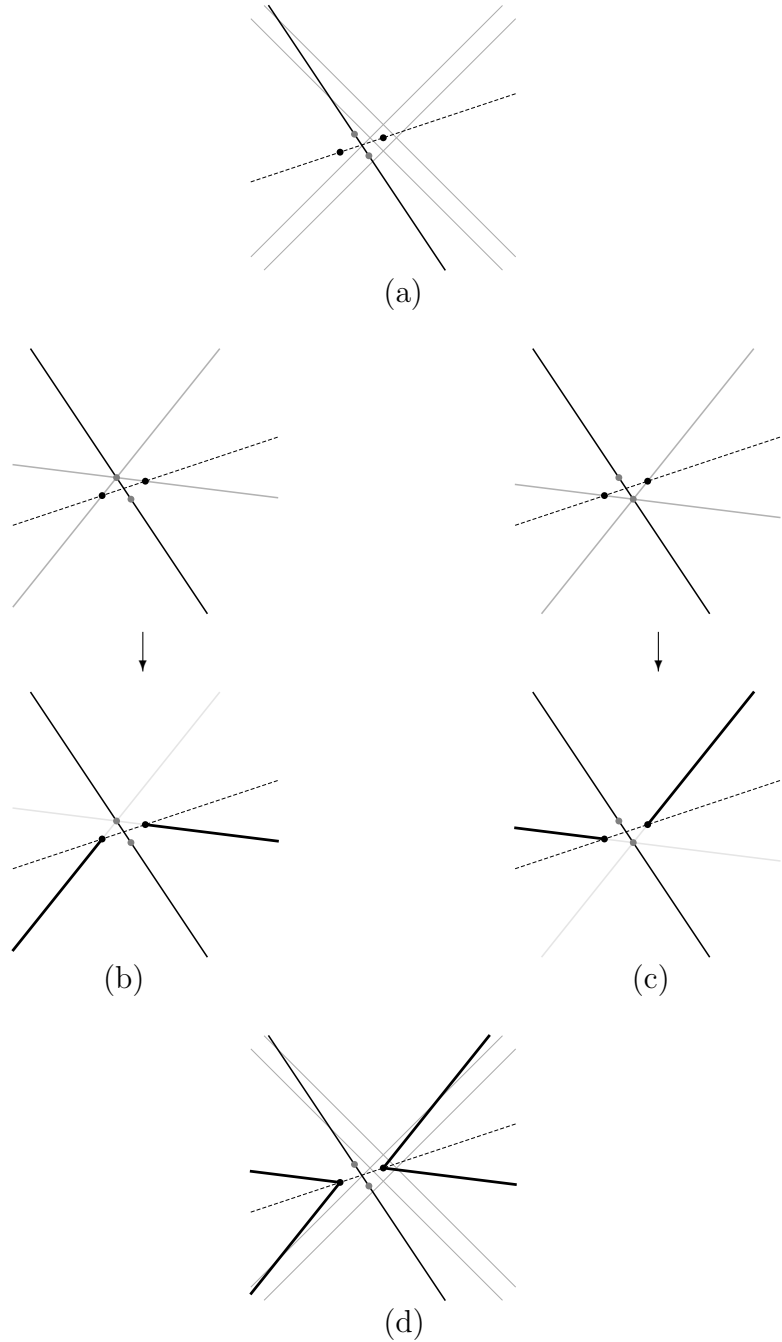
\begin{figure}
\begin{picture}(300,500)
\put(90,400){
\includegraphics[scale = .4, clip = true, draft = false]{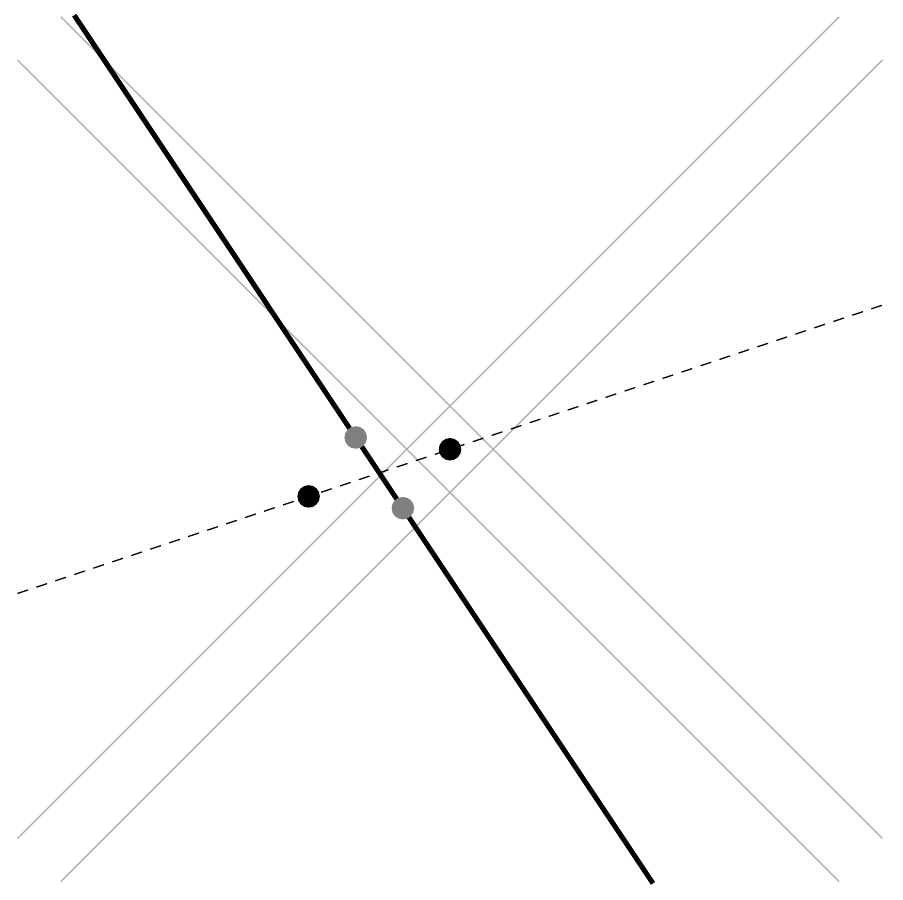}
}
\put(146,390){(a)}
\put(0,270){
\includegraphics[scale = .4, clip = true, draft = false]{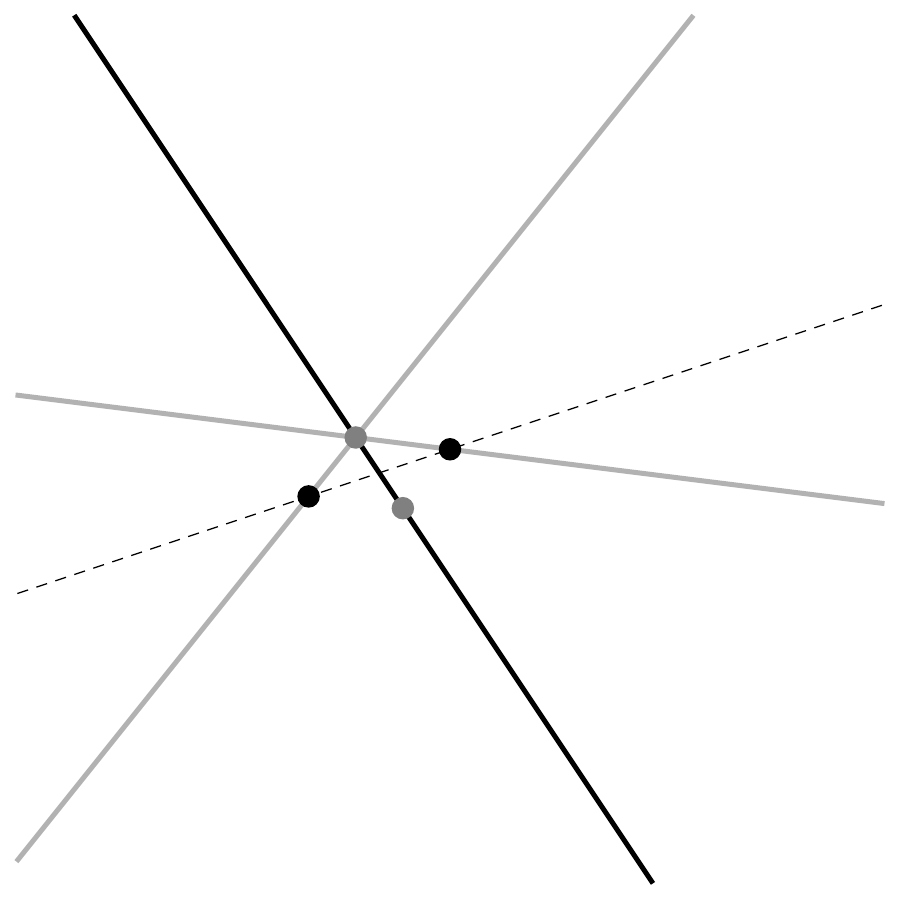}
}
\put(55,265){\vector(0,-1){15}}
\put(190,270){
\includegraphics[scale = .4, clip = true, draft = false]{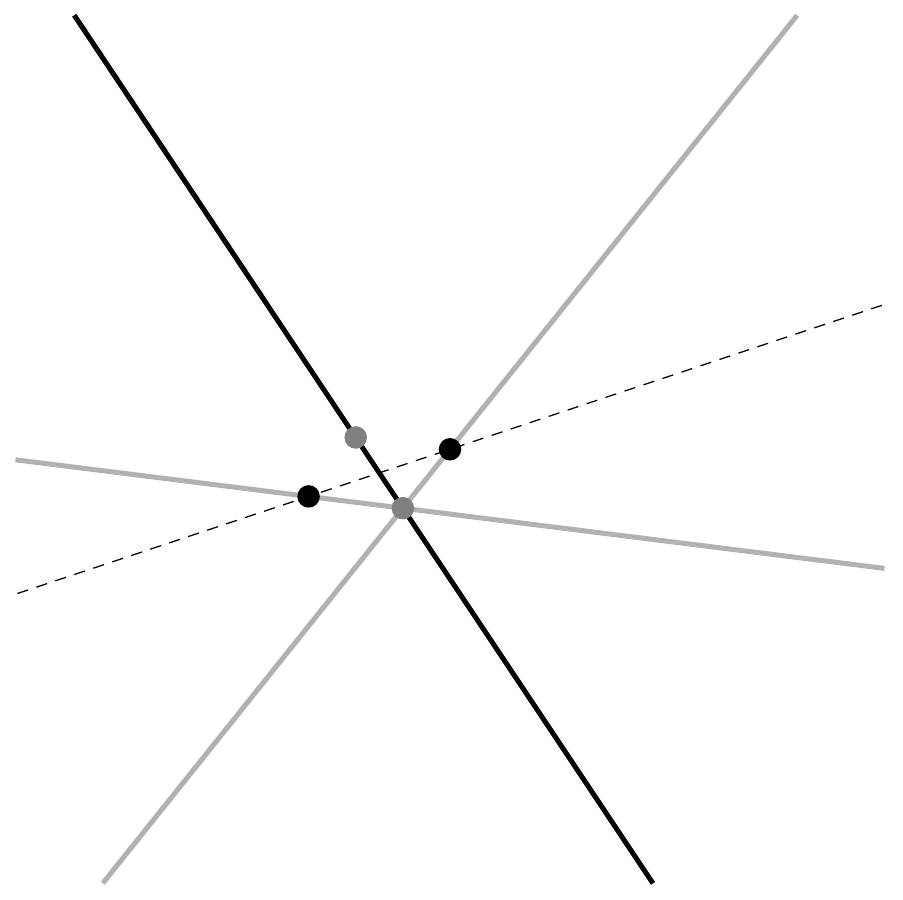}
}
\put(250,265){\vector(0,-1){15}}
\put(0,140){
\includegraphics[scale = .4, clip = true, draft = false]{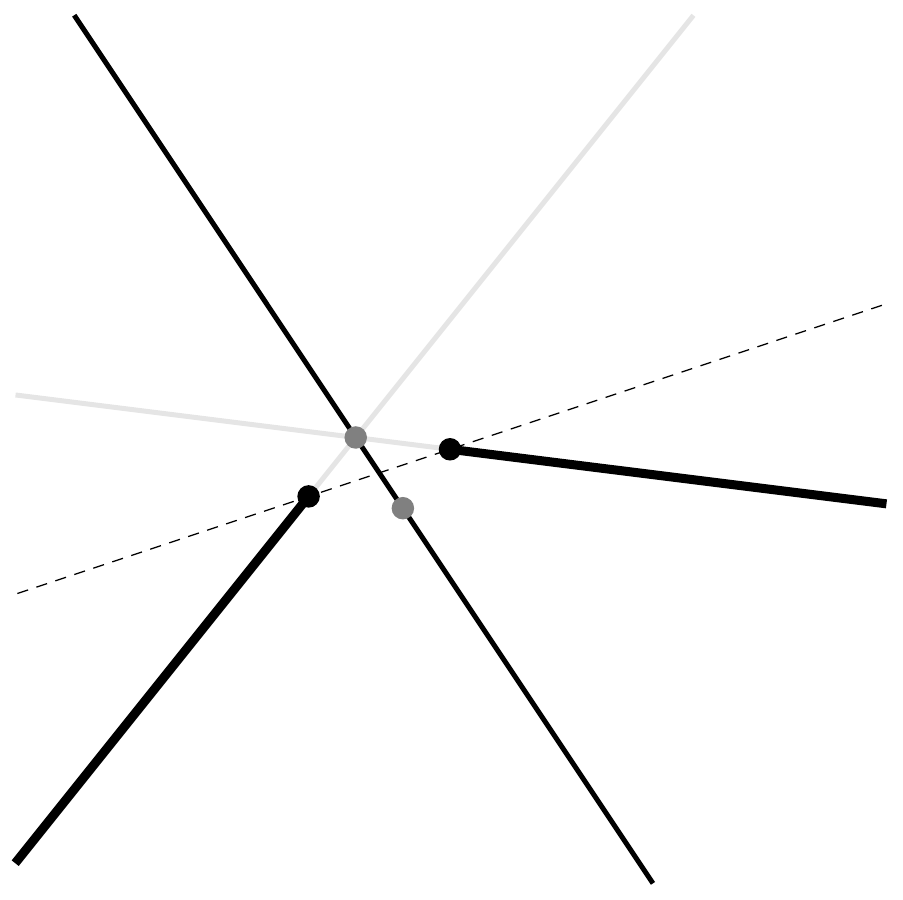}
}
\put(40,130){(b)}
\put(190,140){
\includegraphics[scale = .4, clip = true, draft = false]{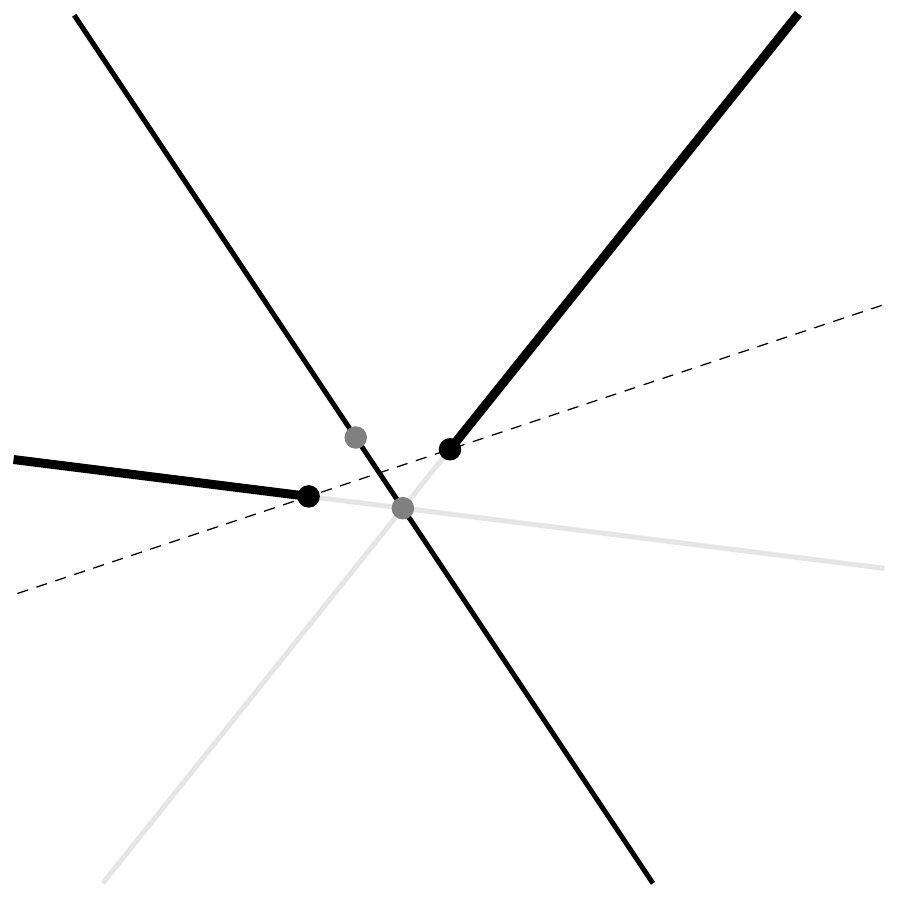}
}
\put(240,130){(c)}
\put(90,10){
\includegraphics[scale = .4, clip = true, draft = false]{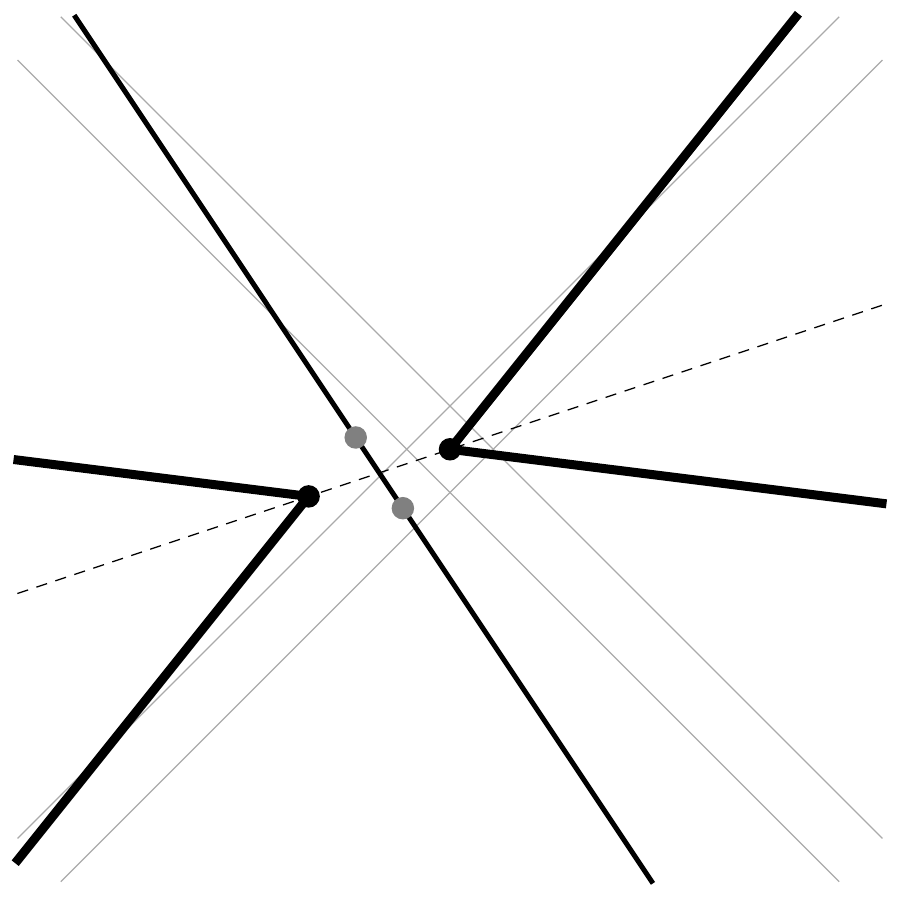}
}
\put(146,0){(d)}
\end{picture}
\caption{Using the rays through auxiliary points to produce a conic section.  In this example, $A = \left(\frac{1}{2}, \frac{1}{3}, 1\right)$, $a = (3, 1, 0)$, and $\kappa = 1$.  Starting with the vertices and active auxiliary points (a), for the rays associated to each auxiliary point, identify the points beyond a single vertex on each ray (b), (c), to produce the complete conic section (d).  Note that the vertices and auxiliary points are the vertices of a parallelogram.} \label{hcasefig}
\end{figure}

\begin{proof}[Proof of Theorem~\ref{auxsectionconstructionthm} when $\ell$ is horizontal]
The lines that define the auxiliary points arise from the edges of the conic section, so it remains only to determine which parts of these lines to include.

On each line, the auxiliary point $\omega$ and the vertex $v$ subdivide the line into the segment and two rays associated to $\omega$ and $v$.  The ray and segment terminating at $\omega$ cannot be part of the conic section because $\omega$ itself is not part of the section.  The remaining ray must therefore be the subset of the line that makes up part of the conic section.
\end{proof}

\subsection{Characteristics of sections} \label{hlinesetioncarsec}

Compared to the case where $\ell$ is non-horizontal, this case is much simpler.

\begin{thm} \label{horizontalconiccharacterizationthm}
If $\ell_a$ is horizontal, the conic section $C \cap S$ for the cone $C = C(\ell_a, P_A, \kappa)$ is always a hyperbola.
\end{thm}

This is consistent with previous results.  When $\ell$ is horizontal, $a$ can be thought of as lying at infinity, which is always outside $\overline{Q}$, and so $v^{3 \pm}$ always lie on opposite sides of $P^S$.  The computation to justify this can also be performed.

\begin{proof}
Note that
\begin{align*}
A_1 v^{3 \pm}_1 + A_2 v^{3 \pm}_2 + \delta
		&= \frac{\delta \pm \frac{M}{\kappa}}{-A_1 a_1 - A_2 a_2} (A_1 a_1 + A_2 a_2) + \delta \\
		&= -\delta \mp \frac{M}{\kappa} + \delta \\
		&= \mp \frac{M}{\kappa}.
\end{align*}
Since $P^S$ is defined by the equation $A_1 x_1 + A_2 x_2 = 0$, the fact that the vertices $v^{3 \pm}$ produce opposite signs in this computation implies that they must lie on opposite sides of $P^S$.
\end{proof}

\subsubsection{Vertices and auxiliary points as limits}
While the horizontal case seems to require a modification of the analysis for the non-horizontal case, this is a somewhat artificial consequence of the choices we are making for the parameters defining $\ell$.
If $\ell = \ell_a$ is horizontal, the vertices and auxiliary points can be seen as limits of corresponding points for non-horizontal lines.  This comes about when using $a \in \mathscr{L}'$ because the cases where $\ell_a$ is horizontal are represented somewhat differently than when $\ell_a$ is non-horizontal.  If we were to work through the corresponding analysis on $\mathscr{L}$ we would find that this artificial distinction would disappear.

\section{Some special cases} \label{specialcasessection}

We explore here a few special cases that serve to fill out the picture of taxicab conic sections.  This is by no means an exhaustive exploration and we leave it to the reader to discover other interesting cases.

\subsection{Horizontal defining plane}

This case complements the horizontal line scenario which itself is a special case of sorts.  In that setting the conic sections are always hyperbolas.  Here we find that if the defining plane is horizontal, all resulting conic sections are ellipses.

\begin{thm} \label{horizontalplanethm}
Given a cone $C(\ell_a, P_{(0,0,1)}, \kappa)$ where $a = (a_1, a_2, 1)$
and slicing plane $S = \{x \in \mathbb{R}^3:  x_3 = 1\}$, the resulting conic section is an ellipse and
\begin{itemize}
\item if $\rho^1$ is active, then the vertices lying on $\rho^1$ are
\[
v^{1 \pm} = a \pm \kappa(1 , 0, 0);
\]
\item if $\rho^2$ is active, then the vertices lying on $\rho^2$ are
\[
v^{2 \pm} = a \pm \kappa (0, 1, 0);
\]
\item if $\rho^3$ is active, then the vertices lying on $\rho^3$ are
\[
v^{3 \pm} = a \pm \kappa (a_1, a_2, 0).
\]
\end{itemize}
\end{thm}

\begin{proof}
The fact that the conic section is always an ellipse follows from the fact that, in this case, the characterizing strip $Q$ is all of $S$.  The formulas for the vertices follow directly from Equations~\eqref{v1altformulaeq}, \eqref{v2altformulaeq},  and \eqref{v3altformulaeq}.
\end{proof}

From the formulas for the vertices, we can see that $\kappa$ just causes a rescaling of the vertices around $a$ and so does not affect the shape.  We can also see that if $\ell$ is steep, the resulting conic section is a taxicab circle, if $\ell$ is shallow or transitional, the resulting conic section is a parallelogram, with rhombi occurring when $a$ lies on a coordinate axis, and when $\ell$ is intermediate, the resulting conic section is a hexagon with parallel opposite sides.  The fact that opposite sides are always parallel can also be seen without knowing the formulas for the vertices by the fact that $P^S$ does not exist and so all auxiliary points lie at infinity.  See Figure~\ref{horizontalplanefig} for some specific examples of the various possibilities.

\begin{figure}
\begin{picture}(250,240)
\put(0,10)
{
\put(-50,130){
\includegraphics[scale = .4, clip = true, draft = false]{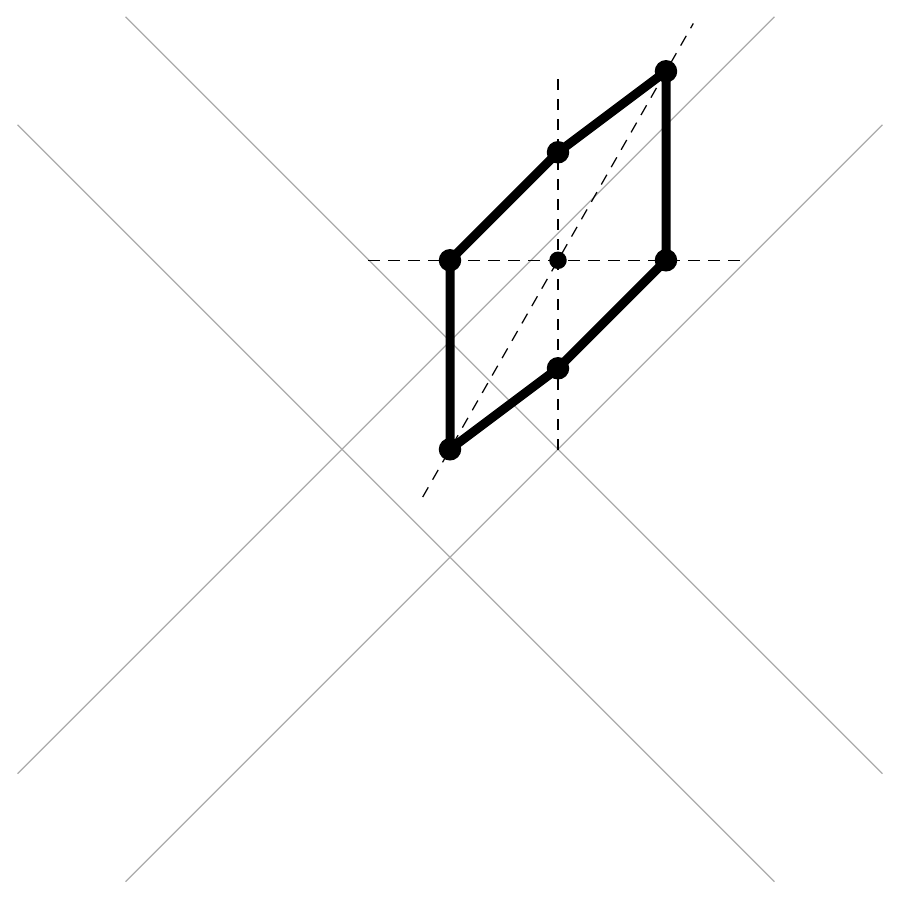}
}
\put(-30,120){$a_1 = 1,\ a_2 = \frac{7}{4}$}
\put(70,130){
\includegraphics[scale = .4, clip = true, draft = false]{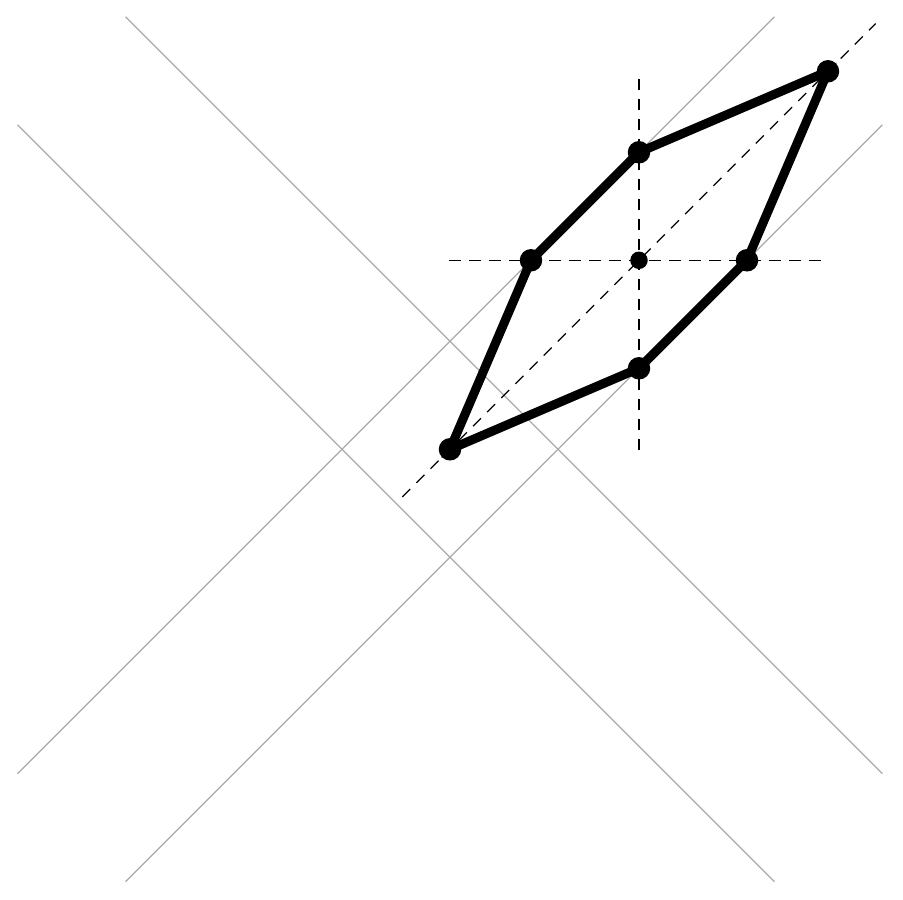}
}
\put(100,120){$a_1 = a_2 = \frac{7}{4}$}
\put(190,130){
\includegraphics[scale = .4, clip = true, draft = false]{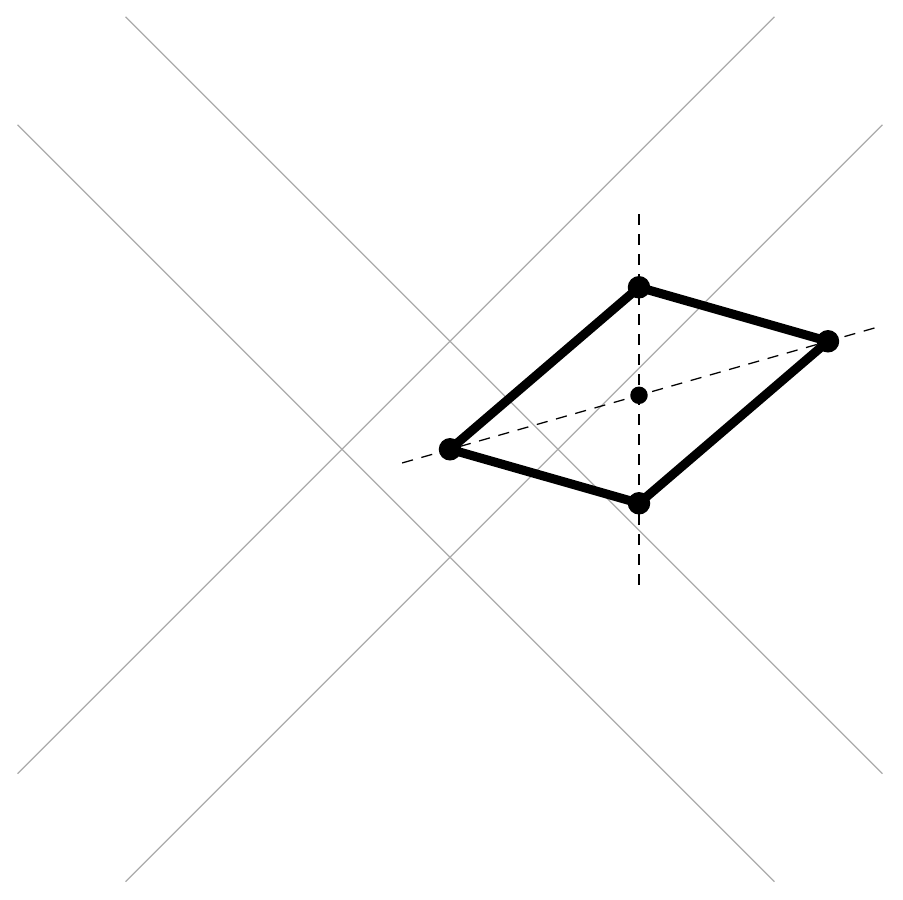}
}
\put(210,120){$a_1 = \frac{7}{4},\ a_2 = \frac{1}{2}$}
\put(10,0){
\includegraphics[scale = .4, clip = true, draft = false]{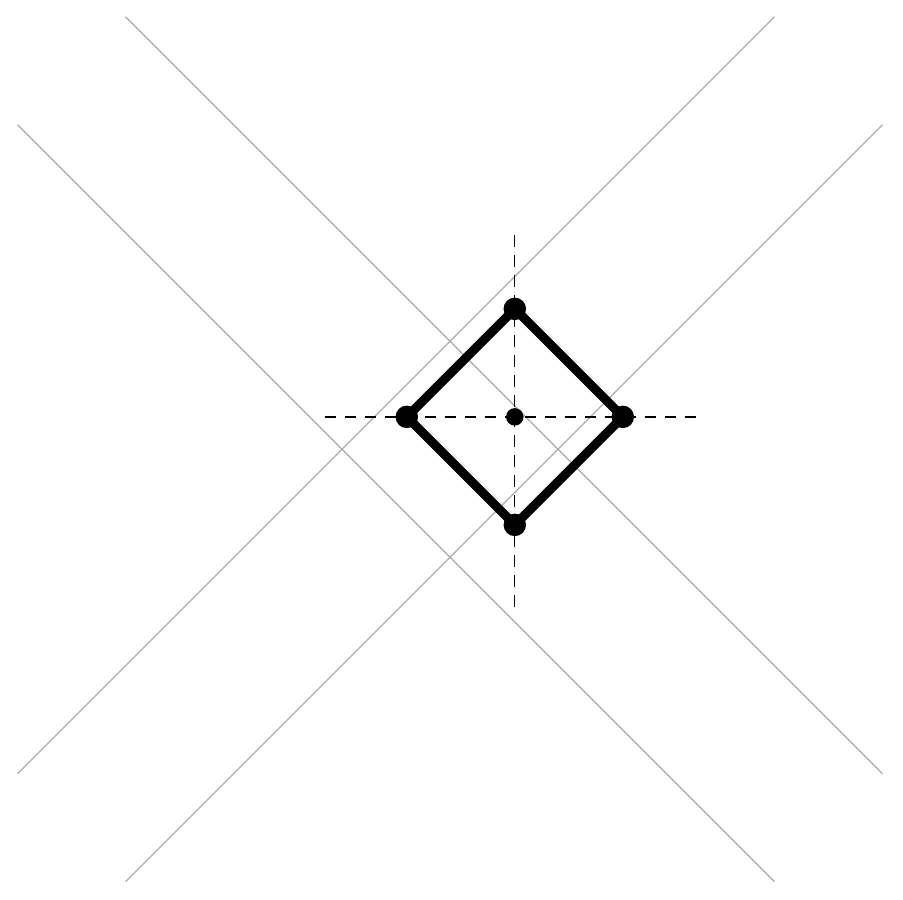}
}
\put(30,-10){$a_1 = \frac{3}{5},\ a_2 = \frac{3}{10}$}
\put(130,0){
\includegraphics[scale = .4, clip = true, draft = false]{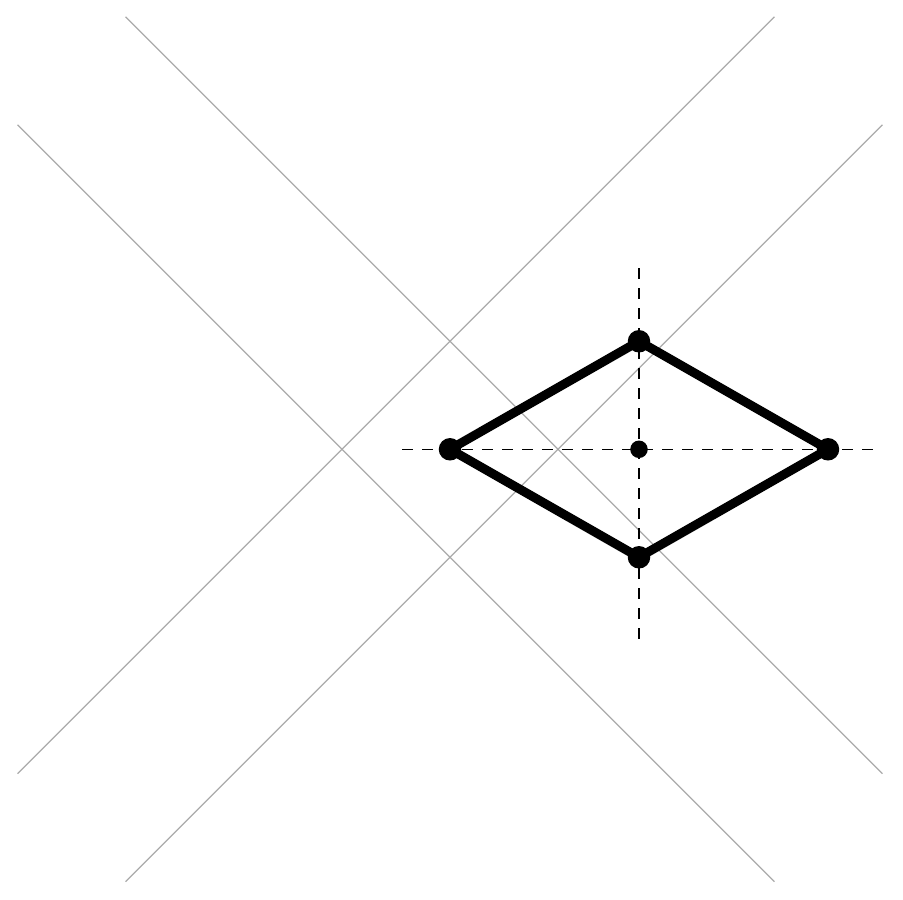}
}
\put(150,-10){$a_1 = \frac{7}{4},\ a_2 = 0$}
}
\end{picture}
\caption{Various conic sections when $P$ is horizontal.  In all cases, $\kappa = 1$.} \label{horizontalplanefig}
\end{figure}
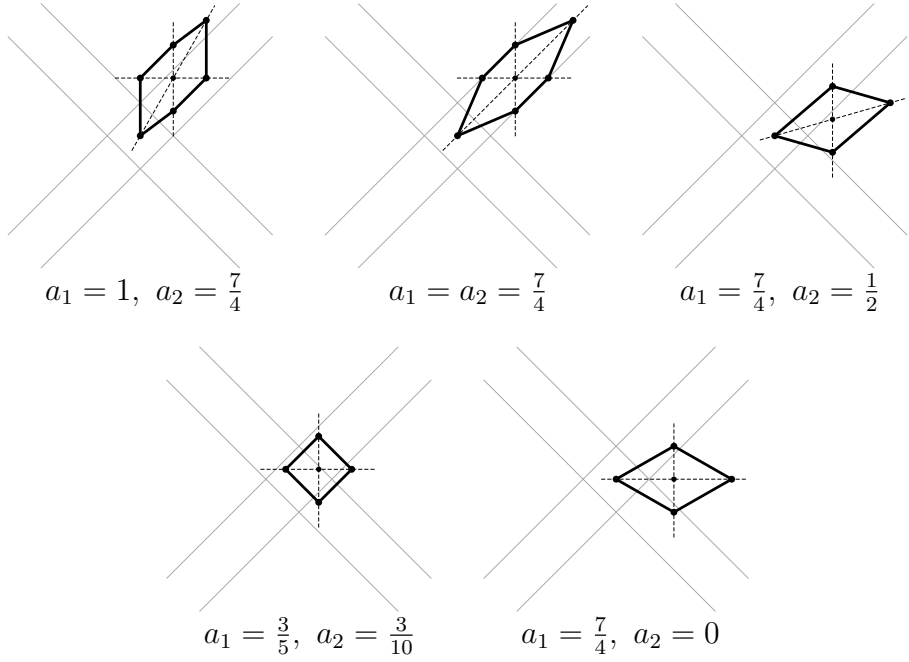

\subsection{``Perpendicular'' line and plane} \label{lineplaneperpsubsec}
As mentioned earlier, there is no definition of angle that naturally arises from the taxicab distance.  Nonetheless, it is worth considering the special case where $a$ is a multiple of $A$, and this is the condition that will be implied when we say a line is perpendicular to a plane.

Note that the definition of angle in \cite{ThompsonDray} is such that an angle is a right angle in the Euclidean setting if and only if it is a right angle in the taxicab setting.  They do not explore the notion of taxicab angles in $\mathbb{R}^3$ but it would be reasonable to expect the idea of ``perpendicular'' to be preserved in any suitably robust definition of angle for $(\mathbb{R}^3, d)$

In this situation, it is not the actual conic sections that are particularly interesting; they share the characteristics discussed already in Sections~\ref{sectioncarsec} and~\ref{hlinesetioncarsec}.  What is interesting is the set of parameters where the resulting conic sections are parabolas.

We restrict parameters to $\mathscr{P}'$ and $\mathscr{L}'$ so that $A = a$.  Also, since the particular cases involving vertical planes, and hence horizontal lines, always results in hyperbolas, moving forward we will work with $\delta = a_3 = 1$.

Let $B^E_r(x)$ be the open Euclidean ball of radius $r$ centered at $x$.  In $\mathbb{R}^2$ define
\[
U_{\kappa} = \begin{cases}
			B^E_{\frac{1}{2 \kappa}}\left(\frac{1}{2 \kappa}, 0 \right)
			\cup B^E_{\frac{1}{2 \kappa}}\left(0, \frac{1}{2 \kappa} \right) \\
			\qquad \mathrm{} \cup B^E_{\frac{1}{2 \kappa}}\left(-\frac{1}{2 \kappa}, 0 \right)
			\cup B^E_{\frac{1}{2 \kappa}}\left(0, -\frac{1}{2 \kappa} \right) \\
			\qquad \mathrm{} \cup B^E_{\frac{1}{\sqrt{\kappa}}} (0, 0) &\ \mathrm{if}\ 0 < \kappa < 1, \\
			\left(-\frac{1}{\kappa}, \frac{1}{\kappa} \right)
			\times \left(-\frac{1}{\kappa}, \frac{1}{\kappa} \right)
			\cap B^E_{\frac{1}{\sqrt{\kappa}}} (0, 0) &\ \mathrm{if}\ \kappa \geq 1.
			\end{cases}
\]
Note that when $0 < \kappa < \frac{1}{2}$, the disk at the origin is a subset of the others and so does not add anything to the union or contribute to the boundary.  In a complementary fashion, when $\kappa > 2$, the disk at the origin contains the square and so does not participate in the intersection or the boundary.

\begin{thm}
Let $A \in \mathscr{P}'$ with $\delta = 1$.  Then, given a cone $C = C(\ell_A, P_A, \kappa)$ the resulting conic section $C \cap S$ is an ellipse when $A$ lies in $U_{\kappa} \times \{(0, 0, 1)\}$, a parabola when $A$ lies in $\partial U_{\kappa} \times \{(0, 0, 1)\}$, and a hyperbola when $A$ lies outside $\overline{U_{\kappa}} \times \{(0, 0, 1)\}$.
\end{thm}

Note that the Cartesian products with $\{(0, 0, 1)\}$ are just to ensure that the resulting sets lie in $S$.  See Figure~\ref{Aisaparabolaparametersfig} for examples of $\partial U_{\kappa}$ for various values of $\kappa$.

\begin{proof}
As discussed before, parabolas occur when at least one vertex lies at infinity, and all the remaining vertices lie on the same side of $P^S$.  As such, to find parabolas, noting that $A = a = (0, 0, 1)$ always results in an ellipse, extend outward until the first non-finite vertex is encountered.  To simplify the work, the computations can be performed in the wedge $0 \leq A_2 \leq A_1$, and then extended to the rest of the plane using taxicab isometries.  Some care must be taken to account for where various vertices are active.  The details are left to the reader.
\end{proof}

At first glance, it is perhaps surprising that Euclidean circles appear in these sets.  Algebraically, they result from the fact that, in the expressions for the vertices $v^{3 \pm}$, $A_i$ are multiplied by $a_i$ in the denominator.  In our setting $A = a$ so the parameters where one of $v^{3 \pm}$ is infinite solve an equation for a circle.  From a more geometric perspective, the condition that $A = a$ means that the defining line and defining plane are perpendicular in the Euclidean sense.  In other words, we are imposing a condition that is fundamentally Euclidean in nature.  Hence, perhaps it is not surprising to see some Euclidean structure sneaking in.

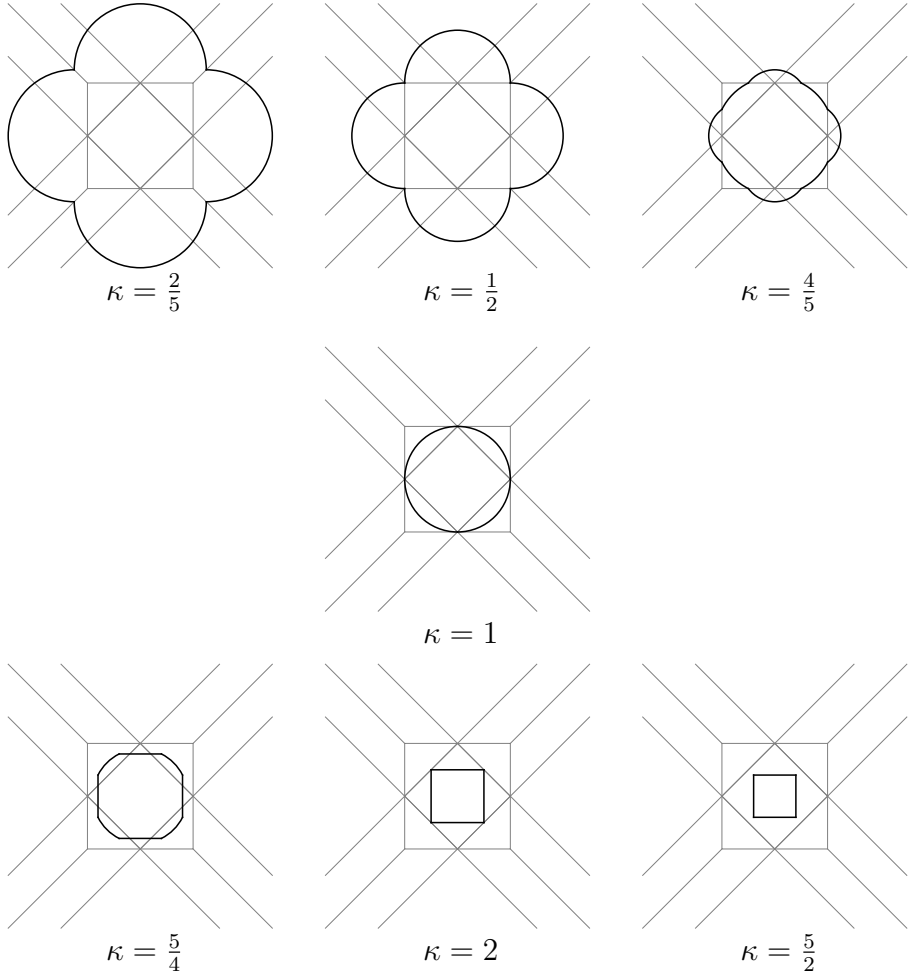
\begin{figure}
\begin{picture}(250,360)
\put(-50,260){
\includegraphics[scale = .4, clip = true, draft = false]{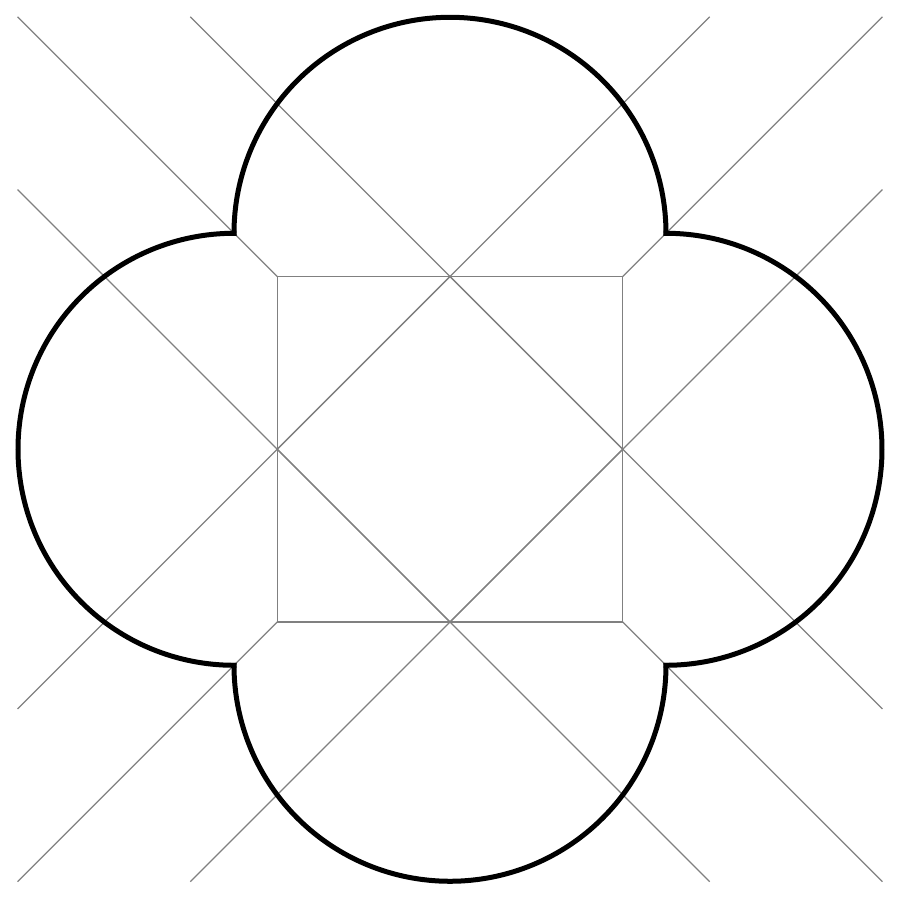}
}
\put(-7,250){$\kappa = \frac{2}{5}$}
\put(70,260){
\includegraphics[scale = .4, clip = true, draft = false]{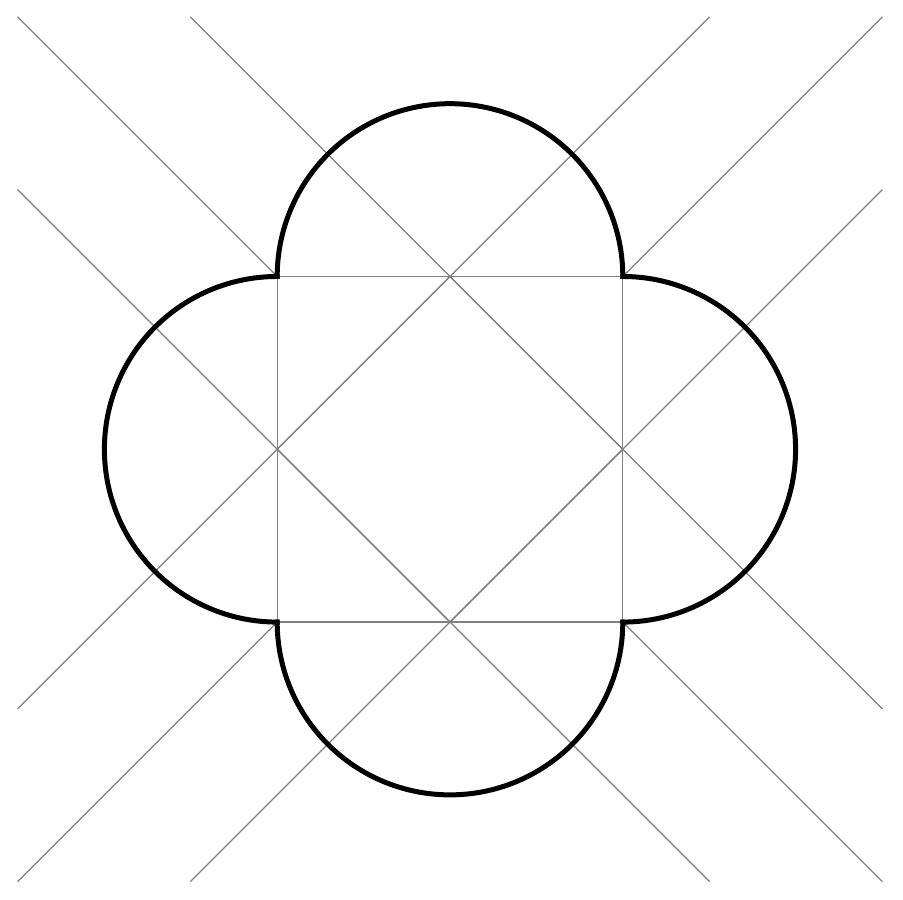}
}
\put(113,250){$\kappa = \frac{1}{2}$}
\put(190,260){
\includegraphics[scale = .4, clip = true, draft = false]{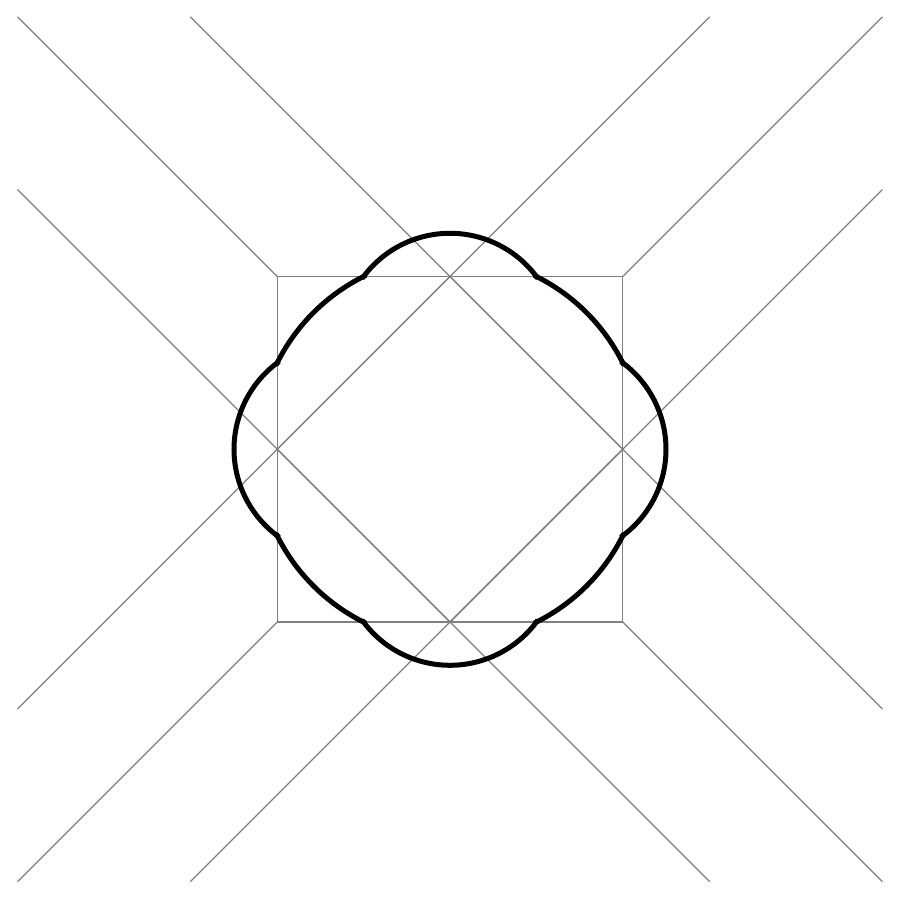}
}
\put(233,250){$\kappa = \frac{4}{5}$}
\put(70,130){
\includegraphics[scale = .4, clip = true, draft = false]{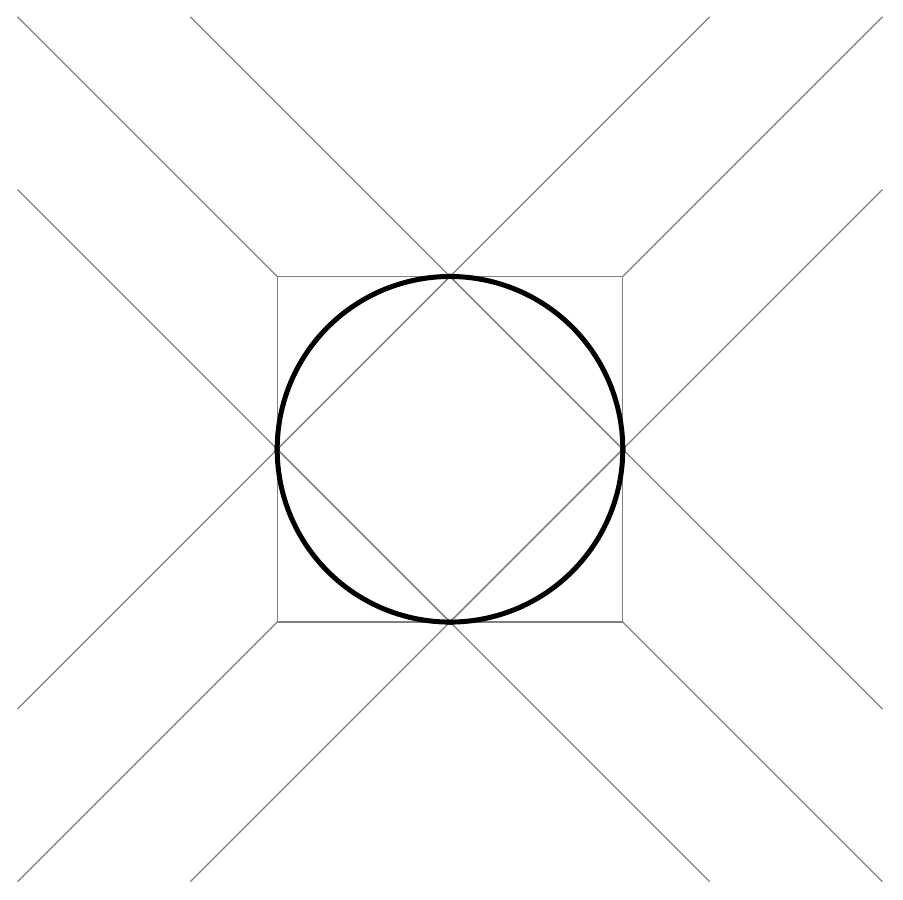}
}
\put(113,120){$\kappa = 1$}
\put(-50,10){
\includegraphics[scale = .4, clip = true, draft = false]{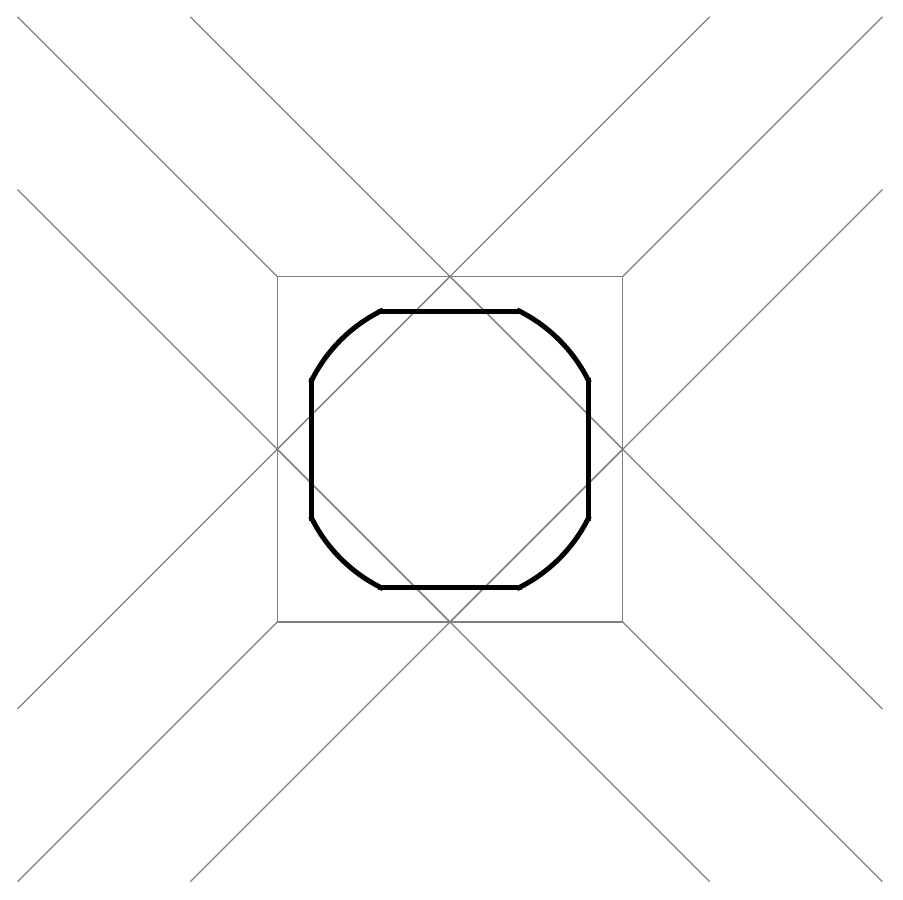}
}
\put(-7,0){$\kappa = \frac{5}{4}$}
\put(70,10){
\includegraphics[scale = .4, clip = true, draft = false]{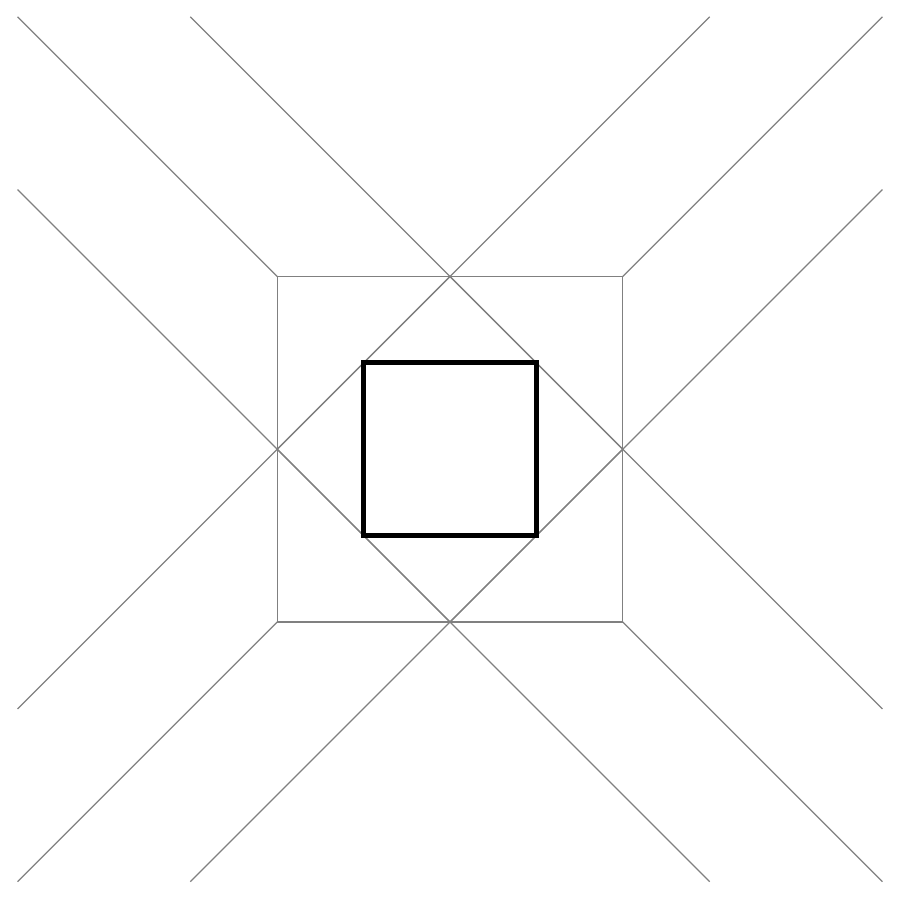}
}
\put(113,0){$\kappa = 2$}
\put(190,10){
\includegraphics[scale = .4, clip = true, draft = false]{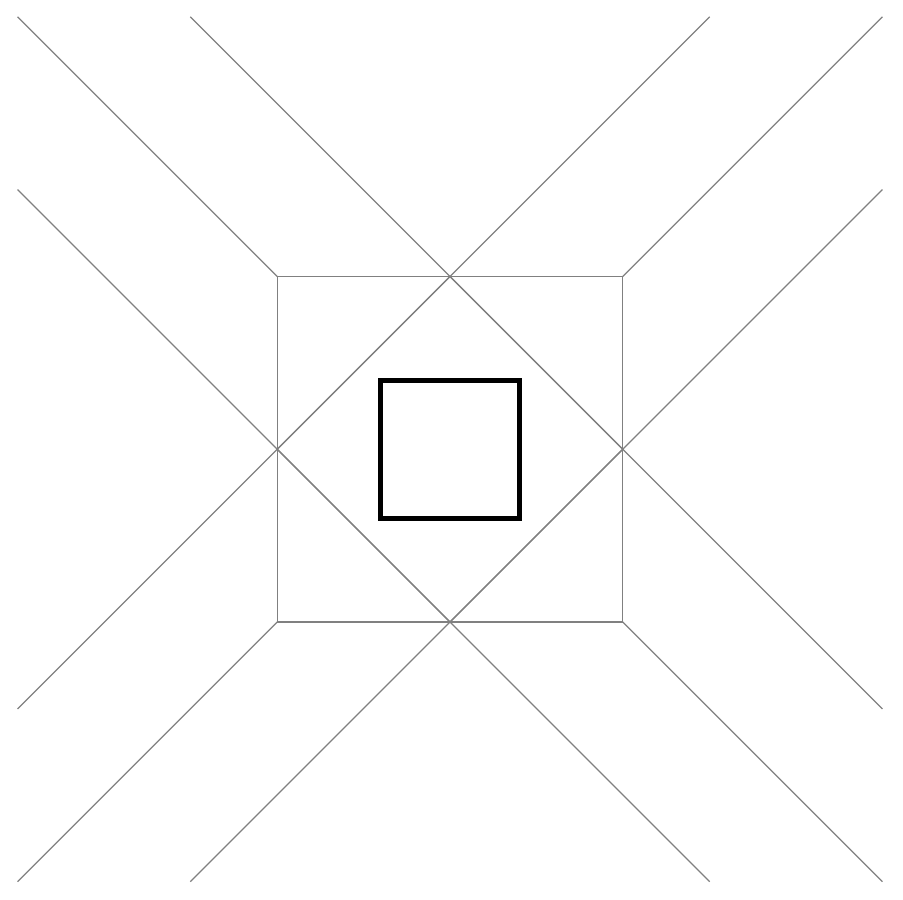}
}
\put(233,0){$\kappa = \frac{5}{2}$}
\end{picture}
\caption{When $A = a$, the parameters that lie in $\partial U_{\kappa}$ and produce parabolas are shown.  Identifying lines for $\mathscr{P}'$ and $\mathscr{L}'$ are included for reference.} \label{Aisaparabolaparametersfig}
\end{figure}

\subsection{Redundancy of conic sections}

While there is undoubtedly a wide variety conic sections that can arise by considering the many different parameters available, at least in some cases, it turns out that there is redundancy in the resulting conic sections.  The redundancies that result from taxicab isometries are certainly part of this, but they are not surprising.  There are additional redundancies that exist through inspired choices of the parameters defining the cones that prove to be more interesting.

\subsubsection{Steep defining lines}

\begin{thm} \label{steeplinethm}
Let $a = (a_1, a_2, 1),\ b = (b_1, b_2, 1) \in \mathscr{L}'$ represent steep lines.  Let $C_a = C(\ell_a, P_A, \kappa)$ and let $C_{b} = C(\ell_{b}, P_A, \kappa)$.  Then, as long as neither cone is degenerate, the resulting conic sections $C_a \cap S$ and $C_{b} \cap S$ are similar.
\end{thm}

Note that this theorem also applies for ``transitionally steep'' lines where $|a_1| + |a_2| = 1$, but unsurprisingly not other transitional lines.

\begin{proof}
For the vertices corresponding to $C_a$, Equations~\eqref{v1altformulaeq} and \eqref{v2altformulaeq} can be written
\begin{align*}
v^{1 \pm} &= a + \left(r_a^{1 \pm}, 0, 0 \right), \\
v^{2 \pm} &= a + \left(0, r_a^{2 \pm}, 0 \right) \\
\end{align*}
where for $i \in \{1, 2\}$
\[
r_a^{i \pm} = \frac{A_1 a_1 + A_2 a_2 + \delta}{\pm \frac{M}{\kappa} - A_i}.
\]
The formulas for $v^{3 \pm}$ are not necessary since these vertices are inactive when $\ell$ is steep.  The equations for the vertices corresponding to $C_b$ are similar.

Since $r_a^{i \pm}$ and $r_b^{i \pm}$ represent the deviations of the vertices from $a$ and $b$ respectively, it will be sufficient to show that the ratios of corresponding deviations are independent of $i$ or the sign choice.  Let $s \in \{+, -\}$.  Then
\begin{align*}
\frac{r_a^{i s}}{r_{b}^{i s}}
	&= \frac{\frac{A_1 a_1 + A_2 a_2 + \delta}{s \frac{M}{\kappa} - A_i}}
		{\frac{A_1 b_1 + A_2 b_2 + \delta}{s \frac{M}{\kappa} - A_i}} \\
	&= \frac{A_1 a_1 + A_2 a_2 + \delta}{A_1 b_1 + A_2 b_2 + \delta}
\end{align*}
where the cancelation can occur since $P_A$ is the same for the two cones.  This shows that the various ratios, independent of $i$ and the sign choice, are the same.
\end{proof}

\subsubsection{Defining planes with parallel intersections with $S$}

\begin{thm} \label{parallelplanethm}
Let $P^S_A$ and $P^S_B$ be parallel.  Let  $C_A = C(\ell_a, P_A, \kappa_A)$.  Then, there exists a value $\kappa_B \in (0, \infty)$ such that for  $C_B = C(\ell_a, P_B, \kappa_B)$, the resulting conic sections $C_A \cap S$ and $C_B \cap S$ are similar as long as neither cone is degenerate.
\end{thm}

The proof of this result is similar to the proof of Theorem~\ref{steeplinethm}, but the analysis is more delicate.

\begin{proof}
Implicit in the statement of the theorem is the fact that neither $P_A$ nor $P_B$ are horizontal.  As such, at least one of $A_1$ and $A_2$ is nonzero, and similarly at least one of $B_1$ and $B_2$ is nonzero.  Furthermore, the statement that $P^S_A \| P^S_B$ is equivalent to the statement that
\begin{equation*}
A_1 B_2 = A_2 B_1.
\end{equation*}
Without loss of generality, suppose $A_1 \neq 0$.  Then the above conditions imply that $B_1$ is also nonzero and
\begin{equation} \label{parallelconditioneqn}
\frac{A_1}{B_1} = \frac{A_2}{B_2}
\end{equation}
as long as both are defined.

For $C_A$, Equations~\eqref{v1altformulaeq}, \eqref{v2altformulaeq}, and \eqref{v3altformulaeq} can be written
\begin{align*}
v^{1 \pm} &= a + \left(r_A^{1 \pm}, 0, 0 \right), \\
v^{2 \pm} &= a + \left(0, r_A^{2 \pm}, 0 \right), \\
v^{3 \pm} &= a + r_A^{3 \pm} (a_1, a_2, 0)
\end{align*}
where for $i \in \{1, 2\}$
\[
r_A^{i \pm} = \frac{A_1 a_1 + A_2 a_2 + \delta_A}{\pm \frac{M_A}{\kappa_A} - A_i}
\]
and
\[
r_A^{3 \pm} = \frac{A_1 a_1 + A_2 a_2 + \delta_A}{\pm \frac{M_A}{\kappa_A} - (A_1 a_1 + A_2 a_2)}.
\]
The equations for $C_B$ are similar.

Since $r_A^{i \pm}$ and $r_B^{i \pm}$ represent the deviations of the vertices from $a$, it will be sufficient to show that the ratios of corresponding deviations are independent of $i$ or the sign choice.  Unlike the proof for Theorem~\ref{steeplinethm}, the question of which vertices correspond to one another is more subtle.  The indices of corresponding vertices will match, but the signs may not.

Consider two cases.  First, if both $P_A$ and $P_B$ are steep or transitional then
\begin{align*}
M_B &= \max\{|B_1|, |B_2|, \delta_A\} \\
	&= \max\{|B_1|, |B_2|\} \\
	&= \left|\frac{B_1}{A_1}\right| \max\{|A_1|, |A_2|\} \\
	&= \left|\frac{B_1}{A_1}\right| \max\{|A_1|, |A_2|, \delta_A\} \\
	&= \left|\frac{B_1}{A_1}\right| M_A
\end{align*}
where the third equality follows from Equation~\eqref{parallelconditioneqn}.  Let $f = \frac{\left|\frac{B_1}{A_1}\right|}{\frac{B_1}{A_1}}$ and, noting that $f \in \{1, -1\}$, let $s_A \in \{+, -\}$ and let
\[
s_B =
	\begin{cases}
	s_A &\ \mathrm{if}\ f = 1, \\
	\mathrm{opposite\ of\ }s_A &\ \mathrm{if}\ f = -1.
	\end{cases}
\]

For this case, let $\kappa_B = \kappa_A$.  Then for $i \in \{1, 2\}$
\begin{align*}
\frac{r_A^{i s_{A}}}{r_B^{i s_{B}}}
	&= \frac{\frac{A_1 a_1 + A_2 a_2 + \delta_A}{s_{A} \frac{M_A}{\kappa_A} - A_i}}
		{\frac{B_1 a_1 + B_2 a_2+ \delta_B}{s_{B} \frac{M_B}{\kappa_B} - B_i}} \\
	&= \left( \frac{A_1 a_1 + A_2 a_2 + \delta_A}{B_1 a_1 + B_2 a_2+ \delta_B} \right)
		\left( \frac{s_{B} M_B - B_i \kappa_B}{s_{A} M_A - A_i \kappa_A} \right)
		\frac{\kappa_A}{\kappa_B} \\
	&= \left( \frac{A_1 a_1 + A_2 a_2 + \delta_A}{B_1 a_1 + B_2 a_2+ \delta_B} \right)
		\left( \frac{s_{B} \left|\frac{B_1}{A_1}\right| M_A - \frac{B_1}{A_1} A_i \kappa_A}
			{s_{A} M_A - A_i \kappa_A} \right) \\
	&= \left( \frac{A_1 a_1 + A_2 a_2 + \delta_A}{B_1 a_1 + B_2 a_2+ \delta_B} \right)
		\left( \frac{s_{B} (f M_A)- A_i \kappa_A}
			{s_{A} M_A - A_i \kappa_A} \right) \frac{B_1}{A_1} \\
	&= \left( \frac{A_1 a_1 + A_2 a_2 + \delta_A}{B_1 a_1 + B_2 a_2+ \delta_B} \right)
		\left( \frac{s_{A} M_A - A_i \kappa_A}
			{s_{A} M_A - A_i \kappa_A} \right) \frac{B_1}{A_1} \\
	&= \left( \frac{A_1 a_1 + A_2 a_2 + \delta_A}{B_1 a_1 + B_2 a_2+ \delta_B} \right)
		\frac{B_1}{A_1}	
\end{align*}
which is independent of $i$ and the signs $s_{A}$ and $s_{B}$.

The calculation for $\frac{r_A^{3 s_A}}{r_B^{3 s_B}}$ is similar, resulting in
\[
\frac{r_A^{3 s_{A}}}{r_B^{3 s_{B}}}
	= \left( \frac{A_1 a_1 + A_2 a_2 + \delta_A}{B_1 a_1 + B_2 a_2+ \delta_B} \right)
		\frac{B_1}{A_1}	.
\]

Since all six ratios result in the same quantity, the resulting conic sections must be similar.

For the second case, if both $P_A$ and $P_B$ are shallow or transitional then $M_A = M_B = 1$.  For this case, let
$\kappa_B = \left| \frac{A_1}{B_1} \right| \kappa_A$.  Then for $i \in \{1, 2\}$
\begin{align*}
\frac{r_A^{i s_{A}}}{r_B^{i s_{B}}}
	&= \frac{\frac{A_1 a_1 + A_2 a_2 + \delta_A}{s_{A} \frac{M_A}{\kappa_A} - A_i}}
		{\frac{B_1 a_1 + B_2 a_2+ \delta_B}{s_{B} \frac{M_B}{\kappa_B} - B_i}} \\
	&= \left( \frac{A_1 a_1 + A_2 a_2 + \delta_A}{B_1 a_1 + B_2 a_2+ \delta_B} \right)
		\left( \frac{s_{B} \frac{1}{\kappa_B} - B_i}{s_{A} \frac{1}{\kappa_A} - A_i} \right) \\
	&= \left( \frac{A_1 a_1 + A_2 a_2 + \delta_A}{B_1 a_1 + B_2 a_2+ \delta_B} \right)
		\left( \frac{s_{B}  \left| \frac{B_1}{A_1} \right| \frac{1}{\kappa_A} - \frac{B_1}{A_1} A_i}
			{s_{A} \frac{1}{\kappa_A} - A_i} \right) \\
	&= \left( \frac{A_1 a_1 + A_2 a_2 + \delta_A}{B_1 a_1 + B_2 a_2+ \delta_B} \right)
		\left( \frac{s_{B} f \frac{1}{\kappa_A} - A_i}{s_{A} \frac{1}{\kappa_A} - A_i} \right)
			\frac{B_1}{A_1} \\
	&= \left( \frac{A_1 a_1 + A_2 a_2 + \delta_A}{B_1 a_1 + B_2 a_2+ \delta_B} \right)
		\left( \frac{s_{A} \frac{1}{\kappa_A} - A_i}{s_{A} \frac{1}{\kappa_A} - A_i} \right)
			\frac{B_1}{A_1} \\
	&= \left( \frac{A_1 a_1 + A_2 a_2 + \delta_A}{B_1 a_1 + B_2 a_2+ \delta_B} \right)
			\frac{B_1}{A_1}
\end{align*}		
which is independent of $i$ and the signs $s_{A}$ and $s_{B}$.

The calculation for $\frac{r_A^{3 s_{A}}}{r_B^{3 s_{B}}}$ is similar, resulting in:
\[
\frac{r_A^{3 s_{A}}}{r_B^{3 s_{B}}}
	= \left( \frac{A_1 a_1 + A_2 a_2 + \delta_A}{B_1 a_1 + B_2 a_2+ \delta_B} \right)
				\frac{B_1}{A_1}.
\]

Again, since all six ratios result in the same quantity, the resulting conic sections must be similar.

To complete the proof, note that similarity is transitive and both cases above include transitional planes, at least one of which is not degenerate.
\end{proof}

Note that the quantity in the proof to which all the ratios are equal can be positive or negative.  This corresponds to the fact that one of the similar conic sections may be rotated by $\pi$ relative to the other.

Based on this result, we could significantly reduce the set of parameters we use to identify the defining planes.  One choice could be to use only vertical planes, except for the one degenerate case necessitating an equivalent plane.  Another choice could be to use only transitional planes.  In this case, the degenerate case would be covered by the other equivalent transitional plane.  Regardless of the choice made, we would also want to include the one horizontal plane.  Finally, it is worth noting that avoiding such a restriction and allowing for the wider variety of planes is still useful, especially in light of Section~\ref{lineplaneperpsubsec}.

\subsection{Finding ``traditional'' taxicab conic sections}

Given the wide variety of conic sections found here as slices of cones, it is illuminating to consider how they relate to the more traditional taxicab conic sections found using the two-foci or focus-directrix definitions as discussed for example in \cite{KAGO}.

We find that almost none of the conic sections defined using the two-focus definition, as indicated in the first and last rows of Figure~\ref{distanceconicsfig}, appear among conic sections defined as slices of cones.  The only exception is that if  $P$ is horizontal and $\ell$ is steep, the resulting section is a circle.  Two-foci ellipses other than circles, and two-foci hyperbolas do not appear in the our slice-formulation.  This can most easily be seen by noting that, for the most part, the vertices of two-foci conic sections do not all lie on valid reference lines, or if they do, the resulting segments cannot stem from lines passing through auxiliary points on $P^S$.

On the other hand, all of the conic sections defined using the focus-directrix definition appear as slices of cones, arising specifically when both $P$ and $\ell$ are steep.  In this case, for $x \in S$,
\[
d(x, P) = d(x, P^S) = d_S(x, P^S)
\]
and
\[
d(x, \ell) = d_{1, 2}(x, \ell) = d(x, \ell \cap S) = d_S(x, \ell \cap S)
\]
where $d_S$ is the 2-dimensional taxicab distance on $S$.  As such, the formula for the set of points in $S$ satisfying $d(x, \ell) = \kappa\, d(x, P)$ reduces to the focus-directrix definition.

\subsubsection{Near misses}
When $\ell$ is not steep, some of the resulting conic sections can be qualitatively similar to conic sections arising from the focus-directrix definition, but they tend to be near misses.  For example, if $A = (1, 4, 1)$, $a = (2, 0, 1)$, and $\kappa = 1$, the resulting conic section $C \cap S$ is the parabola shown in Figure~\ref{parabolafig}(a).  The parabola in $S$ resulting from the focus $a$ and directrix $P^S$ using the focus-directix definition is shown in Figure~\ref{parabolafig}(b).

\begin{figure}
\begin{picture}(310,140)
\put(0,10){
\includegraphics[scale = .5, clip = true, draft = false]{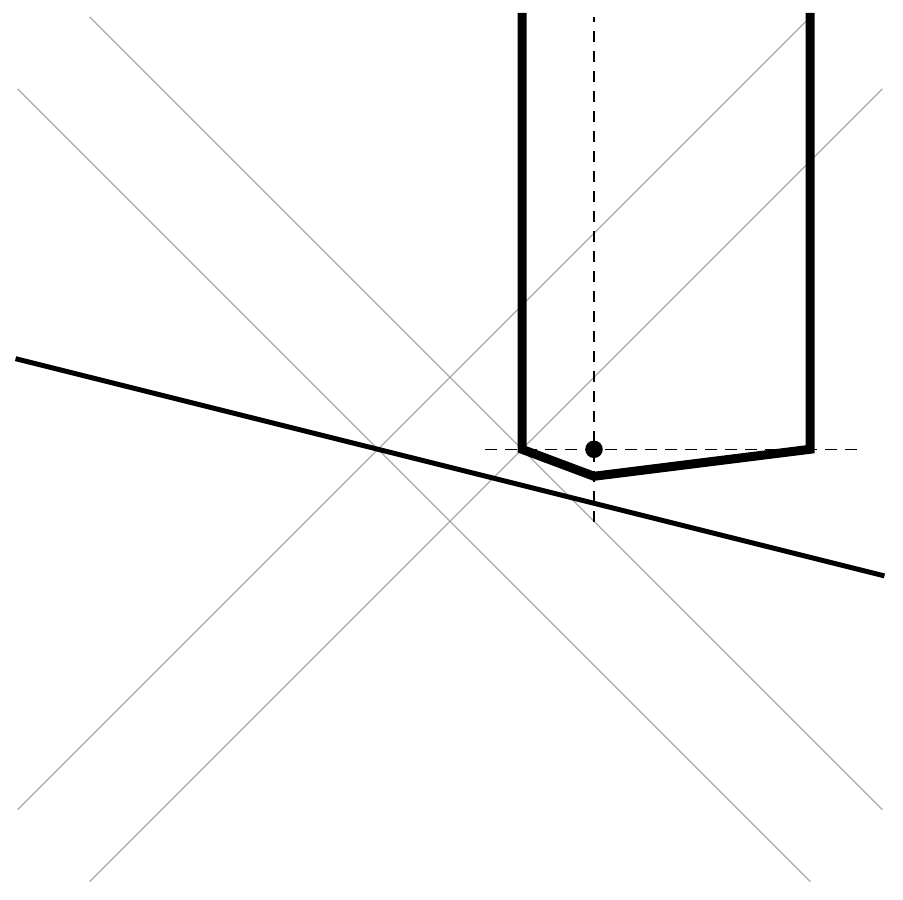}
}
\put(175,10){
\includegraphics[scale = .5, clip = true, draft = false]{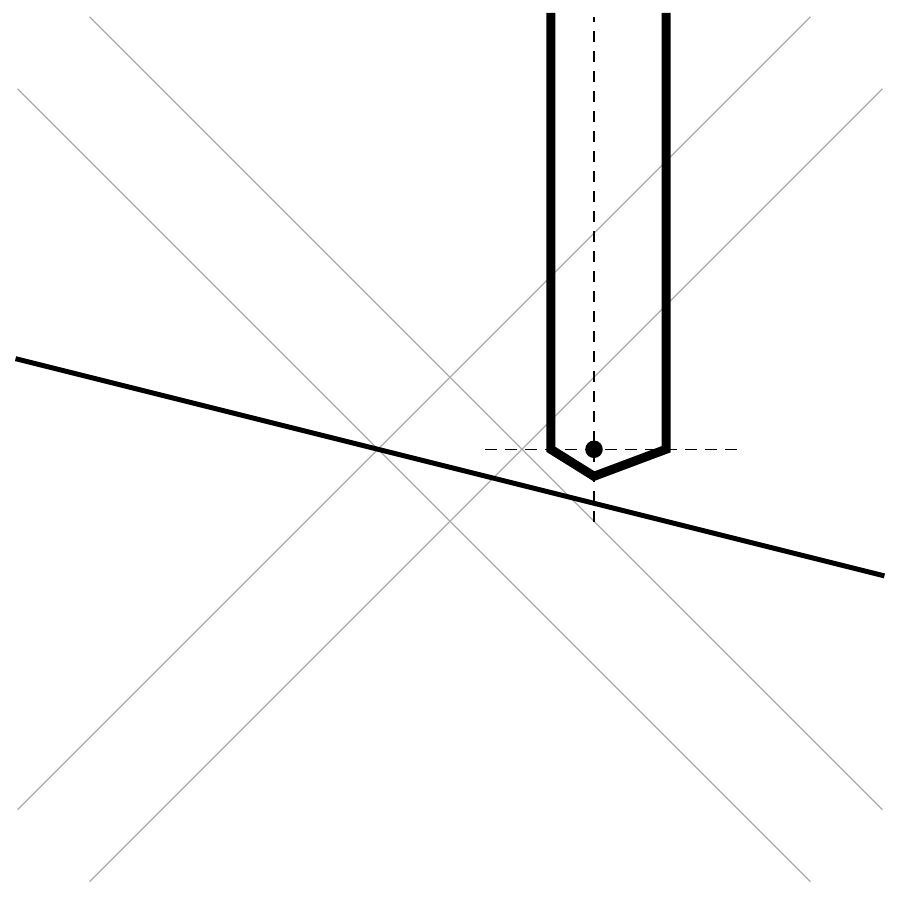}
}
\put(4,75){\scriptsize$P^S$}
\put(92,77){\scriptsize$a$}
\put(179,75){\scriptsize$P^S$}
\put(267,77){\scriptsize$a$}
\put(65,0){(a)}
\put(240,0){(b)}
\end{picture}
\caption{Parabolas using the same data, but different definitions.  In (a), the parabola is the result of slicing a cone.  In (b), the parabola arises from the focus-directrix definition.} \label{parabolafig}
\end{figure}
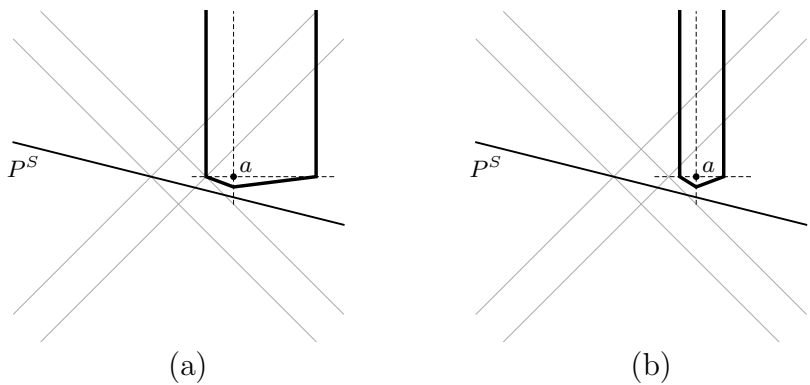

In fact, the cone-slice parabola in Figure~\ref{parabolafig}(a) cannot be a focus-directrix parabola for any focus or directrix.  We leave it as an exercise for the reader to show that for any parabola defined using the focus-directrix definition, if the parallel edges are in the $x_2$ direction, the slopes of the other two edges differ by 1.  In the example here, and in general if $\ell$ is not steep, the corresponding parabolas have slopes of corresponding edges that do not differ by 1.

\section{Final thoughts and next steps} \label{nextstepssection}

As we have seen, defining conic sections in terms of slicing cones as we have done here has proven surprisingly fruitful.  Not only have we broadened the types of objects that arise and developed an understanding of their characteristics, but we have also developed a deeper understanding of taxicab space itself.  In this process, a number of opportunities for further exploration have presented themselves.  We outline a few such opportunities here, but there are surely others that the authors have not considered.

In this paper, we show what arises when slicing a cone defined by certain parameters.  The inverse problem would be a natural follow-up to this:  given a collection of segments and rays that qualitatively resembles a conic section, when is it actually a conic section and what are the parameters that produce it?  There are a number of necessary conditions, such as the requirement that the vertices must lie on appropriate reference lines and then there are many cases to consider owing to the wide variety of defining lines and planes that are available.  The resulting theorem is not likely to be very pretty, but a complete result would further characterize conic sections beyond what has been established here.

Since conic sections are polygonal in nature, interesting connections to finite geometries may exist.  For example when $\ell$ is not intermediate, the four vertices, two reference lines, four edge lines, two auxiliary points, and $a$ form most of a Fano plane.  The last line would be $P^S$ if we defined $a$ to lie on $P^S$ as well.  In light of the symmetries inherent in the Fano plane, could it be that there are associated families of conic sections resulting from shifting the roles of the various points and lines?

One theme running through much of the analysis in this paper relates to variations on taxicab distance, or equivalently, taxicab distance in alternate coordinate systems.  This issue first appears when considering the induced metric on oblique planes, but also manifests itself when considering the partial distances used to compute the distance between a point and a line.  A more complete exploration of these alternative taxicab distances is warranted.

One application would be another way to characterize our conic sections by proving that they are the boundary of a union of simpler shapes which should themselves have boundaries that are conic sections for alternative taxicab distances.  This way of constructing conic sections would be quite similar to the construction method for Apollonian sets found in \cite{BCFHMNSTV}.

Also, in Theorem~\ref{auxpointthm}, the geometric significance of $\beta = 0$ is established in the proof:  it is the equation for $P^S$.  On the other hand, the geometric significance of $\alpha = 0$ is not discussed.  For $d_{1,2}$ they are the equations for the guidelines through $a$, as defined in \cite{BCFHMNSTV}.  In the other two cases, they could be the guidelines for the 2-D taxicab metric corresponding to the given partial distance.


\bibliographystyle{amsalpha}
\bibliography{taxicab-conics}

\end{document}